\documentclass[reqno,10pt]{amsart}
\usepackage{graphicx,amsfonts,amssymb,amsmath,amsthm,url,amscd,comment}
\usepackage{color}
\usepackage{enumerate}

\usepackage{chngcntr}
\usepackage[usenames,dvipsnames]{xcolor}
\usepackage[normalem]{ulem}
\usepackage{pdfsync}
\usepackage{graphicx, amsmath, amssymb, amsfonts, amsthm, stmaryrd, tikz, amscd}
\usepackage[all]{xy}
\usetikzlibrary{matrix,arrows,decorations.pathmorphing}
\usepackage{graphicx, amsmath, amssymb, amsfonts, amsthm, stmaryrd, tikz, amscd}
\usepackage[all]{xy}
\usetikzlibrary{matrix,arrows,decorations.pathmorphing}

 \usepackage[hyperfootnotes=false, colorlinks, citecolor=   RoyalBlue, urlcolor=blue, linkcolor=blue         ]{hyperref}

\theoremstyle{plain} 
\newtheorem{theorem}             {Theorem} 

\newtheorem{corollary}[theorem] {Corollary}

\newtheorem{conjecture}[theorem] {Conjecture}

\theoremstyle{definition}

\theoremstyle{plain} 
\newtheorem{proposition} [theorem] {Proposition}

\theoremstyle{remark}

\newtheorem{definition} [theorem]  {Definition}
\newtheorem*{definition*}  {Definition}

\newtheorem*{example*}    {Example}

\newtheorem{remark}  [theorem]           {Remark}
\newtheorem*{remark*}            {Remark}

\newtheoremstyle{itplain} 
    {6pt}                    
    {5pt\topsep}                    
    {\itshape}                   
    {}                           
    {\itshape}                   
    {.}                          
    {5pt plus 1pt minus 1pt}                       
    {}  

\theoremstyle{itplain} 
\newtheorem{lemma}[theorem]{Lemma}

\newtheorem*{lemma*}{Lemma}

\newtheorem*{corollary*} {Corollary} 

\theoremstyle{remark} 

\newtheorem*{lemmatest*}{Lemma}

\usepackage{etoolbox}
\patchcmd{\section}{\scshape}{\bfseries}{}{}
\makeatletter
\renewcommand{\@secnumfont}{\bfseries}
\makeatother

\renewcommand{\Re}{\mathrm{Re}}
\renewcommand{\Im}{\mathrm{Im}}

\renewcommand{\geq}{\geqslant}
\renewcommand{\leq}{\leqslant}

\numberwithin{equation}{section}
\numberwithin{theorem}{section}

\DeclareMathOperator{\SL}{SL}

\DeclareMathOperator{\GL}{GL}

\DeclareMathOperator{\htt}{ht}

\DeclareMathOperator{\ad}{ad}

\def\eps{\varepsilon}

\def\PGL{\operatorname{PGL}}

\DeclareMathOperator{\U}{U}
\DeclareMathOperator{\norm}{norm}

\DeclareMathOperator{\gen}{gen}

\DeclareMathOperator{\PSL}{PSL}

\DeclareMathOperator{\Lie}{Lie}

\def\O{\operatorname{O}}

\DeclareMathOperator{\Eis}{Eis}
\DeclareMathOperator{\reg}{reg}

\DeclareMathOperator{\Ind}{Ind}

\DeclareMathOperator{\res}{res}

\DeclareMathOperator{\vol}{vol}

\author{Paul D. Nelson}
\address{ETH Z{\"u}rich, Department of Mathematics, R{\"a}mistrasse 101, CH-8092, Z{\"u}rich, Switzerland}
\email{paul.nelson@math.ethz.ch}
\subjclass[2010]{Primary 11F66; Secondary 11F70, 11F03}

\date{\today}
\title{Eisenstein series and the cubic moment for
  $\PGL_2$}
\hypersetup{
  pdfkeywords={},
  pdfsubject={},
  pdfcreator={Emacs 24.5.1 (Org mode 8.2.10)}}
\begin{document}

\begin{abstract}
  Following a strategy suggested by Michel--Venkatesh, we study
  the cubic moment of automorphic $L$-functions on $\PGL_2$
  using regularized diagonal periods of products of Eisenstein
  series.  Our main innovation is to produce vectors whose
  integral transforms achieve arbitrarily weighted moments.
  Applications include general Motohashi-type identities and
  Weyl-type subconvex bounds for some families of $L$-functions,
  extending some results of Conrey--Iwaniec and Petrow--Young to
  the number field setting.  We deduce
  improved estimates for representation numbers of ternary
  quadratic forms over number fields and
  for the
  prime geodesic theorem on arithmetic hyperbolic $3$-folds.
\end{abstract}
\maketitle

\setcounter{tocdepth}{1} \tableofcontents

\section{Introduction}
\label{sec-2}
\subsection{Overview and motivation}\label{sec:intro-overview}
Michel--Venkatesh
(see
\cite[\S4.5.3]{michel-2009}, \cite{MichelVenkateshICM})
suggested a strategy
for establishing spectral
identities between moments of $L$-functions, generalizing those
introduced by Motohashi \cite{MR1226527}.  They emphasized the
further problem \cite[\S4.5.4]{michel-2009} of implementing that
strategy in a sufficiently flexible form, and suggested that
doing so might lead to strong subconvex bounds.

We address that
problem and apply their strategy to establish a general formula
for the cubic moment of central values of automorphic
$L$-functions on $\PGL_2$.
As a first application of that formula,
we generalize a theorem of Conrey--Iwaniec
\cite{MR1779567}
(see also \cite{MR3394377, MR3968874})
from the rational numbers
to general number fields:
\begin{theorem}\label{thm:CI}
  Let $F$ be a number field with adele ring $\mathbb{A}$, 
  let $\chi$ be a quadratic character of
  $\mathbb{A}^\times/F^\times$,
  and let $\sigma$ be
  either a cuspidal automorphic representation
  of $\PGL_2(\mathbb{A})$
  or a unitary Eisenstein series.
  Then the Weyl-type subconvex bound
  \begin{equation}\label{eq:weyl-bound-sigma-chi}
    L(\sigma \otimes \chi, 1/2)
    \ll_{\sigma}
    C(\chi)^{1/3+\eps}
  \end{equation}
  holds,
  with the implied constant depending polynomially
  upon $C(\sigma)$.
  In particular,
  \begin{equation}\label{eqn:weyl-bound-chi}
    L(\chi,1/2) \ll C(\chi)^{1/6+\eps}.
  \end{equation}
\end{theorem}
Here and henceforth $\eps$ denotes a fixed sufficiently small
positive quantity, whose precise value may change from line to
line, and the asymptotic notation $A \ll B$
or $A = \O(B)$ denotes an estimate of the form
$|A| \leq C |B|$, where the \emph{implied constant}
$C \geq 0$ depends only upon $\eps$ and
the number field $F$.
The refined notation
$A \ll_{x,y,z} B$
or $A = \O_{x,y,z}(B)$
signifies
that $C$ may depend also upon
$x,y,z$.
We write $C(\chi)$ for the analytic conductor,
given by the product over all places $\mathfrak{p}$ 
of the local analytic conductor $C(\chi_\mathfrak{p})$
as defined in \cite[\S3.1.8]{michel-2009} or \S\ref{sec:intro-l-factors}.

The above ``subconvex'' estimates improve upon the
respective ``trivial''
or ``convexity'' bounds of $C(\chi)^{1/2+\eps}$ and
$C(\chi)^{1/4+\eps}$, and also upon earlier nontrivial subconvex
bounds (see \cite{2009arXiv0904.2429B, MR3977317, MR3213837, MR3594414} and
references).  Via period formulas as in \cite{MR3885172, MR4001088,MR2322488} (see also \cite{MR646366,
  MR783554, MR1233447,MR1404335,MR3112415, MR3112415,MR3649366}) these
estimates lead to improved bounds for the Fourier coefficients
of half-integral weight modular forms over number fields
(cf. \cite[Cor 1]{2009arXiv0904.2429B}), hence to improved
estimates for representation numbers of ternary quadratic forms
over number fields.
For instance,
we obtain the following
numerical improvement upon \cite[Cor 2]{2009arXiv0904.2429B},
reducing the exponent
$7/16 + \vartheta/8$
to $5/12$:
\begin{corollary}
  Let $Q$ be a positive integral ternary quadratic form over
  a
  totally real number field $F$.
  For an element $n$ of the ring of integers of $F$,
  let $r_Q(n)$ denote the number of integral representations
  of $n$ by $Q$.
  For squarefree $n$ locally represented by $Q$,
  \begin{equation}
    r_Q(n)
    = r(n)
    + \O_Q(\norm(n)^{5/12+\eps}),
  \end{equation}
  where
  $r(n) = \norm(n)^{1/2+o(1)}$
  is the
  product of local densities as in the Siegel mass formula.
\end{corollary}
As a further application, we may combine Theorem \ref{thm:CI}
with
recent work of Balog, Bir{\'o}, Cherubini and Laaksonen
(see
\cite[Cor 1.4, Rmk 2]{2019arXiv191101800B})
to sharpen the error term in the prime
geodesic theorem for $\mathbb{Q}(i)$ \cite[Thm 5.1]{MR723012},
reducing the exponent $\approx 1.60023$ of \cite[Cor
1.2]{2019arXiv191101800B} to $67/42 \approx 1.59524$:
\begin{corollary}
  Let $\Psi(X)$ denote the Chebyshev-type counting function for
  primitive geodesics on
  $\PSL_2(\mathbb{Z}[i]) \backslash \mathbb{H}^3$,
  as defined in
  \cite[p1]{2019arXiv191101800B}.
  We have
  \begin{equation}
    \Psi(X) = (1/2) X^2
    + \O(X^{67/42+ \eps}).
  \end{equation}
\end{corollary}
\begin{proof}
  Theorem \ref{thm:CI} implies the bound
  $L(\chi,1/2 + it) \ll (1 + |t|)^{\O(1)} C(\chi)^{1/6+\eps}$
  for quadratic characters $\chi$ of
  $\mathbb{A}^\times /F^\times$,
  $F = \mathbb{Q}(i)$.  Inserting this bound into
  \cite[Cor 1.4]{2019arXiv191101800B}
  (and computing
  that $3/2 + (4/7) \cdot (1/6) = 67/42$)
  gives the required
  estimate.
\end{proof}


\subsection{The proposed strategy}\label{sec:intro-motivation}
We summarize the strategy proposed by Michel--Venkatesh for
establishing spectral identities
generalizing those of Motohashi.

Let $F$ be a number field with adele ring $\mathbb{A}$.  Set
$G := \PGL_2(F)$, let $A \leq G$ denote the diagonal subgroup,
and write
$[G] := G \backslash G_{\mathbb{A}}$ and $[A] := A \backslash
A_{\mathbb{A}}$ for the corresponding adelic quotients.  Let
$\mathcal{I}(0)$ denote the representation of $G_\mathbb{A}$
defined by normalized induction of the trivial character
of the standard Borel
(see \S\ref{sec:intro-representations}).  For
$f \in \mathcal{I}(0)$, write $\Eis^*(f)$ for the associated
normalized Eisenstein series (\S\ref{sec-5-2}).
For instance, if $F = \mathbb{Q}$
and $f$ is normalized spherical,
then $\Eis^*(f)$
corresponds
to the derivative $\frac{d }{d s} E(s,z)|_{s=1/2}$
at the central point
of the classical 
$\SL_2(\mathbb{Z})$-invariant
Eisenstein series
$E(s,z) = y^s + \dotsb$.
Working formally for the moment
(ignoring important issues of convergence and regularization),
consider the
(divergent) diagonal integral of the product of two such Eisenstein series,
attached to $f_1, f_2 \in \mathcal{I}(0)$, over
the adelic quotient $[A]$ of the diagonal subgroup:
\begin{equation}\label{eq:integral-product-eis}
  \int_{[A]} \Eis^*(f_1) \Eis^*(f_2).
\end{equation}
We may expand \eqref{eq:integral-product-eis} in two ways.
On the one hand, expanding the product of Eisenstein series
over the spectrum of $[G]$
yields
\begin{equation}\label{eq:intro-spectral-expn}
  \int_{\sigma:\text{generic}}
  \sum _{\varphi \in \mathcal{B}(\sigma)}
  \int_{[G]} \Eis^*(f_1) \Eis^*(f_2) \overline{\varphi }
  \int_{[A]} \varphi 
  + (\dotsb),
\end{equation}
where the integral is over generic standard automorphic
representations $\sigma$ of $\PGL_2(\mathbb{A})$
and $(\dotsb)$ denotes the ``contribution of the
one-dimensional representations.''
(We note that the integral
over $\sigma$ is typically
written as a sum over the cuspidal representations
plus an integral over Eisenstein series,
see \S\ref{sec-9-3} for details
and precise
normalizations.)
On the other
hand, the Parseval relation on the group $[A]$ yields
\begin{equation}\label{eq:intro-parseval-A}
  \int_{\omega\text{:unitary}}
  \left(
    \int_{[A]} \Eis^*(f_1)
    \omega
  \right)
  \left(
    \int_{[A]} \Eis^*(f_2) \omega^{-1}
  \right),
\end{equation}
where the integral is taken over unitary characters $\omega$ of
$[A]$.
By unfolding the global Hecke and
Rankin--Selberg integrals (\S\ref{sec-5-4}, \S\ref{sec-5-6}, \S\ref{sec:global-inv-func})
in each of these expansions,
we ``deduce'' that for factorizable vectors
$f_i = \otimes f_{i \mathfrak{p}}$ and some large enough finite
collection $S$ of places of $F$,
\begin{align}
  \label{eqn:basic-moment-identity}
  &\int_{\substack{
    \sigma:\text{generic}, \\
  \text{unram. outside $S$}
  }}
  \frac{L^{(S)}(\sigma,1/2)^3}{L^{(S),*}(\sigma \times \sigma,
  1)}
  h(\sigma)
  \\ \nonumber
  &\quad =
    (\dotsb) +
    \int_{\substack{
    \omega:\text{unitary}, \\
  \text{unram. outside $S$}
  }}
  \frac{|L^{(S)}(\omega,1/2)|^4}{\zeta_F^{(S),*}(1)^2}
  \tilde{h}(\omega),
\end{align}
where $L^{(S)}(\dotsb)$ denotes a partial $L$-function,
an asterisk signifies taking the first nonvanishing
Laurent coefficient,
and the weights
\[
  h(\sigma) =
  \prod_{\mathfrak{p} \in S}
  h_\mathfrak{p}(\sigma_\mathfrak{p}),
  \quad
  \tilde{h}(\omega) = \prod_{\mathfrak{p} \in S}
  \tilde{h}_\mathfrak{p}(\omega_\mathfrak{p})
\]
are products of local weights given
in terms of local Hecke and Rankin--Selberg integrals
(\S\ref{sec-4-3}, \S\ref{sec-4-5}):
\begin{equation}\label{eq:intro-ell-sigma}
    h_\mathfrak{p}(\sigma_\mathfrak{p})
  :=
  \sum _{W_\mathfrak{p} \in \mathcal{B}(\sigma_{\mathfrak{p}})}
  \int_{N_\mathfrak{p} \backslash G_\mathfrak{p}}
  \overline{W_{\mathfrak{p}}}
  W_{f_{1 \mathfrak{p}}}
  f_{2 \mathfrak{p}} 
  \int_{A_{\mathfrak{p}}} W_\mathfrak{p},
\end{equation}
\begin{equation}\label{eq:intro-ell-omega}
    \tilde{h}_\mathfrak{p}(\omega_\mathfrak{p})
  :=
  \int_{A_\mathfrak{p}}
  W_{f_{1 \mathfrak{p}}} \omega
  \int_{A_\mathfrak{p}}
  W_{f_{2 \mathfrak{p}}} \omega^{-1}.
\end{equation}

\subsection{The basic spectral identity}
While the argument just recorded was highly non-rigorous, we
prove that the conclusion is nevertheless valid
(see Corollary \ref{cor:basic-identity-summary}
for the precise statement):
\begin{theorem}\label{thm:basic-spectr-ident}
\label{thm:intro-basic-decomp-mv}
The identity \eqref{eqn:basic-moment-identity} holds
with
$(\dotsb)$ the limit of a sum of fifteen ``degenerate
functionals''
$\mathcal{I}(0) \otimes \mathcal{I}(0) \rightarrow \mathbb{C}$
defined in \S\ref{sec-12}.
\end{theorem}
The proof begins by deforming the Eisenstein series
generically and ends by taking a limit.
Following Zagier \cite{MR656029} and Michel--Venkatesh
\cite[\S4.3]{michel-2009}, the
product of deformed Eisenstein series differs by a
finite linear combination of
Eisenstein series from an
$L^2$-function, which in turn admits a spectral expansion.  On
the other hand, the diagonal integral
\eqref{eq:integral-product-eis} of that product admits a
canonical regularization and enjoys a modified form of
Parseval's identity \eqref{eq:intro-parseval-A}.
Interchanging the regularized diagonal
integral with the spectral expansion introduces
additional terms.  Some of these arguments were sketched or
suggested by Michel--Venkatesh \cite[\S4.5]{michel-2009} in the
special case that $F = \mathbb{Q}$ and the $f_i$ are spherical.


\subsection{Achieving arbitrarily weighted cubic moments}\label{sec:achi-arbitr-weight}
We turn to the main problem addressed by this paper,
which was
raised in \cite[\S4.5.4]{michel-2009}.

In seeking to apply
Theorem \ref{thm:intro-basic-decomp-mv}, several fundamental
questions arise.  Which weights $h$ (or $\tilde{h}$) arise from
the above scheme applied to some choice of
$f_1, f_2 \in \mathcal{I}(0)$, or more generally of some tensor
$f \in \mathcal{I}(0) \otimes \mathcal{I}(0)$?  Let us call such
weights \emph{admissible}.  It is not \emph{a priori} obvious
that the collection of admissible weights is sufficiently rich
for the sake of applications.  Can one find a nonnegative-valued
admissible $h$ (or $\tilde{h}$) that localizes to a given subset
of its domain?  Can one explicitly relate $h$ to $\tilde{h}$?
Can one efficiently estimate the degenerate terms $(\dotsb)$?
The main point of this paper is to initiate the systematic study
of such questions.

We focus here on studying the cubic moment
via \eqref{eqn:basic-moment-identity}.
To do so effectively,
we need to know that the characteristic functions of
interesting spectral families of automorphic representations
$\sigma$ may be approximated by nonnegative admissible weights
$h(\sigma)$.  This possibility is confirmed by one of our main local
results 
(Theorem \ref{thm:constr-admiss-weight}),
summarized here informally:
\begin{theorem}\label{thm:summarize-kuznetsov-weights-adm}
  Let $h(\sigma)$ be a weight function that shows up on the
  spectral side of the pre-Kuznetsov formula (i.e., the relative
  trace formula for twisted unipotent periods with a
  compactly-supported test function).  Then $h$ is admissible.
  The corresponding dual weight $\tilde{h}(\omega)$ may be
  evaluated via explicit integral transforms.
\end{theorem}
We show by
example in \S\ref{sec:appl-cubic-moment}
that the ``pre-Kuznetsov weights''
$h(\sigma)$ alluded to in
Theorem \ref{thm:summarize-kuznetsov-weights-adm}
are adequate for applications.
We show also that the degenerate
terms may be either discarded altogether
(\S\ref{sec:handl-degen-terms-1})
or evaluated explicitly (\S\ref{sec:handl-degen-terms-2}).

The idea behind the proof
of Theorem \ref{thm:summarize-kuznetsov-weights-adm}
is to consider $f$ belonging to a
specific class of generalized vectors (``Whittaker vectors'')
for which the weights $h(\sigma)$ on the cubic moment side
resemble those appearing in the pre-Kuznetsov formula.  We hope
that this technique will be useful more broadly in
period-based approaches to proving spectral identities
between families of $L$-functions.

Theorems \ref{thm:intro-basic-decomp-mv} and
\ref{thm:summarize-kuznetsov-weights-adm} combine to yield
spectral identities for the cubic moment with nonnegative
weights that localize to any given subset of the spectrum.
One can likely adapt the basic proof technique of Theorem
\ref{thm:summarize-kuznetsov-weights-adm} to produce a rich
class of admissible $\tilde{h}$ and study the inverse transform
from $\tilde{h}$ to $h$,
but we leave such pursuits and
their applications to future work.

\subsection{Applications to subconvexity}
The $L$-values $L(\sigma,1/2)$ are known to be nonnegative, so
taking $h$ nonnegative and estimating the RHS of
\eqref{eqn:basic-moment-identity} yields an upper bound for such
$L$-values.  This feature has been exploited over
$F = \mathbb{Q}$ in the work of Conrey--Iwaniec
\cite{MR1779567}, Ivic \cite{MR1879668} and Petrow and Young
\cite{MR3394377, MR3635360, MR3968874,
  2018arXiv181102452P, 2019arXiv190810346P}.
We hope the methods of this paper may
be useful in generalizing such results to
the setting of number fields.

We have already recorded in \S\ref{sec:intro-overview} a
generalization of the theorem of Conrey--Iwaniec
\cite{MR1779567}.
As a further application,
we present a partial generalization of
a recent result of Petrow--Young \cite{2018arXiv181102452P}:

\begin{theorem}\label{thm:PY1}
  Let $F$ be a number field.  Let $\chi$ be a
  unitary character of $\mathbb{A}^\times/F^\times$
  whose infinite component $\chi_\infty$ is
  trivial and whose finite conductor is cubefree.
  Then the Weyl-type subconvex bound \eqref{eqn:weyl-bound-chi}
  holds.
\end{theorem}

The point of departure for our proof of Theorems \ref{thm:CI}
and \ref{thm:PY1} is similar to that of earlier work over
$\mathbb{Q}$ in that we seek to bound a cubic moment.  The
conclusion of our proof is likewise similar: we eventually
reduce to essentially the same character sum estimates (relying
in turn upon Deligne's results) and fourth moment bounds for
$\GL_1$ $L$-functions as in those works.  The remaining parts of
our arguments differ somewhat.  For instance, a major difficulty
encountered in \cite{MR3968874} (see especially
\cite[\S1.3]{MR3968874}) is that the relevant cubic moment
involves newforms of different conductors, so the approximate
functional equations for their $L$-values have different
lengths, which causes problems when one seeks to estimate an odd
power moment of nonnegative $L$-values via the Petersson
formula.  The authors of \cite{MR3968874} overcome this
difficulty by developing general Petersson formulas for
newforms, which are of independent interest.  Such difficulties
are avoided in the approach pursued here.  We do not use
approximate functional equations.  The cubic moment formula
implied by Theorems \ref{thm:basic-spectr-ident} and
\ref{thm:summarize-kuznetsov-weights-adm} reduces our task to
the local problem of producing weights $h(\sigma)$ majorizing
the family of interest and estimating the integral transforms
defining the dual weights $\tilde{h}(\omega)$ and degenerate
terms $(\dotsb)$.  The same formula can likely be supplemented
with additional local analysis to prove many variations on
Theorems \ref{thm:CI} and \ref{thm:PY1}.  We note also that the
delicate archimedean stationary phase calculations required
already in \cite{MR1779567} are not needed in our approach.





The restriction in Theorem \ref{thm:PY1} to cubefree conductor
arises here due to local phenomena as in
\cite{2018arXiv181102452P}.  Recently Petrow--Young
\cite{2019arXiv190810346P} have removed this restriction from
their earlier result (over $F = \mathbb{Q}$) and established the
bound \eqref{eqn:weyl-bound-chi} over $F = \mathbb{Q}$ for all
unitary $\chi$, i.e., without the quadraticity assumption.  We
establish many of the local estimates relevant for adapting
their strategy to our setting (compare Proposition
\ref{prop:non-arch-estimates-key} with
\cite[\S3]{2019arXiv190810346P}).
To complete this adaptation
requires verifying the following generalization of
\cite[Thm 1.4]{2019arXiv190810346P}:
\begin{conjecture}\label{conj:fourth-moment-cosets}
  Let $F$ be a number field.
  Let $\chi$ be a unitary character of $\mathbb{A}^\times$.  Let
  $B = (B_\mathfrak{p})_{\mathfrak{p}}$ be a collection of real
  numbers $B_\mathfrak{p} \geq 1$, indexed by the places of $F$,
  with $B_\mathfrak{p}$ belonging to the value group of
  $F_\mathfrak{p}$ for all $\mathfrak{p}$ and
  satisfying
  $B_\mathfrak{p} = 1$ for all but finitely many $\mathfrak{p}$.
  Let $\mathcal{F}(\chi,B)$ denote the set of all characters
  $\omega$ of $\mathbb{A}^\times/F^\times$ such that for each
  place $\mathfrak{p}$ of $F$, we have
  \[C(\omega_\mathfrak{p}/\chi_\mathfrak{p}) \leq
    B_\mathfrak{p}.
  \]
  Assume that for each $\mathfrak{p}$,
  we have
  \[
    B_\mathfrak{p} \geq C(\chi_\mathfrak{p})^{2/3}.
  \]
  Then
  \begin{equation}\label{eq:required-fourth-moment-bound-for-generalization}
    \int_{\omega \in \mathcal{F}(\chi,B)}
    |L(\omega,1/2)|^4
    \ll
    (\prod_\mathfrak{p} B_\mathfrak{p})^{1+\eps}.    
  \end{equation}
\end{conjecture}
Conjecture \ref{conj:fourth-moment-cosets} is also relevant for
removing the restriction on $\chi_\infty$ in Theorem
\ref{thm:PY1}: the local phenomena responsible for the
restriction to cubefree conductors show up over any local field
in which $-1$ is a square, and thus arise whenever the number
field $F$ has a complex embedding.
One can study the LHS of
\eqref{eq:required-fourth-moment-bound-for-generalization}
by applying
the basic formula \eqref{eqn:basic-moment-identity}
in the direction opposite to the primary one considered in this
paper (i.e., as in the original work of Motohashi),
but we leave this for future work.

\subsection{Further remarks}
It might be interesting to compare the formulas obtained here
with those in \cite{MR1226527} and \cite{MR1879668}, or to
replace the degenerate Eisenstein series occurring in
\eqref{eq:integral-product-eis} with other Eisenstein or with
cusp forms; the case treated here is in some sense the most
degenerate and (apparently) the most relevant for applications.
In another direction, with some refined local analysis it should
be possible to improve the estimate
\eqref{eq:weyl-bound-sigma-chi} to be simultaneously subconvex
in $\sigma$ and $\chi$ (compare with \cite{MR3635360}).

We mention the works \cite{MR2124019, MR2279942,
  2019arXiv190207042B}, which offer other perspectives on
Motohashi's formula and its generalizations.

\subsection{Organization of this
  paper}\label{sec:organ-this-paper}
In Part \ref{part:main-ideas}, we aim to introduce the main
ideas of this paper with minimal technical overhead.  In
\S\ref{sec-6}, we give the precise statement and proof of
Theorem \ref{thm:summarize-kuznetsov-weights-adm}, which allows
us to study arbitrarily weighted cubic moments in terms of
periods of Eisenstein series.  In
\S\ref{sec:appl-cubic-moment}, we define and study a class of
weights concentrating on the ``short families'' of primary
interest in applications of the cubic moment to subconvexity.
We explain in particular how the two-variable exponential sums
that featured in the work of Conrey--Iwaniec and Petrow--Young
arise naturally from local representation-theoretic
considerations.  In \S\ref{sec-8}, we take for granted the basic
spectral identity, Theorem \ref{thm:basic-spectr-ident}, and
give the proofs of our main applications,
Theorems \ref{thm:CI} and \ref{thm:PY1}, modulo some
technicalities concerning the degenerate terms $(\dotsb)$ and
the required polynomial dependence of
\eqref{eq:weyl-bound-sigma-chi} upon $\sigma$.  The remainder of
the paper is then devoted to addressing those technicalities
and proving
Theorem \ref{thm:basic-spectr-ident}.

Part \ref{part:preliminaries} contains preliminaries of a
general nature, concerning regularized integration
(\S\ref{sec-9-1}), Sobolev norms on representations of reductive
groups (\S\ref{sec:norms-repr}), the local (\S\ref{sec-4}) and
global (\S\ref{sec-5}) theory of integral representations of
$L$-functions, and (regularized) spectral decompositions of the
space of automorphic forms (\S\ref{sec-5}).  Our Sobolev norm
discussion is similar to that of \cite[\S2]{michel-2009}, but
applies also to non-unitary representations, as is convenient
when studying degenerate integrals like
\eqref{eq:integral-product-eis} via deformation.

Part \ref{part:basic-identity} gives the precise statement and
proof of Theorem \ref{thm:basic-spectr-ident}.  We begin by
studying in detail the functionals \eqref{eq:intro-ell-sigma}
and \eqref{eq:intro-ell-omega}, in local
(\S\ref{sec:local-inv-func}) and global
(\S\ref{sec:global-inv-func}) settings.  In \S\ref{sec-12}, we
relate these two families of functionals by decomposing in two
ways a deformation of the integral \eqref{sec-12}.

Part \ref{part:analys-local-weights} contains several results
needed for the proofs of our main applications.
In particular, we estimate the degenerate terms
in cases relevant to our applications.

\subsection{General notation and conventions}\label{sec:gener-notat-conv}

\subsubsection{Asymptotic notation and terminology}
Let us recall and elaborate upon the conventions introdued earlier.
The notation $A = \O(B)$ or $A \ll B$ signifies that
$|A| \leq C |B|$ for some \emph{fixed} quantity $C$,
while $A \asymp B$ is shorthand for $A \ll B \ll A$.
Here and henceforth, we consider
a quantity to be \emph{fixed} if we explicitly label it as such,
or if it depends only upon some previously defined fixed
quantities.  The small positive parameter $\eps > 0$ is always
regarded as fixed.
For instance, we have $x^n \ll \exp(\eps x)$ for any fixed
natural number $n$ and all positive reals $x$.

When considering a number field $F$ and a nontrivial unitary
character $\psi$ of its adele class group $\mathbb{A}/F$, we
regard the pair $(F,\psi)$ as fixed.  When working over a local
field $F$ equipped with a nontrivial unitary character $\psi$,
we consider the pair $(F,\psi)$ as fixed \emph{except} when $F$
is non-archimedean and $\psi$ is unramified, in which case we
only regard the absolute degree of $F$ as fixed.  This
convention ensures that implied constants remain uniform as the
pair $(F,\psi)$ traverses the local components of corresponding
global data.


\subsubsection{Local fields}
Let $F$ be a local field,
thus $F$ is either $\mathbb{R}$, $\mathbb{Q}_p, \mathbb{F}_p(t)$
or a finite extension of one of these fields.
When $F$ is non-archimedean, we denote
by $\mathfrak{o}$ the ring of integers and by $\mathfrak{p}$ the
maximal ideal, and set $q := \# \mathfrak{o}/\mathfrak{p}$.
When $F$  is archimedean,
it will be convenient to set $q := 1$.

We will often consider local fields $F$ equipped with nontrivial
unitary characters $\psi : F \rightarrow \U(1)$.  Recall that if
$F$ is non-archimedean, then $\psi$ is \emph{unramified} if it
is trivial on $\mathfrak{o}$ but not on $\mathfrak{p}^{-1}$.  We
say that the pair $(F,\psi)$ is \emph{unramified} if $F$ is
non-archimedean and $\psi$ is unramified.

\subsubsection{Characters}
A \emph{character} of a topological group $G$ is a continuous
homomorphism $\chi : G \rightarrow \mathbb{C}^\times$; a
\emph{unitary character} is one with image in the unit circle
$\U(1)$.

When $G$ is the multiplicative group $F^\times$ of a local field
$F$ or the idele class group $\mathbb{A}^\times/F^\times$ of a
number field $F$, the normalized absolute value $|.|$
defines a character of $G$, and every positive-valued character
is of the form $|.|^c$ for some real number $\chi$.  The
\emph{real part} $\Re(\chi)$ of any character
$\chi : G \rightarrow \mathbb{C}^\times$ is then characterized
by the identity $|\chi| = |.|^{\Re(\chi)}$.

\subsubsection{Local zeta functions}
For a local field $F$,
we write $\zeta_F(s)$ for the local zeta function, given
respectively by $\pi^{-s/2} \Gamma(s/2)$ or
$2 (2 \pi)^{-s} \Gamma(s)$ or $(1 - q^{-s})^{-1}$ according as
$F$ is real or complex or non-archimedean.

\subsubsection{Matrix notation}
When working with $\PGL_2$ over a ring,
we use the notation
\[
  n(x) := \begin{pmatrix}
    1 & x \\
    0 & 1
  \end{pmatrix},
  \quad  a(y) :=
  \begin{pmatrix}
    y & 0 \\
    0 & 1
  \end{pmatrix},
  \quad
  n'(z) :=
  \begin{pmatrix}
    1 & 0 \\
    z & 1
  \end{pmatrix}
  \quad
  w := \begin{pmatrix}
     & -1 \\
    1 & 
  \end{pmatrix}.
\]
We will use that $w n(x) = n(-1/x) a(1/x^2) n'(1/x)$
for invertible $x$.

\subsubsection{Representations}\label{sec:intro-representations}
Let $F$ be a local field,
and let $G$ be a reductive group over $F$.  By a
``representation'' $\sigma$ of $G$, we always mean a smooth
representation in the non-archimedean case (see
\cite[\S4.2]{MR1431508}) and a smooth moderate growth
Fr{\'e}chet representation in the archimedean case (see
\cite{MR1013462}, \cite[\S11]{MR1170566}, \cite{MR3219530}).  In
particular, any vector $v \in \sigma$ is assumed smooth.

Given natural numbers $r_1,\dotsc,r_k$ with sum
$r := \sum_j r_j$, we may form the standard parabolic subgroup
$P = M U$ of $\GL_{r}(F)$, with
$M \cong \GL_{r_1}(F) \times \dotsb \times \GL_{r_k}(F)$.  Given
representations $\chi_j$ of $\GL_{r_j}(F)$ ($j=1..k$), we write
\begin{equation}\label{eqn:parabolic-induction-chi-1-through-chi-k}
  \Ind(\chi_1 \boxtimes \dotsb \boxtimes \chi_k)
\end{equation}
for the
normalized induction from $P$ to $G$ of the corresponding
representation of $M$.
Given a character $\chi$ of $F^\times = \GL_1(F)$,
we abbreviate
\begin{equation}
  \mathcal{I}(\chi) := \Ind(\chi \boxtimes
  \chi^{-1});
\end{equation}
it defines a representation of $\GL_2(F)$ with trivial central
character, hence a representation of $\PGL_2(F)$
consisting of smooth functions
$f : \PGL_2(F) \rightarrow \mathbb{C}$
satisfying
\begin{equation}\label{eq:f-of-n-x-a-y-g-equals-blah}
  f(n(x) a(y) g)
  = |y|^{1/2} \chi(y) f(g).
\end{equation}
We write
simply $\mathcal{I}(s)$ for $\mathcal{I}(|.|^s)$.  We will never
use the potentially ambiguous notation $\mathcal{I}(1)$.

Following \cite[\S1.5]{MR3753910}, by a \emph{Whittaker type}
representation $\sigma$ of $\GL_r(F)$ we mean one of the form
\eqref{eqn:parabolic-induction-chi-1-through-chi-k} with the
$\chi_j$ essentially square-integrable, and in particular
generic.
Thus
$\chi_j = \chi_j^0 \otimes |.|^{c_j}$ with $\chi_j^0$
square-integrable and $c \in \mathbb{R}$,
and
\begin{equation}\label{eqn:whittaker-type-induction}
  \sigma =
  \Ind(
  \chi_1^0 \otimes |.|^{c_1}
  \boxtimes
  \dotsb
  \boxtimes
  \chi_k^0 \otimes |.|^{c_k}
  ).
\end{equation}
(In \cite[\S2.1]{MR701565}, ``Whittaker type'' has a more
general meaning.)

We distinguish between a Whittaker type representation
and its isomorphism class.
Every generic irreducible representation is
isomorphic to at least one representation of Whittaker type.
Writing
$\GL_r(F)^\wedge_{\gen}$ for the set of isomorphism classes of
generic irreducible representations of $\GL_r(F)$, we have a
surjective map with finite fibers
\begin{equation}\label{eqn:from-whittaker-type-to-generic-0}
  \{\text{Whittaker type representations of } \GL_r(F)\}
  \rightarrow \GL_r(F)^\wedge_{\gen}
\end{equation}
given by taking the unique generic subquotient.

A Whittaker type representation of $\PGL_2(F)$ is either
square-integrable or of the form $\mathcal{I}(\chi)$ for some
character $\chi$ of $A$.

\subsubsection{Bounds toward
  Ramanujan}\label{sec:bounds-towards-raman}
Let $F$ be a local field.  We say that a Whittaker type
representation $\sigma$ of $\GL_r(F)$ is
\emph{$\vartheta$-tempered} if in
\eqref{eqn:whittaker-type-induction}, each $c_j$ is bounded in
magnitude by $\vartheta$.  Then $\sigma$ is $0$-tempered iff it
is tempered in the customary sense.

A Whittaker type representation $\sigma$ of $\GL_2(F)$ is
$\vartheta$-tempered precisely when
\begin{itemize}
\item $\sigma = \sigma_0 \otimes |.|^c$
  with $\sigma_0$
  square-integrable and $|c| \leq \vartheta$, or
\item $\sigma = \Ind(\mu_1 \boxtimes \mu_2)$
  of some characters $\mu_1, \mu_2$ of $F^\times$
  with $|\Re(\mu_i)| \leq \vartheta$.
\end{itemize}
It is known \cite{MR2811610} that for $\GL_2$ over a number
field, the local component of any cuspidal automorphic
representation with unitary central character is
$7/64$-tempered.

\subsubsection{$L$-factors}\label{sec:intro-l-factors}
The local $L$-factors considered in this paper are defined and
studied in work of Jacquet--Piatetski-Shapiro--Shalika
\cite{MR701565} and Jacquet \cite{MR0401654}.
Over any local field $F$,
those works attach
to any pair of Whittaker type representations
$\sigma_1, \sigma_2$ of $\GL_{r_1}(F), \GL_{r_2}(F)$
an $L$-factor $L(\sigma_1 \otimes \sigma_2, s)$,
$\eps$-factor $\eps(\psi,\sigma_1 \otimes \sigma_2, s)$
and $\gamma$-factor
\begin{equation}\label{eqn:defn-local-rs-gamma-factor}
  \gamma(\psi,\sigma_1 \otimes \sigma_2, s) = \eps(\psi,\sigma_1
  \otimes \sigma_2, s) \frac{L(\tilde{\sigma}_1 \otimes
    \tilde{\sigma}_2, 1 - s)}{L(\sigma_1 \otimes \sigma_2, s)}.
\end{equation}
When $r_2 = 1$ and $\sigma_2$ is trivial
one writes simply $L(\sigma_1,s)$ for $L(\sigma_1 \otimes
\sigma_2, s)$, and similarly for the $\eps$- and
$\gamma$-factors.

The local $L$-factors
$L(\sigma_1 \otimes \sigma_2, s)$
may be written
\begin{equation}\label{eqn:factorization-of-local-L-factor-into}
  \prod_{j=1}^n \zeta_F(s -
  u_j),
\end{equation}
where $\{u_j\}_{j=1}^n$ is a collection of complex
numbers with $n \leq r_1 r_2$.
If $\sigma_i$
is $\vartheta_i$-tempered,
then
$\Re(u_j) \leq \vartheta_1 + \vartheta_2$.
If is a character of $F^\times$, then $L(\chi,s)$
has the form \eqref{eqn:factorization-of-local-L-factor-into}
with $n \leq 1$ and $\Re(u_j) \leq \Re(\chi)$.
The $L$-factors are multiplicative with respect to induction
in the sense that
\begin{equation}\label{eqn:multiplicativity-L-factors-wrt-induction}
  L(
  \Ind(\chi_1 \boxtimes \dotsb \boxtimes \chi_k)
  \otimes 
  \Ind(\omega_1 \boxtimes \dotsb \boxtimes \omega_{\ell})
  ,s)
  = \prod_{i,j}
  L(\chi_i \otimes \omega_j, s),
\end{equation}
and similarly for the $\eps$- and $\gamma$-factors.

The local $\gamma$-factors may be characterized by the local
functional equation of \emph{loc. cit.}, which we recall below
in special cases as needed.  The local $\eps$-factors will not
play an important role for us.

We define the analytic conductor
$C(\sigma_1 \otimes \sigma_2,s) = C(\sigma_1 \otimes \sigma_2
\otimes |.|^s) \in \mathbb{R}_{\geq 1}$ as in
\cite[\S3.1.8]{michel-2009}: in
the non-archimedean case it is defined by the relation
$\eps(\psi_0, \sigma_1 \otimes \sigma_2, s) =
C(\sigma_1
\otimes \sigma_2, s)^{1/2-s}
\eps(\psi_0, \sigma_1 \otimes \sigma_2, 1/2)$
for an unramified character
$\psi_0$ of $F$, while in the archimedean case it is defined as
$C(\sigma_1 \otimes \sigma_2, s) := \prod_{j} (2 + |\mu_j + s|)$
if
$L(\sigma_1 \otimes \sigma_2,s) = \prod_j \zeta_{\mathbb{R}}(s +
\mu_j)$.
We abbreviate
$C(\sigma_1 \otimes \sigma_2) := C(\sigma_1 \otimes \sigma_2,0)$.
For our purposes, the key property of the analytic conductor is
that
if $\sigma_1$ and $\sigma_2$ are unitary,
then
\begin{equation}\label{eq:stirling-for-general-RS-gamma}
  \gamma(\psi,\sigma_1 \otimes \sigma_2,1/2 + s)  \asymp
  C(\sigma_1 \otimes \sigma_2,s)^{-\Re(s)} q^{\O(1)}
\end{equation}
provided that $s$ is at least some fixed positive distance away
from any poles or zeros of the LHS of
\eqref{eq:stirling-for-general-RS-gamma}.  In the archimedean
case, the estimate \eqref{eq:stirling-for-general-RS-gamma} is a
consequence of Stirling's formula.  In the non-archimedean case,
it follows from the definition of the analytic conductor
in terms of the $\eps$-factor and the fact that
$|\eps(\psi_0,\sigma_1 \otimes \sigma_2,1/2)| = 1$ in the
unitary case; the factor $q^{\O(1)}$ crudely bounds the ratio of
$L$-factors in \eqref{eqn:defn-local-rs-gamma-factor} and may be
omitted when those $L$-factors are identically $1$ or
when $s$ is chosen so that their ratio is $\asymp 1$.
The
estimate \eqref{eq:stirling-for-general-RS-gamma} may be usefully applied in conjunction with the
multiplicativity property of $\gamma$-factors with respect to
induction.

\subsubsection{$L$-functions}
When $F$ is a number field,
we denote completed $L$-functions (including archimedean
factors) by $\Lambda(\dotsb)$ and their finite parts by
$L(\dotsb)$.  For a finite set of places $S$ that contains every
archimedean place, we write $L^{(S)}(\dotsb)$ for the Euler
product taken over $\mathfrak{p} \notin S$.

We write
$\xi_F(s) = \prod_{\mathfrak{p}} \zeta_{F_\mathfrak{p}}(s)$ for
the Dedekind zeta function and $\zeta_F^{(S)}(s)$ for the
corresponding product over $\mathfrak{p} \notin S$.
We use a superscripted asterisk,
as in $\xi_F^*(1)$ or $\zeta_F^{(S)*}(1)$
or $\Lambda^*(\sigma \times \sigma, 1)$ (for a generic
automorphic representation $\sigma$ of $\PGL_2$),
to denote the first nonvanishing Laurent coefficient (typically
the residue).



\subsubsection{Groups and measures}
\label{sec-4-1}
When working over a local field $F$, we
generally use the notation
\[
  G := \PGL_2(F),
\]
$N := \{n(x) : x \in F\},A :=\{a(y) : y \in F^\times \}$,
$N' := \{n'(z) : z \in F\}$, and $B := N A$.
(The exception to this convention is
\S\ref{sec:norms-repr},
in which we work more generally).  We let
$K \leq G$ denote the standard maximal compact subgroup.  We
identify $N \cong F$, $A \cong F^\times$, and $N' \cong F$ in
the evident way; in particular, we identify their character
groups.

We denote by $A^\wedge$ the group of characters
$\chi : A \rightarrow \mathbb{C}^\times$, or equivalently,
$\chi : F^\times \rightarrow \mathbb{C}^\times$.
It has the
natural structure of a Riemann surface with at most countably
connected components, with charts given by
$\chi |.|^s \mapsto s$.

We denote by $G^\wedge_{\gen}$ the set of isomorphism classes of
generic irreducible representations $\sigma$ of $G$.  Via the
map \eqref{eqn:from-whittaker-type-to-generic-0}, we may regard
$G^\wedge_{\gen}$ as a quotient of the set of Whittaker type
representations of $G$.  The latter set is naturally a complex
manifold.
We equip $G^\wedge_{\gen}$ with the quotient topology,
and say that a function
$h : G^\wedge_{\gen} \rightarrow \mathbb{C}$ is
\emph{holomorphic} if it pulls back to a holomorphic function on
the set of Whittaker type representations.

Given a nontrivial unitary character $\psi$ of $F$, we equip the
additive group $F$ with the $\psi$-self-dual Haar measure $d x$,
so that the Fourier inversion formula
$\int_{x \in F} (\int_{y \in F} f(y) \psi(x y) \, d y) \, d x =
f(0)$ holds for $f \in C_c^\infty(F)$.  We equip the
multiplicative group $F^\times$ with the measure
$\frac{d y}{|y|}$.  We note that if $(F,\psi)$ is unramified,
then
\begin{equation}\label{eq:volume-of-unit-group}
  \vol(\mathfrak{o}^\times, \tfrac{d y}{|y|})
  = 1/\zeta_F(1).
\end{equation}
By our identifications, we obtain Haar measures on $N, A, N'$.
We equip $G$ with the Haar given by the pushforward of
$d x \, \frac{d y}{|y|} \, d z$ under the map
$F \times F^\times \times F \rightarrow G$,
$(x,y,z) \mapsto n(x) a(y) w n(z)$, or equivalently, under
$(x,y,z) \mapsto n(x) a(y) n'(z)$.  If $(F,\psi)$ is unramified,
then the chosen Haar measure on $G$ assigns volume
$1/\zeta_F(2)$ to $K$ (see \cite[\S3.1.6]{michel-2009} and
\eqref{eq:volume-of-unit-group}).
It will be convenient
for us to equip $K$ with the Haar measure $d k$
for which the Haar measure on $G$
is given by the pushforward
of $d x \, \frac{d y}{|y|} \, d k$
under $(x,y,k) \mapsto n(x) a(y) k$;
the volume of $d k$ is then $\asymp 1$.

We will use freely that any irreducible representation of $G$
is self-dual, i.e., isomorphic to its own contragredient
(see \cite[Exercise 2.5.8, Thm 4.2.2]{MR1431508}).



\part{Main ideas}\label{part:main-ideas}
In this part, we precisely formulate and prove Theorem
\ref{thm:summarize-kuznetsov-weights-adm}.  We then record most
of the arguments required to deduce our main applications
(Theorems \ref{thm:CI} and \ref{thm:PY1}).  Along the way, we
explain how the character sums appearing in the works of
Conrey--Iwaniec and Petrow--Young arise naturally from the
perspective of representation theory.


The reader who is not familiar with integral representations of
$L$-functions and the related local representation theory might
wish to peruse \S\ref{sec-4} and \S\ref{sec-5} before
proceeding.

\section{Basic study of local weights}
\label{sec-6}
Let $F$ be a local field, and let $\psi$ be a nontrivial unitary
character of $F$.

\subsection{Some induced representations of $G$}
\label{sec-6-2}





Recall that for a character $\chi$ of $F^\times \cong A$,
we write $\mathcal{I}(\chi)$ for the corresponding normalized
induced representation of $G$,
consisting of
smooth functions $f : G \rightarrow \mathbb{C}$
satisfying
\eqref{eq:f-of-n-x-a-y-g-equals-blah}.
We are interested primarily in the
representation
$\mathcal{I}(0) = \mathcal{I}(|.|^0)$ induced by the trivial character,
consisting of smooth functions $f : G \rightarrow \mathbb{C}$
satisfying
\begin{equation}\label{eqn:transformation-law-I-of-0}
  f(n(x) a(y) g)
  = |y|^{1/2} f(g).
\end{equation}
It is a generic irreducible unitary representation.  It arises
naturally as a local component of the Eisenstein series central
to this paper.

For $f \in \mathcal{I}(0)$,
we write $W_f$
for the Whittaker function,
given for $g \in G$ by
\begin{equation}\label{eqn:W-f-via-f-0-defn}
  W_f(g) := 
  \int_{x \in F}
  f(w n(x)g) \psi(-x) \, d x.
\end{equation}
This and similar integrals must be understood in general via
meromorphic continuation or regularization.  For instance, in
the non-archimedean case,
a standard regularization
\cite[\S1.9]{MR3889963}
is
given by summing the integrals over cosets of the
unit group.  One may similarly regularize in the archimedean
case via a smooth dyadic partition of unity
or convolution (see \eqref{eq:whittaker-function-explicated-defn} for details).

We also write $W_f : F^\times \rightarrow \mathbb{C}$
for the corresponding Kirillov model element,
given for $t \in F^\times$ by
\begin{equation}
  W_f(t) :=
  W_f(a(t))
  = |t|^{1/2}
  \int_{x \in F}
  f(w n(x)) \psi(-t x) \, d x,
\end{equation}
which satisfies  the (convergent) Fourier inversion formula
\begin{equation}\label{eqn:f-from-Wf-Fourier-inv}
  \int_{t \in F}
  |t|^{-1/2} W_f(t) \psi(t x) \, d t
  = f (w n(x))
\end{equation}
(see around \eqref{eq:fourier-inversion-for-whittaker-intertwiner}
for details).

\subsection{A representation $\pi$ of $G \times G$}\label{sec:representation-pi-g}
We denote by $\mathcal{I}(0) \otimes \mathcal{I}(0)$ the space
of smooth functions $f : G \times G \rightarrow \mathbb{C}$
satisfying
\[
  f(n(x_1) a(y_1) g_1, n(x_2) a(y_2) g_2)
  = |y_1 y_2|^{1/2} f(g_1,g_2).
\]
It defines a representation of $G \times G$.
We will refer often to this representation,
so we abbreviate
\begin{equation}\label{eqn:defn-pi-I0-I0}
  \pi := \mathcal{I}(0) \otimes \mathcal{I}(0)
\end{equation}
A pair of elements $f_1, f_2 \in \mathcal{I}(0)$
defines an element $f_1 \otimes f_2 \in \pi$
by the rule $(f_1 \otimes f_2)(g_1,g_2)
= f_1(g_1) f_2(g_2)$.
In the non-archimedean case, $\pi$ identifies with the usual tensor product;
in the archimedean case, it may be interpreted
as the completed tensor product,
regarding $\mathcal{I}(0)$
as a
nuclear Fr{\'e}chet space.
In all cases, $\pi$ satisfies the universal property: any
continuous bilinear form on $\mathcal{I}(0)$ extends uniquely to
a continuous linear functional on $\pi$,
as follows (for instance) from
\S\ref{sec:reduct-pure-tens} below.
In our global analysis, an element of $\pi$
gives local data describing a linear combination of products of
pairs of Eisenstein series, as in
\eqref{eq:integral-product-eis}.


\subsection{Some $A$-invariant functionals on $\pi$}
\label{sec-6-3}
We consider two families
of $A$-invariant functionals
$\pi \rightarrow \mathbb{C}$,
the motivation for which
should be clear from \S\ref{sec:intro-motivation}.
(Here we identify $G$ with a subgroup of $G \times G$
via the diagonal embedding;
in particular, we regard $A$
as a subgroup of $G \times G$.)

\subsubsection{$G \times G \geq G \geq A$}
One family is indexed by generic irreducible
representations
$\sigma$ of $G$:
for each such representation,
we define
\begin{equation}\label{eq:defn-ell-sigma-basic}
  \ell_\sigma (f_1 \otimes f_2)
  :=
  \sum_{W \in \mathcal{B}(\sigma)}
  (\int_{N \backslash G}
  \tilde{W} W_{f_1}  f_2)
  \int_A W
\end{equation}
Here and henceforth $W$ traverses a basis
for  the Whittaker model $\mathcal{W}(\sigma,\psi)$
and
$\tilde{W}$
the corresponding element
of a dual basis for the dual Whittaker model
$\mathcal{W}(\sigma,\bar{\psi})$,
with the duality given by regularized integration over $A$
(\S\ref{sec:whitt-intertw-dual}).
We note in passing that $\sigma$,
like any irreducible representation of $G$, is self-dual,
so it is unnecessary
to pass to the contragredient.

The integral over $N \backslash G$ is to
be understood as a local Rankin--Selberg integral (see
\S\ref{sec-4-5} for details); it converges if $\sigma$ is
unitary, and may be meromorphically continued to all $\sigma$
for which $L(\sigma,1/2)$ is finite.  The integral over $A$ is
a local Hecke integral to which similar remarks apply (see
\S\ref{sec-4-3} for details).

The precise choice of basis $\mathcal{B}(\sigma)$
is described in \S\ref{sec:norms-whit-type-reps}.
In particular,
we take $W, \tilde{W}$
to be $K$-isotypic.
The sum \eqref{eq:defn-ell-sigma-basic}
then
converges as written;
see \S\ref{sec:defn-ell-sigma-s-makes-sense} for details.

The definition may be understood
without reference to a  basis:
$f_1 \otimes f_2 \mapsto \sum_{W \in \mathcal{B}(\sigma)} (\int_{N
  \backslash G} \tilde{W}  W_{f_1}  f_2) W$ is the
linear map $\pi \rightarrow \mathcal{W}(\sigma,\psi)$ adjoint to
the Rankin--Selberg functional
$\pi \otimes \mathcal{W}(\sigma,\overline{\psi }) \rightarrow
\mathbb{C}$.
The existence of the required adjoint
is closely related to the convergence
of the indicated sum over suitable bases.

We will be primarily (although not exclusively) concerned with
the case that $\sigma$ is unitary, in which case one may take
for $\tilde{W}$ the complex conjugate of $W$ and for
$\mathcal{B}(\sigma)$ an orthonormal basis of
$\mathcal{W}(\sigma,\psi)$, with the norm defined by the
absolutely-convergent integral over $A$.


\subsubsection{$G \times G \geq A \times A \geq A$}
The other family of functionals is indexed
by characters $\omega$ of $F^\times$:
for each such character,
we define
\[
  \ell_{\omega}(f_1 \otimes f_2)
  :=
  (\int_{A}
  W_{f_1} \omega )
  (\int_{A}
  W_{f_2} \omega^{-1} ).
\]
The integrals are understood as local Hecke integrals (\S\ref{sec-4-3});
they converge for unitary $\omega$,
and extend meromorphically to all $\omega$
for which $L(\omega, 1/2)$ is finite.

\subsubsection{Further remarks}
These definitions extend continuously
from pure tensors to general elements of $\pi$
(see \S\ref{sec:defn-ell-sigma-s-makes-sense} for
proof of continuity).
They may be
understood as compositions
\begin{equation}\label{eqn:ell-sigma-arrows}
  \ell_{\sigma} : \pi \rightarrow \sigma \rightarrow \mathbb{C},
\end{equation}
\begin{equation}\label{eqn:ell-omega-arrows}
  \ell_{\omega} : \pi \rightarrow
  \omega^{-1} \otimes \omega \rightarrow \mathbb{C}.
\end{equation}
The first arrow in \eqref{eqn:ell-sigma-arrows}
is $G$-invariant,
and described by the Rankin--Selberg integral;
the second is $A$-invariant,
and described by the Hecke integral.
The first arrow in \eqref{eqn:ell-omega-arrows}
is $A \times A$-invariant,
and given by a pair of Hecke integrals;
the second is the $A$-invariant canonical pairing.

\subsection{Variation with respect to the additive character}\label{sec:vari-with-resp-1}
The definitions of the functionals $\ell_{\sigma}$ and
$\ell_{\sigma}$ depend upon the choice of additive character
$\psi$.  The dependence is mild: changing $\psi$ to
$\psi^b(x) := \psi(b x)$ has the effect of multiplying these
functionals by an explicit power of $|b|$.  For the purpose of
proving estimates, there is thus no harm in assuming $\psi$ to
be of a convenient form, e.g., unramified when $F$ is
non-archimedean.  We record explicit formulas describing this
variation, in a more general setting, in
\S\ref{sec:vari-with-resp}.

\subsection{Definition of admissible weights}
\label{sec-6-4}
Recall the general notation of \S\ref{sec-4-1}.
We say that
a function
\[
  h : G^\wedge_{\gen} \rightarrow \mathbb{C},
\]
respectively,
\[
  \tilde{h} : A^{\wedge} \rightarrow \mathbb{C},
\]
is
\emph{admissible} if for some $f \in \pi$
we have
\[\text{$h(\sigma) = \ell_\sigma(f)$ for all $\sigma$,}\]
respectively,
\[\text{$\tilde{h}(\omega) = \ell_{\omega}(f)$
    for all $\omega$.}\]
We say that
$h$ and $\tilde{h}$
are dual 
if they may be defined using
the same $f$.

\subsection{Crude estimates}\label{sec:crude-estimates}
\begin{lemma}\label{lem:crude-estimates}
  Any admissible weight $h$ or $\tilde{h}$ as above is
  meromorphic on its domain, with poles controlled by local
  $L$-factors in the sense
  described in \S\ref{sec:local-inv-func}.
  There
  are no poles on the unitary subsets of
  $A^\wedge, G^\wedge_{\gen}$.

  Any admissible weight is bounded and rapidly-decreasing
  on the unitary subset
  $\{\text{unitary } \sigma \in G^\wedge_{\gen}\}$.
  More precisely:
  \begin{itemize}
  \item If $F$ is non-archimedean, then $h$ and $\tilde{h}$ are
    uniformly bounded on the unitary subsets, and vanish on
    arguments of sufficiently large conductor.
  \item If $F$ is archimedean, then the restrictions of
    $h$ and
    $\tilde{h}$ to the unitary subset are bounded by constant
    multiples of $C(\cdot)^{-d}$ for each $d \geq 0$.
  \end{itemize}
\end{lemma}
The conclusion follows from the smoothness of $f$ and
``integration by parts''
(see around
\eqref{eq:crude-estimate-for-ell-sigma-s}
and
\eqref{eq:crude-ell-omega-s-estimate}
for details).

\subsection{Models}
\label{sec-6-5}
Here we consider a pair of models of the representation
$\pi$ that arise
naturally when studying the functionals
$\ell_{\sigma}$ and $\ell_{\omega}$.

\subsubsection{}
We write $C^\infty(N \backslash G, \psi)$
for the space of smooth functions $V$ on $G$
satisfying $V(n(x) g) = \psi(x) V(g)$.

For $f = f_1 \otimes f_2 \in \pi$,
we define
\[V_f \in C^\infty(N \backslash G, \psi),
  \quad
  W_f \in C^\infty(F^\times \times F^\times)
\]
by
the formulas
\begin{equation}
  V_f(g) := W_{f_1}(g) f_2(g),
\end{equation}
\begin{equation}
  W_f(t_1,t_2) := W_{f_1}(t_1) W_{f_2}(t_2).
\end{equation}
The map $f \mapsto V_f$
arose also in \cite[\S3.2.7]{michel-2009}.

These definitions extend continuously to all $f \in
\pi$,
and may be defined directly
by the formulas
\begin{equation}\label{eqn:V_f-formula}
  V_f(g)
  = \int_{x \in F} f(w n(x) g, g) \psi(-x) \, d x,
\end{equation}
\begin{equation}\label{eqn:W-f-formula}
  W_f(t_1,t_2)
  =
  |t_1 t_2|^{1/2}
  \int_{x_1,x_2 \in F}
  f(w n(x_1), w n(x_2)) \psi(- t_1 x_1 - t_2 x_2) \, d x_1 \, d x_2.
\end{equation}
The functionals defined previously
are then given by
\begin{equation}\label{eqn:ell-sigma-via-V-f-sigma}
  \ell_\sigma(f)
  =
  \int_{A} (V_f)_{\sigma}
  \text{ with }
  (V_f)_\sigma :=
  \sum_{W \in \mathcal{B}(\sigma)}
  (\int_{N \backslash G}
  \tilde{W}  V_f)
  W,
\end{equation}
\begin{equation}
  \ell_\omega(f)
  =
  \int_{A \times A}
  W_f \cdot (\omega \otimes \omega^{-1} ).
\end{equation}

The maps $f \mapsto V_f$ and $f \mapsto W_f$ define models of
$\pi$.
Our aim in this section is to verify that these models
contain the compactly-supported elements and to establish
integral transforms relating them.
We will see later that such transforms
are at the heart
of the spectral identities
discussed in \S\ref{sec-2}.

\subsubsection{}
We define for any $V \in C_c^\infty(N \backslash G,\psi)$
the integral transform
$V^\wedge : F^2 \rightarrow \mathbb{C}$ by
\begin{equation}\label{eqn:V-wedge-defn}
  V^\wedge(\xi,z)
  :=
  \int_{y \in F^\times}
  |y|^{-1} V(a(y) n'(z)) \psi(\xi y) \, d y.
\end{equation}
For $f \in \pi$, we may define
$V_f^\wedge : F^2 \rightarrow \mathbb{C}$ by the same formula.
We note that
if $V$ is supported on
the dense open
neighborhood $N A N'$ of the identity in $G$,
then the map $F^\times \times F \ni (y,z) \mapsto V(a(y) n'(z))$
is compactly-supported,
and so $V^\wedge$ belongs to the Schwartz space
$\mathcal{S}(F^2)$.
\begin{lemma}
  For $f \in \pi$
  and almost all $(\xi,z) \in F^2$,
  we have
  \begin{align}\label{eqn:formula-Vf-wedge-via-f-1}
    V_f^\wedge(\xi,z)
    &=
      f(w n(\xi) n'(z), n'(z))
    \\
    \label{eqn:formula-Vf-wedge-via-f-2}
    &=
      |x_2 - x_1|
    f(w n(x_1), w n(x_2))
  \end{align}
  where
  \begin{equation}
    x_1 =
    \frac{\xi }{1 + \xi z},
    \quad
    x_2 = \frac{1}{z},
  \end{equation}
  so that
  \begin{equation}\label{eqn:xi-z-via-x2-x1}
    \xi = \frac{x_2 x_1}{x_2 - x_1},
    \quad
    z = \frac{1}{x_2}.
  \end{equation}
\end{lemma}
\begin{proof}
  The first identity follows (for $f = f_1 \otimes f_2$,
  hence in general)
  from
  \eqref{eqn:f-from-Wf-Fourier-inv}, the second from the readily
  verified identities
  \[
    w n(\xi) n'(z)
    =
    n(\dotsb)
    a (\frac{1}{(1 + \xi z)^2})
    w n(x_1),
  \]
  \[
    n'(z) = n(1/z) a(1/z^2) w n(x_2),
  \]
  \[
    x_2 - x_1
    = \frac{1}{z(1+\xi z)}
  \]
  and the transformation law \eqref{eqn:transformation-law-I-of-0}.
\end{proof}

\subsubsection{}
We verify here that our models contain
the compactly-supported elements.
\begin{lemma}\label{lem:many-W-f}
  The set $\{W_f : f \in \pi \}$
  contains $C_c^\infty(F^\times \times F^\times)$.
\end{lemma}
\begin{proof}
  This may be deduced from the corresponding property of the
  Kirillov model of $\mathcal{I}(0)$, or proved in the same way.
\end{proof}
We have recorded lemma \ref{lem:many-W-f} for motivational
purposes; we do not use it directly in the present work.
More relevant for our purposes is the following analogue:

\begin{lemma}\label{lem:many-V-f}
  The set $\{V_f : f \in \pi \}$ contains
  $C_c^\infty(N \backslash G, \psi)$.
\end{lemma}
Before giving the proof, we sketch informally why this should be
true.  Take $f = f_1 \otimes f_2$, where
$y \mapsto W_{f_1}(a(y))$ and $z \mapsto f_2(n'(z))$ approximate
$\delta$-masses
at
$y = 1$ and $z=0$, respectively.  Then
$V_f(a(y) n'(z))$ approximates a $\delta$-mass at $y=1, z=0$.
By taking linear combinations of translates of $f$, we can
approximate any element of $C_c^\infty(N \backslash G, \psi)$,
giving something like the required conclusion.


\begin{proof}
  Both sets in question are invariant under right translation by
  $G$, so we reduce to verifying that $\{V_f : f \in \pi \}$
  contains every element $V$ of
  $C_c^\infty(N \backslash G, \psi)$ whose support is contained
  in some small neighborhood of the identity.  In particular, we
  may assume that $V$ is supported on $N A N'$,
  so that $V^\wedge$ is Schwartz.

  We aim to construct $f \in \pi$ with $V_f = V$.  It is enough
  to check that $V_f^\wedge = V^\wedge$.  It seems simplest to
  us to use the following Schwartz space parametrization: for
  each $\Phi \in \mathcal{S}(F^3)$, there is a unique $f \in \pi$
  for which
  \[
    f(g, n'(z))
    = |\det g|^{1/2}
    \int_{r \in F}
    \Phi((0,r) g, z) \, d r.
  \]
  Let us compute $V_f^\wedge$ for such an $f$.
  First, we specialize \eqref{eqn:V_f-formula}
  to see that
  \begin{align*}
    V_f(a(y) n'(z))
    =
    |y|
    \int_{\xi \in F}
    f(w n(\xi) n'(z), n'(z))
    \psi(-\xi y) \, d \xi.
  \end{align*}
  Since $(0,r) w n(\xi) n'(z)
  = (r (1 + \xi z), r \xi)$,
  the definition of $f$
  then gives
  \[
    V_f(a(y) n'(z))
    =
    |y|
    \int_{\xi,r \in F}
    \Phi(r ( 1 + \xi z), r \xi, z)
    \psi(-\xi y) \, d \xi \, d r,
  \]
  hence
  \[
    V_f^\wedge(\xi,z) = 
    \int_{r \in F}
    \Phi(r (1 + \xi z), r \xi, z)
    \, d r.
  \]
  This formula suggests the choice
  \begin{equation}\label{eqn:Phi-via-phi-0-V-wedge}
    \Phi(x,y,z)
    :=
    \phi_0(x - y z)
    V^\wedge (\frac{y}{x - y z}, z)
  \end{equation}
  for some
  $\phi_0 \in C_c^\infty(F^\times)$
  with $\int_{r \in F} \phi_0(r) \, d r = 1$.
  Then $V_f^\wedge = V^\wedge$.
\end{proof}

\subsubsection{}
For $V \in C_c^\infty(N \backslash G, \psi)$,
we define the integral
transform $V^\sharp : F^2 \rightarrow \mathbb{C}$ by
the convergent integral
\begin{equation}\label{eqn:V-sharp-defn}
  V^\sharp(x,y)
  :=
  \int _{\xi \in F}
  V^\wedge(\xi,-x/\xi) \psi(-\xi y) \, d \xi.
\end{equation}

\begin{lemma}\label{lem:W_f-via-V_f-0}
  For $f \in \pi$,
  we have
  \begin{equation}\label{eqn:W-f-via-V-f-sharp}
    W_f(t_1,t_2)
    =
    |t_1 t_2|^{1/2}
    \int_{x \in F}
    V_f^\sharp (x, \frac{(t_1 + t_2) x - t_2}{x(1-x)})
    \, \frac{d x}{|x(1-x)|}.
  \end{equation}
\end{lemma}
The RHS of \eqref{eqn:W-f-via-V-f-sharp} does not in general
converge absolutely, but may be regularized as in the definition
\eqref{eqn:W-f-via-f-0-defn} of the Whittaker intertwiner,
either by analytic continuation with respect to the parameters
of the induced representations
$\mathcal{I}(s_1) \otimes \mathcal{I}(s_2)$ deforming
$\pi = \mathcal{I}(0) \otimes \mathcal{I}(0)$, or using smooth
dyadic partitions of unity near the singular points $x =0$ and
$x = 1$.  Indeed, the proof below will show that the $x$ and
$\xi$ integrals in \eqref{eqn:V-sharp-defn} and
\eqref{eqn:W-f-via-V-f-sharp} arise via a change of variables
from a pair of Whittaker intertwiners
\eqref{eqn:W-f-via-f-0-defn}.
\begin{proof}
  By the formulas \eqref{eqn:W-f-formula} and
  \eqref{eqn:formula-Vf-wedge-via-f-2} relating $W_f$ and
  $V_f^\wedge$ to $f$,
  we have
  \begin{equation}
    W_f(t_1,t_2)
    =
    |t_1 t_2|^{1/2}
    \int_{x_1, x_2 \in F}
    \psi(-t _1 x_1 - t_2 x_2)
    V_f^\wedge(\xi,z) \, \frac{d x_1 \, d x_2}{|x_2-x_1|},
  \end{equation}
  where $\xi, z$ are given by \eqref{eqn:xi-z-via-x2-x1}.
  We set $x := - \xi/x_2$, so that
  $x_1 = \xi / (1-x), x_2 = - \xi/x$
  and
  \[
    t_1 x_1 + t_2 x_2 = \xi \frac{(t_1 + t_2) x - t_2}{x(1-x)}.
  \]
  The Jacobian calculation
  $|\partial(x_1,x_2)/\partial(\xi,z)|
  = |x_2 - x_1|/|x(1-x)|$
  then yields \eqref{eqn:W-f-via-V-f-sharp}.
\end{proof}
  

\begin{remark}
  Lemmas  \ref{lem:many-V-f} and \ref{lem:W_f-via-V_f-0}
  should hold (in modified form) for
  any representation $\pi$ of $G \times G$ of the form
  $\pi_1 \otimes \pi_2$, with $\pi_1$ generic irreducible and
  $\pi_2$ belonging to the principal series.  The proof
  technique indicated above applies when $\pi_1$ also belongs to
  the principal series (see \S\ref{sec:holom-famil-local}); in
  general, one can argue using the local functional equation for
  $\pi_1$.  The case indicated above is the most relevant one
  for our applications.
\end{remark}

\subsection{Pre-Kuznetsov weights}\label{sec:pre-kuzn-weights}
\begin{definition}
  By a \emph{pre-Kuznetsov weight}
  $h : G_{\gen}^\wedge \rightarrow \mathbb{C}$,
  we mean
  a function
  for which there exists
  $\phi \in C_c^\infty(G)$
  so that
  \begin{equation}\label{eqn:defn-h-via-phi}
    h(\sigma) =
    \sum_{W \in \mathcal{B}(\sigma)}
    (\int_G \tilde{W} \phi)
    W(1)
  \end{equation}
  for all $\sigma$.
  In that case, we refer to $\phi$ as a \emph{kernel} for $h$.
\end{definition}
We give several examples of such weights in
\S\ref{sec:constr-suit-weight}.  We note the sum
\eqref{eqn:defn-h-via-phi} converges absolutely and that any
such $h$ is holomorphic (\S\ref{sec:holomorphy-of-pre-K-wts}).
Moreover, if $\phi = \phi_1 \ast \phi_2$ is the convolution
of $\phi_1, \phi_2 \in C_c^\infty(G)$,
then
\begin{equation}
    h(\sigma) = \sum_{W \in \mathcal{B}(\sigma)}
  (\int_G \tilde{W} \phi_1)
  (\int_G W \phi_2^\iota),
  \quad
  \phi_2^{\iota}(g) := \phi_2(g^{-1}).
\end{equation}

\begin{remark}
  In the archimedean case, one could generalize the above
  definition to allow rapidly decaying $\phi$ in the sense of
  \cite[p392]{MR1013462}, but doing so is not necessary for our
  purposes.
\end{remark}

We record here a cheap construction that is often useful at
``bad primes.''
\begin{lemma}\label{lem:crude-lower-bound-individual}
  For each compact subset $\Sigma \subseteq G^\wedge_{\gen}$
  consisting of unitary representations,
  there is a pre-Kuznetsov weight
  $h : G^\wedge_{\gen} \rightarrow \mathbb{C}$ for which
  \begin{itemize}
  \item $h(\sigma)  \geq 0$ for all
    unitary $\sigma \in G^\wedge_{\gen}$, and
  \item $h(\sigma) \geq 1$ for all $\sigma  \in \Sigma$.
  \end{itemize}
\end{lemma}
\begin{proof}
  By continuity and compactness, we may assume that
  $\Sigma = \{\sigma _0\}$ is a singleton.  Let
  $W_0 \in \sigma_0$ with $W_0(1) = 1$.  By continuity, we may
  find a small neighborhood $U$ of the identity in $G$ so that
  $\Re(W_0(u)) \geq 1/2$ for all $u \in U$.  Choose a
  nonnegative-valued $\phi_0 \in C_c^\infty(U)$ with
  $\int \phi_0 = 1$.  Take for $\phi$ the convolution
  $\phi_0^* \ast \phi_0$, where
  $\phi_0^*(g) := \overline{\phi_0(g^{-1})}$.  Take for $h$ the
  pre-Kuznetsov weight with kernel $\phi$.  Then
  \[
    h(\sigma)
    = \sum_{W \in \mathcal{B}(\sigma)}
    |\phi_0 \ast W|^2(1) \geq 0.
  \]
  Moreover, $\Re(\phi_0 \ast W_0) \geq 1/2$, so likewise
  $h(\sigma_0) \geq (1/4) \|W_0\|^{-2} > 0$.  We conclude by
  replacing $h$ with a suitable positive multiple.
\end{proof}
Lemma \ref{lem:crude-lower-bound-individual} has the
disadvantage of being ineffective.  For instance, invoking it in
the proof of Theorem \ref{thm:CI} would lead to a weaker form of
the estimate \eqref{eq:weyl-bound-sigma-chi} in which the
implied constant depends in an unspecified manner (rather than
polynomially) upon $\sigma$.  A more complicated but effective
variant will be given below in \S\ref{sec:crude-local-estim}.

\subsection{Whittaker vectors}\label{sec:whittaker-vectors}

\subsubsection{}
Let $\sigma$ be a smooth representation of $G$ (smooth moderate
growth Fr{\'e}chet
in the archimedean case).
The main example here is when $\sigma$
is the restriction to $G \hookrightarrow G \times G$
of $\pi$.

By a \emph{generalized vector} in $\sigma$, we mean an element
of the algebraic dual of the contragredient.  Any closely
related definition (e.g., via negative index Sobolev spaces as
in \cite[\S2]{michel-2009} or \cite[\S3.2]{nelson-venkatesh-1}
or \S\ref{sec:norms-repr}) would also work for us.  We identify
$\sigma$ with a subspace of the space of generalized vectors.
The $G$-action extends naturally to this larger space.

By a \emph{$\psi$-Whittaker vector} (or
simply \emph{Whittaker vector} when $\psi$ is understood) we
mean a generalized vector $v$ for which $n(x) v = \psi(x) v$ for
all $x \in F$.

\subsubsection{}
Given any (smooth) vector $v$,
the formula
\begin{equation}
  N_\psi v := \int_{x \in F}
  \psi(-x) n(x) v \, d x
\end{equation}
defines a Whittaker vector in $\sigma$.
We note that $N_\psi$ commutes with $G$-equivariant homomorphisms.

\subsubsection{}
Given any $\phi_0 \in C_c^\infty(A)$,
the convolution
\begin{equation}\label{eqn:phi0-compose-N-psi-v}
  \phi_0 \ast N_\psi v :=
  \int_{a \in A}
  \phi_0(a) a N_\psi v
\end{equation}
defines an actual vector, i.e., an element of $\sigma$.
We may give a direct definition of this vector
via the formula
\begin{equation}\label{eqn:direct-defn-phi0-N-psi}
  \phi_0 \ast N_\psi v
  =
  \int_{x \in F}
  (\int_{u \in F^\times}
  \phi_0(1/u)
  \psi(-x u)
  n(x) a(1/u) v
  \, \frac{d u }{|u|})
  \, d x
\end{equation}
obtained from \eqref{eqn:phi0-compose-N-psi-v}
by writing $a = a(1/u)$ substituting
$x \mapsto xu$.
The iterated integral
\eqref{eqn:direct-defn-phi0-N-psi}
converges in the indicated order.
(To see this,
integrate by parts in the $u$-integral with
respect to the phase $\psi(-x u)$
in
the archimedean case,
and
use that $v$ is invariant by an open subgroup of the unit group
in the non-archimedean case.)

It follows that for any
functional $\ell$ on $\pi$ that transforms under $A$ with
respect to a character $\nu$, we may canonically define
$\ell(N_\psi v)$ to be $\ell(\phi_0 \ast N_\psi V)$ for any
$\phi_0$ whose Mellin transform takes the value $1$ at $\nu$.

\subsubsection{}
We record how $N_\psi$ acts on the $\psi$-Kirillov model of a
generic irreducible representation $\sigma$ of $G$.  Let
$W \in \sigma$, realized as
$W : F^\times \rightarrow \mathbb{C}$ with
the actions
$n(x) W(t) = \psi(x t ) W(t)$, $a(y) W(t) = W(t y)$.  Then
$N_\psi W = W(1) \delta_1$, where $\delta_1$ denotes the
$\delta$-mass at $1 \in F^\times$.  It follows that
\begin{equation}\label{eqn:action-phi0-N-psi-on-kirillov-model}
  \phi_0 \ast N_\psi W(t) =
  \phi_0(1/t) W(1).
\end{equation}

\subsubsection{}
For the sake of concreteness, we will work throughout the paper
primarily with the smooth vectors defined by
\eqref{eqn:phi0-compose-N-psi-v} rather than the generalized
vectors $N_\psi v$.

\subsection{Admissibility of
  pre-Kuznetsov weights}\label{sec:constr-admiss-weight}
\label{sec-6-6}
We are now prepared to state the precise form of Theorem
\ref{thm:summarize-kuznetsov-weights-adm}.

\begin{theorem}\label{thm:constr-admiss-weight}
  Let $h : G_{\gen}^\wedge \rightarrow \mathbb{C}$
  be a pre-Kuznetsov
  weight
  with kernel $\phi \in C_c^\infty(G)$.
  Then $h$ is admissible (\S\ref{sec-6-4}).
  An admissible dual
  $\tilde{h} : A^\wedge \rightarrow \mathbb{C}$
  is
  given by the convergent formulas
  \begin{equation}\label{eqn:h-tilde-of-omega-via-h-sharp}
    \tilde{h}(\omega)
    =
    \int_{t \in F}
    h^\sharp(t)
    \omega (\frac{1-t}{t})
    \, \frac{d t}{|t(1-t)|^{1/2}}
  \end{equation}
  where
  \begin{equation}
    h^\sharp(t)
    = \int_{x \in F}
    V^\sharp(x, \frac{x-t}{x(1-x)})
    \, \frac{d x}{|x(1-x)|}
  \end{equation}
  where $V^\sharp$ is defined
  as above in terms of  the element
  $V \in C_c^\infty(N \backslash G, \psi)$
  defined by
  \begin{equation}\label{eqn:defn-V-via-phi}
    V(g) := \int_{x \in F} \phi(n(x) g) \psi(-x) \, d x.
  \end{equation}
\end{theorem}
\begin{proof}
  Note that $\int_G \tilde{W}  \phi
  = \int_{N \backslash G} \tilde{W}  V$.
  By lemma \ref{lem:many-V-f},
  we may find $f_0 \in \pi$ so that
  $V_{f_0} = V$.
  Define
  \begin{equation}\label{eqn:V-sig-proj-defn}
    V_\sigma(g) := \sum_{W \in
      \mathcal{B}(\sigma)}
    (\int_{N \backslash G} V \tilde{W}) W(g).
  \end{equation}
  (See \S\ref{sec:holomorphy-of-pre-K-wts}
  for details regarding convergence.)
  Then $V_\sigma(1) = h(\sigma)$.

  Choose $\phi_0 \in C_c^\infty(F^\times) \cong C_c^\infty(A)$
  with $\int_{A} \phi_0 = 1$,
  set $\phi_0^{\iota}(y) := \phi_0(1/y)$,
  and define
  (with notation as in \S\ref{sec:whittaker-vectors})
  \begin{equation}\label{eqn:defn-f-via-f0-basic-case}
    f
    :=
    \phi_0^{\iota} \ast N_\psi f_0
    \in \pi.
  \end{equation}
  In the Kirillov model,
  we
  may expand
  \[
    W_f(t_1,t_2)
  =
  \int_{x \in F}
  \int_{u \in F^\times}
  \phi_0(u)
  \psi(((t_1 + t_2)/ u - 1) x)
  W_{f_0}(t_1 /u, t_2 /u)
  \, d x
  \, \frac{d u}{|u|}
  \]
  and apply Fourier inversion
  to see that
  \begin{equation}\label{eqn:W-f-via-W-f0}
    W_f(t_1,t_2)
    =
    \phi_0(t_1 + t_2)
    W_{f_0} (\frac{t_1}{t_1 + t_2}, \frac{t_2}{t_1 + t_2}).
  \end{equation}
  For each $\sigma$, the projection
  $(V_f)_{\sigma}$
  defined by analogy to \eqref{eqn:V-sig-proj-defn}
  is given by  
  \begin{equation}
    (V_f)_\sigma(a(y))    = \phi_0(y)    (V_{f_0})_\sigma(1),
  \end{equation}
  by \eqref{eqn:action-phi0-N-psi-on-kirillov-model}.
  Thus
  \begin{equation}
    \ell_{\sigma}(f)    =    (V_{f_0})_\sigma(1)    = h(\sigma).
  \end{equation}
  It remains only to verify the required formula
  for $\ell_{\omega}(f)$.
  By definition,
  \[
    \ell_{\omega}(f)
    =
    \int_{t_1,t_2}
    \omega(t_1/t_2)
    W_f(t_1,t_2)
    \, \frac{d t_1 \, d t_2}{|t_1 t_2|}
    =
    \int_{s,t}
    \omega(s)
    W_f(s t,t)
    \, \frac{d s \, d t}{|s t|}.
  \]
  Inserting \eqref{eqn:W-f-via-W-f0}
  gives
  \begin{equation}
    \ell_{\omega}(f)
    =
    \int_{s,t}
    \omega(s)
    \phi_0((s+1)t)
    W_{f_0}
    (\frac{s}{1+s},
    \frac{1}{s + 1})
    \, \frac{d s \, d t}{|s t|}.
  \end{equation}
  Invoking the normalization of $\phi_0$,
  the above integral simplifies to
  \[
    \int_{s}
    \omega(s)
    W_{f_0}
    (\frac{s}{s + 1},
    \frac{1}{s + 1})
    \, \frac{d s}{|s|}.
  \]
  The substitution $s := 1/(1-t)$
  then yields
  \[
    \int_{t}
    \omega(\frac{1-t}{t})
    W_{f_0}
    (1 - t, t)
    \, \frac{d t}{|t(1-t)|}.
  \]
  We evaluate
  $W_{f_0}$ using \eqref{eqn:W-f-via-V-f-sharp}
  to arrive at
  \eqref{eqn:h-tilde-of-omega-via-h-sharp}.
  The convergence in each step above follows
  from that of the Hecke integrals $\ell_{\omega}(f)$.
\end{proof}

\begin{remark}
  One can verify, using the Whittaker--Plancherel
  theorem for $G$, that 
  \begin{equation}
    \int_{\text{unitary }\sigma \in G^\wedge_{\gen}}
    |h(\sigma)|^2
    = \int_{\text{unitary }\omega \in A^\wedge}
    |\tilde{h}(\omega)|^2,
  \end{equation}
  where the integrals are taken with respect
  to the Plancherel measures dual to the measures
  defined on $G$ and $A$, respectively.  
\end{remark}

\begin{remark}
  A pre-Kuznetsov weight $h$ typically admits many kernels
  $\phi$, so the reader may find it unnatural that $\tilde{h}$
  is expressed in terms of $\phi$ rather than $h$.
  One can express $\tilde{h}$ directly in
  terms of the Bessel transform
  $J_h(g) := \int_\sigma h(\sigma) J_\sigma(g)$, where
  $J_\sigma(g)$ denotes the Bessel distribution
  $\sum_{W \in \mathcal{B}(\sigma)}\tilde{W}(1) W(g)$,
  but our experience suggests that it is more efficient
  to pass first through $\phi$.
\end{remark}

\section{Local estimates for short families}\label{sec:appl-cubic-moment}

\subsection{Families}\label{sec:short-families}
Here we record some notation for referring in a unified manner
to families of representations corresponding to ``short''
families of automorphic forms, such as
\begin{itemize}
\item for a large positive real $T$,
  the set of Maass forms of eigenvalue $1/4 + t^2$
  for some $t \in [T-1, T + 1]$
  (corresponding below to the family
  $\Sigma_{\mathbb{R}}(\omega)$
  with $\omega = |.|^{i T}$), or
\item
  for a large prime $p$,
  the set of twists by the quadratic character $(\cdot|p)$
  of a newform on $\SL_2(\mathbb{Z})$ or $\Gamma_0(p)$
  (corresponding below to the family $\Sigma_{\mathbb{Q}_p}(\omega)$,
  with $\omega$ a ramified quadratic character of $\mathbb{Q}_p^\times$).
\end{itemize}

Let $F$ be a local field.  We call a character $\chi$
of
$F^\times$ \emph{analytically unramified} if
\begin{itemize}
\item $F$ is non-archimedean and $\chi$ is unramified
  in the usual sense (trivial restriction
  to the unit group), or
\item $F$ is archimedean and $\chi = |.|^{s}$
  for some $s \in \mathbb{C}$ with $|\Im(s)| \leq 1$.
\end{itemize}
For a character $\chi$ of $F^\times$, let $\Omega_F(\chi)$
denote the set of characters $\omega$ of $F^\times$ for which
the ratio $\omega/\chi$ is analytically unramified,
and let
$\Sigma_F(\chi)$ denote the set of irreducible representations
$\sigma$ of $\PGL_2(F)$ for which
there exists $\omega \in \Omega_F(\chi)$
so that either
\begin{itemize}
\item $\sigma$ is the principal series representation
  $\mathcal{I}(\omega)$ induced by
  $\omega$ (\S\ref{sec:intro-representations}), or
\item
  $F$ is non-archimedean, $\omega$ is quadratic and $\sigma$ is the twist by
  $\omega$ of the Steinberg representation of $G$
  (thus $\sigma$ is the unique irreducible subrepresentation
  of $\mathcal{I}(\omega |.|^{1/2})$).
\end{itemize}


\subsection{Construction of suitable weights}\label{sec:constr-suit-weight}
Let $F$ be a local field,
let $\psi$ be a nontrivial unitary character of $F$,
and let $\chi$ be a unitary character
of $F^\times$.
We denote by $Q := C(\chi)$ its analytic conductor.
We aim to construct a pre-Kuznetsov weight
$h$ that is nonnegative on unitary representations
and uniformly bounded from below on the family $\Sigma_F(\chi)$.

We give the construction of $h$ separately in the
non-archimedean and archimedean cases, but the reader will
notice very close parallels.

\subsubsection{Non-archimedean case}\label{sec:constr-wt-non-archimedean-case}
Suppose that  $F$ is non-archimedean.
We assume in this case that $\chi$ is ramified,
so that $Q \in \{q, q^2, q^3, \dotsc \}$.
We assume also that $\psi$ is unramified;
this last assumption simplifies slightly the construction,
but has no effect on the subsequent estimates,
as explained in \S\ref{sec:vari-with-resp-1}.

Let $J \leq G$ denote the compact open subgroup consisting of
elements of the form
\[
  \text{$g = n(x) a(y) n'(z)$
    with $|x| \leq 1, |y| = 1, |z| \leq 1/Q$.}
\]
We note that $\vol(J) = \zeta_F(1)/Q$.

Let $\chi_J$ denote the
character of $J$ given by $n(x) a(y) n'(z) \mapsto \chi(y)$.

Define
$\phi \in C_c^\infty(G)$ to be supported on $J$ and given there
by $\chi_J^{-1}$.

Let $h : G^\wedge_{\gen} \rightarrow \mathbb{C}$ denote the
pre-Kuznetsov weight with kernel $\phi$.

\subsubsection{Archimedean case}\label{sec:constr-wt-archimedean-case}
Suppose now that $F$ is archimedean, thus $F = \mathbb{R}$ or
$F = \mathbb{C}$.

We take $\alpha_0 \in (0,1)$ sufficiently small but fixed, and
then take $\alpha_1 \in (0,1)$ fixed but small enough in terms
of $\alpha_0$.  We set
\begin{equation}
  J := \left\{ n(x) a(y) n'(z) :
    |x|, |y-1|, Q |z| \leq \alpha_1
  \right\},
\end{equation}
and note that $\vol(J) \asymp 1/Q$.

We take for $\phi_0 \in C_c^\infty(G)$ a ``smoothened
characteristic function of $J$,'' by which we mean more
precisely an element with the following properties:
\begin{itemize}
\item $\phi_0$ is supported in $J$, and nonnegative.
\item $\int_G \phi_0 \asymp \vol(J)$.
\item
  For any fixed multi-indices $\alpha,\beta,\gamma
  \in \mathbb{Z}_{\geq 0}^{[F:\mathbb{R}]}$, 
  \[
    \partial_x^{\alpha}
    \partial_y^{\beta}
    \partial_z^{\gamma}
    \phi_0 (n(x) a(y) n'(z))
    \ll
    Q^{|\gamma|}.
  \]
  Here the partial derivatives are the usual ones when
  $F = \mathbb{R}$ and are given by differentiating with respect
  to the real and imaginary coordinates when $F = \mathbb{C}$.
  Here
  $|\gamma| = \sum_{1 \leq j \leq [F:\mathbb{R}]} |\gamma_j|$,
  as usual.
\end{itemize}
Such an element exists (e.g., take a product of suitably rescaled bump
functions of the coordinates $x,y,z$).

Let $\chi_J$ denote the function on $J$ given by
$n(x) a(y) n'(z) \mapsto \chi(y)$.

We may define
the multiple $\phi_1 := \chi_J^{-1} \phi_0 \in C_c^\infty(G)$ of
$\phi_0$.  Let $\phi_1^*(g) := \overline{\phi_1(g^{-1})}$ denote
its adjoint, and define the convolution product
$\phi := \vol(J)^{-1} \phi_1^* \ast \phi_1 \in C_c^\infty(G)$.

Let $h$ denote the
pre-Kuznetsov weight with kernel $\phi$.

\subsection{Lower bounds for weights}\label{sec:lower-bounds-for-wts}
\begin{theorem}\label{thm:lower-bounds-weights}
  Fix $\vartheta \in [0,1/2)$.
  Let $h$ be as in
  \S\ref{sec:constr-wt-non-archimedean-case}
  ($F$ non-archimedean, $\chi$ ramified)
  or
  \S\ref{sec:constr-wt-archimedean-case}
  ($F$ archimedean, $\chi$ general).
  Let $\sigma \in G_{\gen}^\wedge$ be unitary.
  \begin{enumerate}[(i)]
  \item We have $h(\sigma) \geq 0$.
  \item If $\sigma$ is $\vartheta$-tempered (\S\ref{sec:bounds-towards-raman})
    and belongs to $\Sigma_F(\chi)$,
    then
    $h(\sigma) \gg_{\vartheta} 1/Q$.
  \item\label{item:assertion-about-unramified-vanishing-of-h}
    If $F$ is non-archimedean
    and $\sigma$ is unramified,
    then
    $h(\sigma) = 0$.
  \end{enumerate}
\end{theorem}

\begin{proof}[Proof in the non-archimedean case]
  Since $\chi_J$ is a character of $J$, the normalized element
  $\vol(J)^{-1} \phi$ defines an idempotent in the convolution
  algebra $C_c^\infty(G)$.
  Thus $h(\sigma) \geq 0$.
  More precisely, by \eqref{eqn:defn-h-via-phi},
  we see that $h(\sigma)$ is the sum of $\vol(J) |W(1)|^2$ taken
  over $W$ in an orthonormal basis for the space
  \[
    \sigma^{\chi_J}   := \{W \in \sigma : g W = \chi_J(g) W
    \text{ for all } g \in J\}.
  \]

  We now determine a
  criterion for when $\sigma^{\chi_J}$
  is nontrivial
  (see \cite[Lem 22]{nelson-padic-que}
  for a related argument).
  We verify readily that the image of $\sigma^{\chi_J}$
  under the twisting map
  $\sigma \rightarrow \sigma \otimes \chi^{-1}$
  consists of all vectors
  transforming under the group
  \[
    K_0(Q) := \{\begin{pmatrix}
      a & b \\
      c & d
    \end{pmatrix} : |a| = |d| = 1, |b| \leq 1, |d| \leq 1/Q\}
    \leq \GL_2(F)
  \]
  via the character
  \[
    \begin{pmatrix}
      a & b \\
      c & d
    \end{pmatrix} \mapsto \chi^{-2}(d).
  \]
  By newvector theory,
  it follows that
  \begin{equation}\label{eq:nontrivial-chi-J-invariants-criterion}
    \sigma^{\chi_J} \neq \{0\}
    \iff
    C(\sigma \otimes \chi^{-1}) \leq Q = C(\chi).
  \end{equation}

  Assertion \eqref{item:assertion-about-unramified-vanishing-of-h}
  follows:
  under its hypotheses,
  we have $C(\sigma \otimes \chi^{-1}) = C(\chi^{-1})^2 = Q^2 >
  Q$,
  so $\sigma^{\chi_J} = \{0\}$ and thus $h(\sigma) = 0$.

  Since $\vol(J) \asymp 1/Q$,
  the proof of the remaining
  assertions will be complete
  once we show that for each
  $\vartheta$-tempered $\sigma \in \Sigma_F(\chi)$
  that there is a nonzero
  $W \in \sigma^{\chi_J}$
  with $|W(1)| \gg_{\vartheta} \|W\|$.
  If $\sigma$ belongs to the principal
  series, then
  $\sigma = \mathcal{I}(\chi \eta)$
  with $\eta$ unramified,
  so
  $C(\sigma \otimes \chi^{-1}) = C(\eta) C(\chi^{-2} \eta^{-1})
  = C(\chi^{-2}) \leq C(\chi)$
  (see \cite[p8]{Sch02}).  If
  $\sigma$ is a twist of Steinberg, then
  $\sigma$ is an irreducible subrepresentation
  of $\mathcal{I}(|.|^{1/2} \chi \eta)$
  with $\eta$ unramified,
  so
  $C(\sigma \otimes \chi^{-1}) = q \leq C(\chi)$
  (see \emph{loc. cit.}).
  In either case, the 
  inequality in \eqref{eq:nontrivial-chi-J-invariants-criterion}
  holds,
  and so $\sigma^{\chi_J} \neq 0$.

  The proof of the criterion
  \eqref{eq:nontrivial-chi-J-invariants-criterion}
  shows moreover that if
  $\sigma \in \Sigma_F(\chi)$, then $\sigma^{\chi_J}$
  contains the inverse image of a newvector $W$ under
  the
  map $\sigma \rightarrow \sigma '$
  for some unitary
  twist $\sigma '$ of $\sigma$.  As recalled in
  \S\ref{sec:whitt-intertw-dual},
  we have $\|W\|^2 = |W(1)|^2 L(\ad \sigma ', 1) / \zeta_F(2)$
  for any such $W$.
  Since $\sigma$ and hence also $\sigma '$
  is $\vartheta$-tempered, we have $L(\ad \sigma ', 1)
  \asymp_{\vartheta} 1$,
  thus $W(1) \gg_{\vartheta} \|W\|$.
  (The parameter $\vartheta$
  is relevant only when $\sigma$ is a quadratic twist of a non-tempered
  unramified
  representation,
  since otherwise both representations are tempered.)
  The twisting map
  $\sigma \rightarrow \sigma '$ is unitary, so the inverse image
  of $W$ has the same norm as $W$, and the required lower bound
  $W(1) \gg_{\vartheta} \|W\|$ follows.
\end{proof}
\begin{proof}[Proof in the archimedean case]
  (We note that the results proved
  here are not used
  in the present paper, but should be of
  direct use in future work.)
  We have
  \begin{equation}
    h(\sigma) =
    \vol(J)^{-1} \sum_{W \in \mathcal{B}(\sigma)}
    | \phi_1 \ast W (1) |^2,
  \end{equation}
  so $h(\sigma) \geq 0$.  Suppose now that
  $\sigma \in \Sigma_F(\chi)$.
  Since $\vol(J) \asymp 1/Q$,
  it suffices to show that there
  is a unit vector $W \in \sigma$ with
  $|\phi_1^* \ast W(1)| \gg 1/Q$.  The proof will be very
  similar to that given above in the non-archimedean case, but
  making use of the ``analytic newvector theory''
  given in \S\ref{sec:analyt-newv-oper} rather than the ``usual''
  newvector theory.
  The reader might wish to skim \S\ref{sec:analyt-newv-oper}
  before proceeding.

  Set $\sigma ' := \sigma \otimes \chi^{-1}$.  We may assume
  $\alpha_0$ taken small enough that Theorem
  \ref{lem:produce-whittaker-against-random-central-character-wannabe}
  holds with $\delta = 1/2$ and $\eps = \alpha_0$.  Let
  $W' \in \sigma'$ denote the unit vector produced by that
  result, and let $W \in \sigma$ denote the inverse image of
  $W'$.
  With notation as in \S\ref{sec:analyt-newv-oper},
  set \[J ' := K_1(C(\sigma'), C(\omega_{\sigma '} \omega^{-1}),
    \alpha_0) \subseteq \GL_2(F).\]
  The conclusion of Theorem
  \ref{lem:produce-whittaker-against-random-central-character-wannabe}
  then reads
  \begin{equation}
    |W'(g) - \eta_{\chi^{-2}}(g)| \leq 1/2
    \text{ for all } g \in J'.
  \end{equation}

  The central
  character $\omega_{\sigma '}$ is $\chi^{-2}$, so with
  $\omega := \chi^{- 2}$ we have
  $C(\omega_{\sigma '} \omega^{-1}) \ll 1$ (perhaps equal to
  $1$, depending upon one's convention).
  We may write $\sigma$
  as the normalized induction of some character $\eta$ of
  $F^\times$, with $\eta/\chi$ analytically unramified.  From
  this it follows that
  $C(\sigma ') = C(\eta/\chi) C(\eta^{-1}/\chi) \asymp
  C(\chi^{-2}) \asymp Q$.
  Since $\alpha_1$ is small in terms of
  $\alpha_0$, we see that
  $J$ is contained
  in the image of $J'$ under the quotient map $\GL_2(F)
  \twoheadrightarrow G$.
  In other words, each $g \in J$ admits
  a lift $\tilde{g} \in J'$.

  We note also that the product $\chi(\det g)
    \eta_{\chi^{-2}}(g)$,
  defined initially for $g \in J'$,
  descends to a well-defined function
  of $g \in J$.
  More precisely,
  by lifting $g = n(x) a(y) n'(z)$
  to the matrix
  \[
    \tilde{g}
    = \begin{pmatrix}
      1 & x \\
      & 1
    \end{pmatrix}
    \begin{pmatrix}
      y &  \\
      & 1
    \end{pmatrix}
    \begin{pmatrix}
      1 &  \\
      z & 1
    \end{pmatrix}
    = \begin{pmatrix}
      y + x z & x \\
      z & 1
    \end{pmatrix}
    \in \GL_2(F)
  \]
  we see that
  \begin{equation}
    \chi(\det g)
    \eta_{\chi^{-2}} (g) =
    \chi_J(g)
  \end{equation}
  Since $W(g) = \chi(\det g) W'(g)$,
  we deduce that
  \begin{equation}
    |W(g) - \chi_J(g)| \leq 1/2
    \text{ for all } g \in J.
  \end{equation}
  Expanding out
  \begin{equation}
    \phi_1 \ast W(1) = \int_{g \in G} \phi_0(g) \chi_J^{-1}(g)
    W(g),
  \end{equation}
  it follows that
  \begin{equation}
    |\phi_1 \ast W(1)
    -
    \int_G \phi_0
    |
    \leq (1/2) \int_G \phi_0,
  \end{equation}
  hence that $|\phi_1 \ast W(1)| \geq (1/2) \int_G \phi_0 \gg 1/Q$, as required.
\end{proof}

\subsection{Upper bounds for dual weights}\label{sec:upper-bounds-dual-wts}
Let $\tilde{h}$
denote
the dual admissible weight
furnished by Theorem \ref{thm:constr-admiss-weight},
and let $\omega$ be a unitary character
of $F^\times \cong A$.
We aim to estimate $\tilde{h}(\omega)$.

We do this in the case that $F$ is non-archimedean and, as
before, $\chi$ is ramified and $\psi$ is unramified.  This case
is the relevant one for our immediate applications.  It should
be possible to carry out analogous archimedean calculations, but
we leave those for future work.

We assume moreover that \emph{$F$ is of characteristic zero},
since it will be convenient in some proofs to refer to the
exponential map without fuss.  This case is anyway the relevant
one in applications to the subconvexity problem.

\begin{definition}
  We say that the pair $(\chi,\omega)$ is
  \emph{atypical}
  if
  \begin{itemize}
  \item $q$ is odd and sufficiently large
    in terms of the degree of $F$,
  \item $-1$ is a square in $\mathfrak{o}/\mathfrak{p}$,
  \item $C(\chi) \geq q^3$,
    say
    $C(\chi) = q^{\alpha + \alpha '}$
    with $1 \leq \alpha \leq \alpha ' \leq \alpha + 1$,
  \item $C(\omega) = C(\chi)$, and
  \item the class
    $\xi \in \mathfrak{o}^\times / (1 + \mathfrak{p}^{\alpha'})$
    characterized by $\omega(\exp(a)) = \chi(\exp(\xi a))$ for
    $a \in \mathfrak{p}^{\alpha}$ satisfies
    $1 + 4 \xi^2  \in \mathfrak{p}$.
  \end{itemize}
\end{definition}
We note that $(\chi,\omega)$ is atypical
if and only if $(\chi,\omega^{-1})$ is atypical.
\begin{proposition}\label{prop:non-arch-estimates-key}
  Let $\omega$ be a unitary character of $F^\times \cong A$.
  \begin{enumerate}[(i)]
  \item If $C(\omega) = 1$,
    then $\tilde{h}(\omega) \ll 1$.
  \item If $C(\omega) > Q$,
    then $\tilde{h}(\omega) = 0$.
  \item If $1 < C(\omega) \leq Q$,
    then
    $\tilde{h}(\omega) \ll 1/Q$
    unless $(\chi,\omega)$ is atypical.
  \item Suppose that $(\chi,\omega)$ 
    is atypical.
    Let $\alpha$ and $\xi$ be as above.
    Then
    \begin{equation}
      \tilde{h}(\omega) \ll
      \frac{N_\alpha(\xi)}{Q}
      \cdot
      \begin{cases}
        q^{1/2} & \text{ if } C(\chi) = q^{2 \alpha + 1}, \\
        1& \text{ if } C(\chi) = q^{2 \alpha},
      \end{cases}
    \end{equation}
    where
    $N_\alpha(\xi) := \# \{\tau \in \mathfrak{o}/\mathfrak{p}^\alpha : \xi^2
    \tau^2 - \tau  - 1 \equiv  0 \}$.
  \end{enumerate}
\end{proposition}
Note the similarity between these estimates and
those of \cite[\S3]{2019arXiv190810346P}.

\begin{proof}
  Define $V$ in terms of $\phi$ as in \S\ref{sec-6-6}.
  Then
  \begin{equation}
    V(a(y) n'(z)) = 1_{|y| = 1}
    1_{|z| \leq 1/Q} \chi(y),
  \end{equation}
  and so
  \begin{equation}
    V^\wedge(\xi,z)
    =
    1_{|z| \leq 1/Q}
    \int_y
    1_{|y| = 1}
    \chi(y) \psi(\xi y) \, d y.
  \end{equation}
  Standard properties of Gauss sums
  imply that this last integral
  vanishes unless $|\xi| = Q$.
  Thus
  \begin{equation}
    V^\wedge(\xi,-x/\xi)
    =
    1_{|x| \leq 1}
    \int_y
    1_{|y| = 1}
    \chi(y)
    \psi(\xi y) \, d y
  \end{equation}
  and so by Fourier inversion,
  \begin{equation}\label{eqn:V-sharp-non-arch-evald-0-0}
    V^\sharp(x,y)
    = 1_{|x| \leq 1} 1_{|y| = 1} \chi(y).
  \end{equation}
  Theorem \ref{thm:constr-admiss-weight}
  now gives
  \begin{equation}
    \tilde{h}(\omega)
    =
    \int_{\substack{
        x, t \in F :  \\
        |x| \leq 1, \\
        |t - x| = |x(1-x)|
      }
    }
    \omega \left( \frac{1-t}{t} \right)
    \chi \left( \frac{x - t}{x(1-x)} \right)
    \, \frac{d t \, d x}{|t(1-t)|^{1/2} |x(1-x)|}.
  \end{equation}
  Observe that the support conditions
  imply that at least one of the following
  pair of conditions holds:
  \begin{itemize}
  \item $|x| = |t| = 1$
  \item $|1 - x| = |1 - t| = 1$
  \end{itemize}
  (For instance, if $|x| < 1$ and $|1-t| < 1$,
  then $|t| = 1 = |1-x|$,
  thus $1 = |t - x| = |x| < 1$,
  contradiction.)
  The two pairs of conditions are swapped
  by the substitution
  $x \mapsto 1-x, t \mapsto 1-t$,
  which also swaps
  $\omega$ with $\omega^{-1}$.
  Thus
  \[
    |\tilde{h}(\omega)|
    \leq
    |\rho(\omega)|
    + |\rho(\omega^{-1})|
    + |\rho'(\omega)|
  \]
  where
  \begin{equation}
    \rho(\omega)
    :=
    \int_{\substack{
        x, t \in F :  \\
        |x| \leq 1, \\
        |t - x| = |x(1-x)|, \\
        |1 - x| = |1-t| = 1
      }
    }
    \omega \left( \frac{1-t}{t} \right)
    \chi \left( \frac{x-t}{x(1-x)} \right)
    \, \frac{d t \, d x}{|t(1-t)|^{1/2} |x(1-x)|}.
  \end{equation}
  and $\rho '$ is defined similarly but with the
  additional conditions
  $|x| = |t| = 1$.
  We thereby reduce
  to bounding $\rho(\omega)$ and $\rho ' (\omega)$.
  
  We introduce the new variables
  $u := x/t, v := 1/x$,
  so that $x = 1/v, t = 1/u v$.
  We compute that $|\partial (x, t) / \partial (u, v)|
  = 1/|u^2 v^3| = |t^2 x|$,
  from which it follows
  that
  \begin{equation}
    \frac{d t \, d x}{|t(1-t)|^{1/2} |x(1-x)|}
    =
    \left\lvert \frac{u v}{u v - 1} \right\rvert^{1/2}
    \left\lvert \frac{v}{v-1} \right\rvert
    \, \frac{d u \, d v}{|u v|^{3/2}},
  \end{equation}
  thus
  \begin{equation}
    \rho(\omega)
    =
    \int_{\substack{
        u,v \in F: \\
        |u  - 1| = |u|, \\
        |v - 1| = |v|, \\
        |u v - 1| = |u v|
      }
    }
    \omega \left( u v - 1 \right)
    \chi \left( \frac{1 - 1/u}{1 - 1/v} \right)
    \, \frac{d u \, d v}{  |u v|^{3/2} }.
  \end{equation}
  (Compare with the exponential
  sums occurring in \cite{MR1779567, 2018arXiv181102452P,
    2019arXiv190810346P}.)
  The integration conditions
  force $|u|, |v| \geq 1$,
  so we may split the integral
  dyadically as
  \[
    \rho(\omega)
    = \sum_{U,V \in \{1, q, q^2, \dotsc \} }
    (U V)^{-1/2} \rho_{U, V},
  \]
  where
  \begin{equation}
    \rho_{U,V}
    :=
    \int_{\substack{
        u,v \in F: \\
        |u  - 1| = |u| = U, \\
        |v - 1| = |v| = V, \\
        |u v - 1| = U V
      }
    }
    \omega \left( u v - 1 \right)
    \chi \left( \frac{1 - 1/u}{1 - 1/v} \right)
    \, \frac{d u \, d v}{  |u v| }.
  \end{equation}
  We note that $\rho '(\omega) = \rho_{1,1}$.
  Our task thereby reduces to estimating
  the dyadic integrals $\rho_{U,V}$
  suitably.
  We record proofs of adequate estimates
  in Appendix \ref{lem:char-sum-CI}.
\end{proof}


\section{Reduction of the proofs of the main applications}
\label{sec-8}
Here we record the essential
parts of the proofs of our main applications, Theorems
\ref{thm:CI} and \ref{thm:PY1}.

More precisely,
-- that is to say, assuming
that the basic identity \eqref{eqn:basic-moment-identity} holds
and ignoring the degenerate terms $(\dotsb)$,
whose definition and treatment we postpone.

\subsection{The case of Theorem \ref{thm:CI}}\label{sec:CI}
Let $F$ be a number field
and
$\psi$
a nontrivial unitary character of $\mathbb{A}/F$.  Let $\sigma_0$ be either a cuspidal
automorphic representation of $\PGL_2(\mathbb{A})$ or a unitary
Eisenstein series.  Let $\chi$ be a quadratic character of
$\mathbb{A}^\times/F^\times$,
with analytic conductor $Q := C(\chi)$.
We regard $\sigma_0$ as fixed.
We explain first how to prove the estimate
\begin{equation}\label{eqn:desired-bound-for-L-sigma-0-chi-one-half} 
  L(\sigma_0 \otimes \chi, 1/2) \ll Q^{1/3+\eps},
\end{equation}
with polynomial
dependence of the implied constant upon $C(\sigma_0)$.
We will do so using \eqref{eqn:basic-moment-identity}, the lower
bounds established for various $h_\mathfrak{p}$, and the upper
bounds established for various $\tilde{h}_\mathfrak{p}$.

It suffices to consider the following cases:
\begin{itemize}
\item $\sigma_0$ is cuspidal.
\item $\sigma_0 = \Eis^*(\mathcal{I}(0))$,
  i.e., $\sigma_0$ is the unitary Eisenstein series
  with degenerate parameter
  for which
  $L(\sigma_0, s) = \zeta_F(s)^2$.
  In that case,
  we use the standard estimate
  \begin{equation}\label{eqn:eis-upper-bound-integral-trick}
    L(\chi,1/2)^6
    \ll
    Q^{\eps}
    \int_{t \in \mathbb{R}  : |t| \leq 1}
    t^2
    |L(\chi, 1/2 + i t)|^6
    \, d t,
  \end{equation}
  which may be deduced (in sharper forms) via the convexity
  bound for the derivatives of $t \mapsto L(\chi,1/2+it)$.
  (The exponents $2$ and $6$ are not special.)
\end{itemize}

Let $S_0$ denote the set consisting of all infinite places of
$F$ together with all finite places at which $\sigma_0$
ramifies.  Let $S_1$ denote the set of finite places not in
$S_0$ at which $\chi$ ramifies.  Set $S := S_0 \cup S_1$.
The consequence $\exp(\# S) \ll Q^{\eps}$ of the divisor bound
may be used to control products of implied constants indexed by
$S$.  We note also that, thanks to any fixed bound
$\vartheta < 1/2$ towards Ramanujan, any local $L$-factor
appearing on either side of \eqref{eqn:basic-moment-identity} is
$\exp(\O(1))$.  These observations imply that any product of
such factors taken over $\mathfrak{p} \in S$ is of size
$Q^{o(1)}$ as $Q \rightarrow \infty$.

To each $\mathfrak{p} \in S$
we attach an admissible weight $h_\mathfrak{p}$,
as follows:
\begin{itemize}
\item Let $\mathfrak{p} \in S_0$.
  By lemma
  \ref{lem:crude-lower-bound-individual},
  we may find
  $h_\mathfrak{p}$ nonnegative on unitary representations and
  bounded from below by $1$ on the local component
  $\sigma_{0, \mathfrak{p}}$ together with each of its quadratic
  twists.
  If $\sigma = \Eis^*(\mathcal{I}(0))$,
  we may assume moreover
  that $h_\mathfrak{p}$
  is $\geq 1$ on $\{\mathcal{I}_\mathfrak{p} (i t) : |t| \leq
  1\}$,
  where
  $\mathcal{I}_\mathfrak{p}(i t)$
  denotes the corresponding induced
  representation of $\PGL_2(F_\mathfrak{p})$.
  We then bound the dual
  $\tilde{h}_\mathfrak{p}$ crudely via lemma
  \ref{lem:crude-estimates}.

  The arguments just recorded do not quite suffice for our purposes:
  they lead to estimates
  \eqref{eqn:desired-bound-for-L-sigma-0-chi-one-half}
  with an \emph{unspecified}
  dependence upon $\sigma_0$.  The required polynomial
  dependence follows from the slightly finer arguments given below
  in \S\ref{sec:crude-local-estim}.
\item For $\mathfrak{p} \in S_1$,
  our assumptions imply that
  $\chi_\mathfrak{p}$ is ramified.  We construct $h_\mathfrak{p}$
  as in \S\ref{sec:constr-suit-weight} to majorize the family
  $\Sigma_{F_\mathfrak{p}}(\chi_\mathfrak{p})$, and estimate the
  dual $\tilde{h}_\mathfrak{p}$ via
  \S\ref{sec:upper-bounds-dual-wts}.
\end{itemize}

By \eqref{eqn:eis-upper-bound-integral-trick},
the lower bounds for $h_\mathfrak{p}$
proved in \S\ref{sec:lower-bounds-for-wts},
and
standard upper bounds
for $L^{(S),*}(\sigma \times \sigma, 1)$,
we have
\begin{equation}\label{eqn:bound-cube-via-integral}
  L(\sigma_0 \otimes \chi,1/2)^3
  \ll_{\sigma_0}
  Q^{1+\eps}
  \int_{\substack{
      \sigma:\text{generic}, \\
      \text{unram. outside $S$}
    }
  }
  \frac{L^{(S)}(\sigma,1/2)^3}{L^{(S),*}(\sigma \times \sigma,
    1)}
  \prod_{\mathfrak{p} \in S} h_\mathfrak{p}(\sigma_\mathfrak{p}).
\end{equation}
By the precise form of Theorem \ref{thm:constr-admiss-weight}
stated below in \S\ref{sec:summary-main-results}, the integral
on the RHS of \eqref{eqn:bound-cube-via-integral}
is equal to a degenerate term $(\dotsb)$ plus
\begin{equation}\label{eqn:dual-fourth-moment-in-pf-sketch}
  \int_{\substack{
      \omega:\text{unitary}, \\
      \text{unram. outside $S$}
    }}
  \frac{|L^{(S)}(\omega,1/2)|^4}{\zeta_F^{(S),*}(1)^2}
  \prod_{\mathfrak{p} \in S}
  \tilde{h}_\mathfrak{p}(\omega_\mathfrak{p}),
\end{equation}
at least if $\psi_\mathfrak{p}$ is unramified for all
$\mathfrak{p} \in S$; in general, the indicated relation holds
up to a scalar $\asymp 1$, by the remarks of
\S\ref{sec:vari-with-resp-1}.

By the upper bounds for $\tilde{h}_\mathfrak{p}$
noted above,
we may majorize \eqref{eqn:dual-fourth-moment-in-pf-sketch} by
\begin{equation}\label{eqn:fourth-moment-after-1}
  Q^{-1+\eps}
  \int_{\substack{
      \omega:\text{unitary}, \\
      \text{unram. outside $S$}, \\
      C(\omega_\mathfrak{p})
      \ll 1 \text{ for finite }
      \mathfrak{p} \in S_0, \\
      C(\omega_\mathfrak{p})
      \leq C(\chi_\mathfrak{p}) \text{ for }
      \mathfrak{p} \in S_1
    }
  }
  |L(\omega,1/2)|^4
  \prod_{\text{infinite } \mathfrak{p}}
  C(\omega_\mathfrak{p})^{-d}
\end{equation}
for any fixed $d$,
which we take sufficiently large.

We claim that \eqref{eqn:fourth-moment-after-1} is $\ll Q^{2 \eps}$.
To see this, it suffices to show 
for any collection $(X_\mathfrak{p})_{\mathfrak{p}}$
of parameters $X_\mathfrak{p} \geq 1$,
with $X_\mathfrak{p} = 1$ for almost all $\mathfrak{p}$ and
$X_\mathfrak{p}$
belonging to the value group of $F_\mathfrak{p}$
for all $\mathfrak{p}$,
that
\begin{equation}\label{eqn:fourth-moment-after-2}
  \int_{\substack{
      \omega:\text{unitary}, \\
      C(\omega_\mathfrak{p})
      \leq X_\mathfrak{p} \text{ for all } \mathfrak{p}
    }
  }
  |L(\omega,1/2)|^4
  \ll (\prod_\mathfrak{p} X_\mathfrak{p} )^{1+\eps}.
\end{equation}
Such an estimate is implicit in work of Han Wu \cite{MR3977317}
on the subconvexity problem for the $L$-functions
$L(\omega,1/2)$.  Wu gives a geometric proof, using unipotent
translates of Eisenstein series as in Sarnak \cite{MR780071}
and Michel--Venkatesh \cite{michel-2009}, of bounds for amplified fourth moments of
$L(\omega,1/2)$ over families as in
\eqref{eqn:fourth-moment-after-2}.  Omitting the amplifier, his
arguments give the required estimate.  In fact, such arguments
may be understood as special cases of
\eqref{eqn:basic-moment-identity} in which the cubic moment side
is particularly simple, so it should be instructive in future work
to revisit those cases from this perspective.

We verify below in \S\ref{sec:handl-degen-terms-1},
\S\ref{sec:handl-degen-terms-2} that the degenerate term
$(\dotsb)$ in \eqref{eqn:basic-moment-identity} satisfies the
estimate $(\dotsb) \ll Q^{\eps}$.
(Morally, this corresponds ``up to logarithms''
to the fact that the kernels
$\phi_\mathfrak{p}$
used to define $h_\mathfrak{p}$ for $\mathfrak{p} \in S_1$
satisfy $\phi_\mathfrak{p}(1) \ll 1$.)
Assuming this, the required bound for $L(\sigma_0 \otimes
\omega,1/2)$
follows.



\subsection{The case of Theorem \ref{thm:PY1}}\label{sec:PY1}
Recall that $\chi$ is a character of
$\mathbb{A}^\times/F^\times$ with $\chi_\infty$ trivial and
finite conductor cubefree.  We define $S = S_0 \cup S_1$, with
$S_0$ the set of archimedean places and $S_1$ the set of finite
places at which $\chi$ ramifies.  For $\mathfrak{p} \in S_0$, we
choose $h_\mathfrak{p}$ to be nonnegative on unitary
representations and bounded from below by $1$ on
$\{\mathcal{I}_p(i t) : |t| \leq 1\}$.  For
$\mathfrak{p} \in S_1$, we choose $h_\mathfrak{p}$ as in
\S\ref{sec:constr-suit-weight} to majorize
$\Sigma_{F_\mathfrak{p}}(\chi_\mathfrak{p})$.  We then argue
exactly as in \S\ref{sec:CI}.
The cubefree hypothesis
ensures that we avoid the atypical case
of Proposition \ref{prop:non-arch-estimates-key}
when we estimate $\tilde{h}_\mathfrak{p}$ for
$\mathfrak{p} \in S_1$.


\part{Preliminaries}\label{part:preliminaries}
We recall here the basic local and global theory relevant for studying $L$-functions on $\PGL_2$ via their (regularized) integral
representations, together with some basics concerning
automorphic forms and their (regularized) spectral expansions.
As general references, we mention \cite{MR2508768, michel-2009, MR0379375,
  MR1431508, MR0401654,Ja72, MR546600}.

\section{Regularized integration}
\label{sec-9-1}
We
record a notion of regularized integration, adapted from
\cite[\S4.3]{michel-2009} (see also \cite{zagier-mellin,
  MR656029}), that is adequate for our purposes.

\subsection{Regularizable functions}
Let $A$ be a
locally compact abelian group,
but not a finite group.  Let $X$ be an
$A$-space, equipped with an invariant measure $\mu$.  We say that a
measurable function $f : X \rightarrow \mathbb{C}$ is
\emph{regularizable} (with respect to $\mu$ and $A$)
if there is a bounded variation complex
Borel measure $r$ on $A$, with $\int_A r \neq 0$, so that the
convolution $r \ast f$ lies in $L^1(X,\mu)$; the
\emph{regularized integral}
is then
defined to be the ratio
\[
  \int_X^{\reg} f \, d \mu
  :=
  \frac{\int_X r \ast f}{\int_A r }.
\]
The definition is independent of the choice of $r$, and defines
an $A$-invariant functional on the space of regularizable
functions on $X$.  This notion of regularized
integration generalizes
those
cited above, but applies a
bit more generally, e.g., to the integrals used to define
Whittaker functions of principal series representations.

The definition extends with the evident modifications to the
case that $\mu$ is quasi-invariant, i.e., transforms under $A$
with respect to a character.


\subsection{Finite functions}
A \emph{finite}
function $\phi : X \rightarrow \mathbb{C}$ is one whose
$A$-translates span a finite-dimensional space $\langle A \phi  \rangle$; if that space
does not contain the trivial representation of $A$, then $\phi$
is called \emph{admissible}.
An \emph{exponent} $\chi$ of a finite function $\phi$ is a
character
of $A$ that arises as a generalized eigenvalue
for the action of $A$ on $\langle A \phi  \rangle$.

We say that $f$ is \emph{strongly regularizable}
if there is a finite cover
$X = \cup U_i$ and admissible finite functions $\varphi_i$ on
$X$ so that $f - \varphi_i$ is $\mu$-integrable on $U_i$;
in that case,
$f$ is regularizable.

\subsection{Holomorphy criteria}
Given a family of integrable functions $f_s$ depending pointwise
holomorphically upon a complex parameter $s$,
a standard criterion for their integrals
$\int f_s$ to vary holomorphically is that $|f_s| \leq h$ for
some integrable function $h$.
This criterion may be applied locally in $s$.

We have a similar criterion for regularized integrals,
which may also be applied locally:
\begin{lemma}\label{lem:reg-int-holom-var}
  Suppose given a complex manifold $M$
  and a family of measurable functions
  $f_s : X \rightarrow \mathbb{C}$
  indexed by $s \in M$
  that vary pointwise holomorphically
  (i.e., for each $x \in X$,
  the map $s \mapsto f_s(x)$ is holomorphic).
  Suppose that we may find
  \begin{itemize}
  \item a finite cover $X = \cup_{i=1}^n U_i$,
  \item admissible finite functions $\varphi_{i,s}$ on $X$
    that vary pointwise holomorphically, and
  \item integrable functions $h_i$ on $U_i$
    so that $|f_s - \varphi_{i,s}| \leq h_i$.
  \end{itemize}
  Assume that
  there exists $r \geq 0$ so that for each
  $i$ and each $s$, the dimension of the span of the
  $A$-translates of $\varphi_{i,s}$ is bounded by $r$.
  Then each $f_s$ is strongly regularizable
  and the map
  $M \ni s \mapsto \int_X^{\reg} f_s \, d \mu$ is holomorphic.
\end{lemma}
\begin{proof}
  Since $A$ is infinite, we may find distinct elements
  $a_1,\dotsc,a_{r+2} \in A$.  For each $i \in \{1..n\}$ and
  $s \in M$, consider the system of linear equations in the
  variables $c^i_{1}(s),\dotsc,c^i_{r+2}(s) \in \mathbb{C}$
  given by
  \[
    \sum_j c^i_j(s) = 1,
  \]
  \[
    \sum_j c^i_j(s) \varphi_{i,s}(a_j x) = 0
    \text{ for all } x \in X.
  \]
  The assumption on the dimension of the span of the translates
  of the $\varphi_{i,s}$ implies that there are more variables
  than independent equations.  The assumption that $\varphi_{i,s}$ is
  admissible implies then that the system is solvable for each
  $s$.  By passing to an open subset of $M$, we may find a
  specific invertible minor in the matrix describing this system
  and hence a family of solutions $c^i_j(s)$ that vary
  holomorphically with $s$.  Having chosen one such family, let
  $\kappa_s^i$ denote the measure on $A$ given by the linear
  combination of point masses $\sum_j c^i_j(s) \delta_{a_j}$.
  Take for $\kappa_s := \kappa_s^1 \ast \dotsb \ast \kappa_s^n$
  their convolution product.  Then $\int \kappa_s = 1$, while
  each convolution $\kappa_s \ast \varphi_{i,s}$ vanishes.
  Thus $\kappa_s \ast f_s$ is integrable
  and
  \begin{equation}\label{eq:integrate-regularized-via-convolution-holomorphy}
    \int_X^{\reg} f_s \, d \mu
    =
    \int_X (\kappa_s \ast f_s) \, d \mu.
  \end{equation}
  Moreover, $|\kappa_s \ast f_s|$ is bounded on $U_i$,
  locally uniformly in $s$,
  by a linear combination of $h_i$ and its translates.
  Thus the standard criterion implies
  that the RHS of
  \eqref{eq:integrate-regularized-via-convolution-holomorphy} is
  holomorphic in $s$.
\end{proof}

\subsection{Main examples}\label{sec:regularization-specialized-to-A}
We will apply these notions primarily when either
\begin{itemize}
\item $A$ is the multiplicative group of a local field $F$ 
  and $X = A$, or
\item $A$ is the multiplicative group $\mathbb{A}^\times$ of the
  adele ring $\mathbb{A}$
  of a global field $F$
  and $X = \mathbb{A}^\times / F^\times$,
\end{itemize}
equipped with suitable Haar measures.
In either case, a convenient cover is given by $X = U_{\infty} \cup U_0$
with
$U_\infty  := \{x \in X : |x| \geq 1\}$
and
$U_0 := \{x \in X : |x| < 1\}$.
The finite
functions are the linear combinations of functions of the form
$y \mapsto \chi(y) \log^{m-1} |y|$ for some character $\chi$ of
$X$ and some positive integer $m$.
If $f$ is strongly regularizable
(with respect to this cover),
then its regularized integral may be defined
as above using convolution,
or (as noted in \cite[\S4.3]{michel-2009}) in other equivalent ways:
\begin{itemize}
\item By truncation:
  We may write
  $\int_{x \in X : 1/T \leq |x| \leq T}
  f(x) \, d \mu(x)$
  as the sum $g(T) + h(T)$,
  where $g$ is an admissible finite
  function on the value group  $\{ |x| : x \in X \}$
  and $h(T)$ has a limit as $T \rightarrow \infty$;
  then $\int_X^{\reg} f = \lim_{T \rightarrow \infty} h(T)$.
\item By meromorphic continuation:
  The integrals
  $\int_{x \in U_\infty} f(x) |x|^{s} \, d \mu(x)$ and
  $\int_{x \in U_0} f(x) |x|^s \, d \mu(x)$ converge absolutely for $\Re(s)$
  sufficiently negative and positive, respectively.
  They extend meromorphically to
  functions $F_\infty, F_0$ on the complex plane
  for which
  $\int_X^{\reg} f = F_\infty(0) + F_0(0)$.

  Alternatively, we may find an admissible finite function $\phi_\infty$
  so that the integral
  $F(s) := \int_{x \in X} (f - \phi_\infty)(x) |x|^s \, d
    \mu(x)$
  converges
  absolutely for $\Re(s)$ sufficiently large.
  Then $F$
  continues meromorphically to the complex plane, and
  \begin{equation}\label{eqn:regularized-integral-via-mellin-analysis}
    \int_X^{\reg} f = F(0).
  \end{equation}
\end{itemize}

\section{Norms on representations}\label{sec:norms-repr}
Let $F$ be a local field.  Let $G$ be a reductive group over
$F$.
(The groups $\GL_1(F)$, $\PGL_2(F)$ and products thereof
are the relevant ones for this paper.)  
Our aim in this section is to define a system of Sobolev
norms $\mathcal{S}_d$ ($d \in \mathbb{R}$) on certain
representations $\sigma$ of $G$.  Michel--Venkatesh
\cite[\S2]{michel-2009} constructed a suitable system when
the representation $\sigma$ is unitary.  It will be important for
us to work with \emph{non-unitary} representations (possibly far
from the ``tempered axis''), so we give a more general
construction sufficient for our aims.

We will often use these norms to estimate linear functionals
$\ell : \sigma \rightarrow \mathbb{C}$ by asserting that
$\ell(v) \ll \mathcal{S}_d(v)$ for some fixed $d$.  In the
archimedean case, the existence of such an estimate is
equivalent to the continuity of $\ell$.  In the non-archimedean
case (where our conventions imply that any linear functional is
continuous), such estimates should be understood informally as
asserting that $\ell(v)$ is bounded polynomially with respect to
the ramification of $v$.
The purpose of these norms is to give a
convenient way to formulate such estimates \emph{uniformly} as
the representation $\sigma$ and the underlying local field $F$
vary.  We use them in this paper primarily to verify that
certain estimates (e.g., \eqref{eq:weyl-bound-sigma-chi}) depend
polynomially upon auxiliary parameters.

\subsection{Setting}\label{sec:norms-reps-setting}
We assume that $G$ comes equipped with a faithful linear
representation $G \hookrightarrow \SL_r(F)$ and a maximal
compact subgroup $K$.  In the non-archimedean case, we assume
that $K$ is special and contains $G \cap \GL_r(\mathfrak{o})$.
We assume given
a Haar measure $d k$ on $K$ of volume $\asymp 1$
(e.g., the probability Haar).

Let $P = M U$ be a parabolic subgroup of $G$, with Levi $M$ and
unipotent radical $U$.  Let $\chi_0$ be an irreducible
unitary representation
of $M$, with inner product $\langle , \rangle$ and norm $\|.\|$.
We may then define the normalized induction
$\sigma_0 := \Ind_P^G(\chi_0)$, consisting of smooth
$f : G \rightarrow \chi_0$ satisfying
$f(n m g) = \delta_P^{1/2}(m) \chi_0(m) f(g)$.  It is a unitary
representation with respect to the inner product
$\langle f_1, f_2 \rangle := \int_{k \in K} \langle f_1(k),
f_2(k) \rangle \, d k$.  By the decomposition $G = P K$, we see that restriction to
$K$ identifies the restricted representation $\sigma_0|_K$ with
$\Ind_{M \cap K}^K (\chi_0|_{M \cap K})$.

Let $\theta : M \rightarrow \mathbb{R}^\times_+$ be a
positive-valued character of $M$.  We may then form the twisted
representation $\chi := \chi_0 \otimes \theta$ of $M$ and its
induction $\sigma := \Ind_P^G(\chi)$.
We regard $\chi$ as the same underlying space
as $\chi_0$, but with a modified action.  Since $\theta$ is
positive-valued and $M \cap K$ is compact,
the representations $\chi$ and $\chi_0$ have
the same restrictions to $M \cap K$, hence their inductions $\sigma$
and $\sigma_0$ have the same restrictions to $K$.
We equip $\sigma$ with the inner product $\langle , \rangle$
and norm $\|.\|$ transferred from $\sigma_0$.
The inner product on $\sigma$ is thus $K$-invariant,
but not in general $G$-invariant.

We note that if $G$ is a quasi-split classical group (e.g., if $G =
\GL_r(F)$),
then
each generic irreducible representation $\sigma$
is known to 
arise (essentially uniquely) from the above construction
with $\chi_0$ tempered and
$\theta$ strictly dominant
(see the second paragraph of
\cite{MR2046512}, or \cite{MR507800,GKdim}).

\subsection{Construction of the Sobolev norms}
\label{sec:constr-sobol-norms}
We now define Sobolev norms $\mathcal{S}_d$ on $\sigma$.  In
summary, we first apply (a slight modification of) the
construction of \cite[\S2]{michel-2009} to obtain such norms on
$\sigma_0$.  We then transfer those norms to $\sigma$ via the
$K$-equivariant identification $\sigma \cong \sigma_0$.

We formulate the construction in
terms of $K$-types.  Let $K^\wedge$ denote the set of
isomorphism classes of
irreducible representations of $K$.  For $\nu \in K^\wedge$, we
denote by $\sigma^{\nu}$ the $\nu$-isotypic component of
$\sigma$.
The notation applies also to $\sigma_0$,
and we have $\sigma_0^{\nu} = \sigma^\nu$ as representations
of $K$.

For each $\nu \in K^\wedge$
such that $\sigma^{\nu} \neq \{0\}$,
we define a scalar $N_{\sigma \nu} \geq 1$
as follows.
In the non-archimedean case, we define for $n \geq 0$ the
principal congruence subgroup
\[
  K[n] := G \cap \{g \in \GL_r(\mathfrak{o}) : g \equiv 1
  (\mathfrak{p}^n)
  \}
\]
of $K$.  We set
\[
  N_{\sigma \nu} := q^n
\]
if $n \geq 0$ is the
smallest nonnegative integer such that $K[n]$ acts trivially
on $\nu$.
In the archimedean case, we fix an orthonormal basis
$\{x_i\} \cup \{y_j\}$ for $\mathfrak{g} := \Lie(G)$ with
respect to the trace pairing derived from the given linear
embedding, where the $x_i \in \mathfrak{k}$ and
$y_j \in \mathfrak{p}$ belong to summands of the Cartan
decomposition
$\mathfrak{g} = \mathfrak{k} \oplus \mathfrak{p}$.  The
corresponding Casimir elements for $G$ and $K$ are then given
by $\mathcal{C}_G = -\sum_i x_i^2 + \sum_j y_j^2$ and
$\mathcal{C}_K = -\sum_i x_i^2$.  The subspaces $\sigma^{\nu}$
and $\sigma_0^{\nu}$ are eigenspaces for the actions of these
Casimir elements.  We write $c_{\sigma}, c_{\sigma_0}$ for the
corresponding eigenvalues of $\mathcal{C}_G$ and $c_\nu$ for
that of $\mathcal{C}_K$.  Set
\begin{equation}\label{eqn:Delta-G-defn}
  \Delta_G := - \sum_i x_i^2 - \sum_j y_j^2 = - \mathcal{C}_G +
2 \mathcal{C}_K.
\end{equation}
Since $\sigma_0$ is unitary, the element
$\Delta_G$ acts on $\sigma_0$ by a positive-definite operator,
as do its summands $- \sum_i x_i^2$ and $- \sum_j y_j^2$.  It
follows that the quantities $c_\nu, c_\nu - c_{\sigma_0}$ and
$- c_{\sigma_0} + 2 c_\nu$ are nonnegative.  We set
\begin{equation}
  N_{\sigma \nu} :=
  (1 - c_{\sigma_0} + 2 c_{\nu})^{1/2} \in \mathbb{R}_{\geq 1}.
\end{equation}
In other words,
$N_{\sigma \nu}$ is the eigenvalue
for $(1 + \Delta_G)^{1/2}$ on $\sigma_0^{\nu}$.
In either case, we define for $d \in \mathbb{R}$ the Sobolev
norm $\mathcal{S}_d$ on $\sigma$ by the rule
\begin{equation}\label{eq:defn-S-d-of-v-squared}
  \mathcal{S}_d(v)^2 :=
  \sum_{\nu}
  N_{\sigma \nu}^{2 d}
  \|v_\nu \|^2,
\end{equation}
where $v = \sum_{\nu} v_\nu$ is the decomposition of the
(smooth) vector $v \in \sigma$ into $K$-isotypic components.  In
the non-archimedean case, this is really a finite sum, while in
the archimedean case the sum converges in the natural
Fr{\'e}chet topology on $\sigma$.

\begin{remark}
  The norms $\mathcal{S}_d$ coincide with those defined in
\cite[\S2]{michel-2009} when $\sigma = \sigma_0$, except for a
slight difference in normalization
(our $\mathcal{S}_d$ is the
``$\mathcal{S}_{d/2}$'' of \cite[\S2]{michel-2009}
when $F$ is archimedean).
In the non-archimedean case, the construction of
\cite[\S2]{michel-2009} refers only to the restriction of
$\sigma$ to $K$, hence applies directly even when $\sigma$ is
non-unitary. The subtlety in the archimedean case responsible
for the roundabout definition given above is that the operator
$\sigma(\Delta_G)$ need not be positive-definite, or even
self-adjoint.  For that definition to be useful in practice, we
still need to control the operator $\sigma(\Delta_G)$ and the
norms $\mathcal{S}_d$ in terms of one another, or in other
words, to compare the eigenvalues of $\sigma(\Delta_G)$ and
$\sigma_0(\Delta_G)$ on their common eigenspaces
$\sigma_0^\nu \cong \sigma_0^{\nu}$.  The required comparison is
given below in lemma \ref{lem:comp-betw-lapl}.
\end{remark}

\begin{remark}
When $\sigma$ is a non-tempered unitary representation, the
norms defined here typically differ from those defined in
\cite[\S2]{michel-2009}.
\end{remark}

\begin{remark}
We note that in the archimedean case, the seminorms
$\mathcal{S}_d$ define the natural Fr{\'e}chet topology on
$\sigma$.  Thus the continuous functionals
$\ell : \sigma \rightarrow \mathbb{C}$ are precisely those for
which $|\ell(v)| \leq C \mathcal{S}_d$ for some $C$ and $d$.
\end{remark}

\subsection{Norms on Schwartz spaces}\label{sec:norms-schw-spac}
Here we
define Sobolev norms $\mathcal{S}_d$ ($d \in \mathbb{R}$) on
the Schwartz spaces $\mathcal{S}(F^r)$ ($r \geq 1$).  These may
be obtained by adapting the above construction to standard
representations of Heisenberg groups, but it seems simpler for
our purposes to give the definition more directly.

We assume given a nontrivial unitary character $\psi$ of $F$,
hence a Haar measure on $F$ and on $F^r$.

In the archimedean case, we choose a basis for
$F$ over $\mathbb{R}$ to identify $F^r$ with
$\mathbb{R}^{[F:\mathbb{R}] r}$.
We write $x \in F^r$
in coordinates as $x = (x_1,\dotsc,x_{[F:\mathbb{Q}]})$
For multi-indices $\alpha, \beta \in \mathbb{Z}_{\geq
  0}^{[F:\mathbb{R}]}$,
we denote by $M^\alpha$ the function
$\mathcal{S}(F^r) \ni x \mapsto \prod_{j=1}^{[F:\mathbb{R}]}
x_j^{\alpha_j}$
and by $\partial^\beta$ the differential
operator
$\prod_{j=1}^{[F:\mathbb{R}]} \frac{\partial }{\partial x_j}$.
For $\phi \in \mathcal{S}(F^r)$,
we then define $\mathcal{S}_d(\phi)$
to be the sum, over all multi-indices
$\alpha, \beta \in \mathbb{Z}_{\geq 0}^{[F:\mathbb{R}]}$ with
$\sum \alpha_i + \sum \beta_j \leq d$, of
$\|M^\alpha \partial^{\beta} \phi \|_{L^2(F^r)}$.

We turn to the non-archimedean case.  Choose an unramified
character $\psi_0$ of $F$.  For $x,y \in \mathfrak{o}^{r}$ and
$\phi \in \mathcal{S}(F^r)$, define
$(x,y) \cdot \phi \in \mathcal{S}(F^r)$ by the formula
$(x,y) \cdot f(z) = \phi(z + x) \psi_0(y z)$.  This defines an
action of the compact group $\mathfrak{o}^{2 r}$ on
$\mathcal{S}(F^r)$.  Each $\phi \in \mathcal{S}(F^r)$ may be
decomposed accordingly as a finite sum $\phi = \sum \phi_\nu$,
where $\nu$ runs over the characters of $\mathfrak{o}^{2 r}$.
We define $N_\nu := q^n$ if $n \geq 0$ is the smallest
nonnegative integer for which $\nu$ has trivial restriction to
$(\mathfrak{p}^n)^{\oplus 2 r}$, and set
$\mathcal{S}_d(\phi)^2 := \sum_{\nu} N_{\nu}^{2 d}
\|\phi_\nu\|^2_{L^2(F^r)}$.  Explicitly, we may write each
$\phi \in \mathcal{S}(F^r)$ uniquely as a linear combination
$\phi = \sum_{x,\xi} c(x,\xi) \phi_{x,\xi}$ of the functions
$\phi_{x,\xi} \in \mathcal{S}(F^r)$ defined for
$x,\xi \in F^r / \mathfrak{o}^r$ by
$\phi_{x,\xi}(y) := 1_{x + \mathfrak{o}}(y) \psi_0(\xi y)$, and
we have
$\mathcal{S}_d(\phi)^2 = \sum_{x,\xi} \max(1,|x_1|,\dotsc,|\xi_r|)^2
|c(x,\xi)|^2$.

Many of the results stated below for the norms $\mathcal{S}_d$
attached above to representations of reductive groups hold with
minor modifications for the norms $\mathcal{S}_d$ defined here
on Schwartz spaces.

\subsection{Uniformity}\label{sec:uniformity}
For the estimates stated below,
we assume given a \emph{fixed} number field $\mathbf{F}$, a
\emph{fixed} reductive group $\mathbf{G}$ over $\mathbf{F}$ with
a \emph{fixed} linear embedding
$\mathbf{G} \hookrightarrow \mathbf{G} \mathbf{L}_r$ from which
the local field $F$, the group $G$ and its linear embedding
$G \hookrightarrow \GL_r(F)$ arise as local components at some
place $\mathfrak{p}$ of $\mathbf{F}$.
Implied constants
are thus allowed to depend upon $\mathbf{F}$
and $\mathbf{G}$,
but not upon $F$ or $G$.

In the case that $F$ is archimedean, we assume that
$\theta$ belongs to a \emph{fixed} bounded collection $\Theta$
of positive-valued characters of $M$.  For instance, if
$G = \PGL_2(F)$ and $M = A \cong F^\times$, then
$\theta = |.|^c$ for some $c \in \mathbb{R}$, and we may define
$\Theta$ by requiring that $|c| \leq \vartheta$ for some fixed
$\vartheta > 0$.

As in \cite[\S2]{michel-2009}, we adopt the ``implied index''
convention that $\mathcal{S}(v)$ denotes a Sobolev norm of the form
$\mathcal{S}_d(v)$ for some fixed index $d$.

\subsection{Comparison between Laplace eigenvalues}
\label{sec:comp-betw-lapl}

\begin{lemma}\label{lem:comp-betw-lapl}
  Suppose that $F$ is archimedean.
  Let $\nu \in K^\wedge$ be such that
  the isotypic component $\sigma^{\nu}$
  (equivalently, $\sigma_0^{\nu}$) is nontrivial.
  Let $\delta = - c_{\sigma} + 2 c_\nu$
  and $\delta_0 =  - c_{\sigma_0} + 2 c_\nu$
  denote the respective eigenvalues
  for $\Delta_G$.
  Then
  \begin{equation}
    \delta
    =
    \delta_0 + \O(\delta_0^{1/2} + 1).
  \end{equation}
  In particular, if $C \geq 1$ is fixed
  but large enough
  (in terms of $\Theta$),
  then
  \begin{equation}\label{eqn:C-plus-delta-comparison}
    C + \delta \asymp C + \delta_0.
  \end{equation}
\end{lemma}
\begin{proof}
  We denote by $\langle , \rangle$ the norm on
  $\Ind_{M \cap K}^{K}(\chi_0) = \sigma_0|_K \cong \sigma|_K$
  defined as in \S\ref{sec:norms-reps-setting} by
  $\langle f_1, f_2 \rangle = \int_{k \in K} \langle f_1(k), f_2(k)
  \rangle\, d k$.  Let $f$ be a unit vector in $\sigma_0^{\nu}$,
  so that $\delta_0 = \langle \Delta f, f \rangle$ with $\Delta := \Delta_G$.  We may
  extend $\theta$ from $M$ to $P$ via pullback and then to $G$
  via the decomposition $G = P K$.  Since $\theta|_K \equiv 1$,
  the product $\theta f$ is a unit vector in $\sigma^{\nu}$, and
  so $\delta = \langle \Delta (\theta f), \theta f \rangle$.  On
  the other hand, we have
  $\Delta(\theta f) = \theta (\Delta f) + u_1 + u_0$, where
  \begin{equation}
    u_1 :=
    -2\sum_{x \in \mathcal{B}(\mathfrak{g})}
    (x \theta)  (x f),
    \quad
    u_0
    :=
    (\Delta \theta ) f.
  \end{equation}
  Since $\langle \theta (\Delta f), f \rangle = \langle \Delta
  f, f \rangle = \delta_0$,
  it follows that
  $\delta = \delta_0 +  \langle u_1 + u_0, f \rangle$,
  hence by Cauchy--Schwarz that
  $|\delta - \delta_0| \leq \|u_1\| + \|u_0\|$.
  Since  $\theta$ lies in the fixed bounded collection $\Theta$,
  the $L^\infty(K)$-norms of $\theta$ and $x \theta$
  are $\O(1)$.
  Thus $\|u_0\| \ll 1$.
  It follows from (e.g.)
  \cite[\S3.5]{nelson-venkatesh-1}
  that
  \begin{equation}\label{eq:estimate-for-first-order-diff-op}
    \|x f\| \ll 1 + \delta_0^{1/2},
  \end{equation}
  thus $\|u_1\| \ll 1 + \delta_0^{1/2}$.
  The required estimate
  follows.
\end{proof}

In the non-archimedean case, we \emph{define} the operators
$\Delta^d := \Delta_G^d$ ($d \in \mathbb{R}$) by
\begin{equation}\label{eqn:defn-Delta-nonarch}
  \Delta^d \sum v_\nu = \sum N_{\sigma \nu}^{d/2} v_\nu,
\end{equation}
so that
$\mathcal{S}_d(v)^2 = \langle \Delta^d v, v \rangle$.
In either case, we see that if $C$ is large enough but fixed,
then
\begin{equation}\label{eqn:sobolev-norms-via-C-plus-delta}
  \mathcal{S}_d(v)^2 \asymp
  \langle (C + \Delta)^d v, v \rangle.
\end{equation}

\subsection{Uniform trace class property}\label{sec:uniform-trace-class}
\begin{lemma}\label{lem:uniform-bounds-N-sigma-nu}
  There is a fixed $d_0 \geq 0$
  so that
  for all all irreducible
  representations $\sigma$ of $G$,
  \begin{equation}\label{eq:sum-N-sigma-nu-neg-d-0}
    \sum_{\nu}
    \dim(\sigma^\nu)
    N_{\sigma \nu}^{-d_0}
    \ll 1.
  \end{equation}
\end{lemma}
\begin{proof}
  This estimate is the same for $\sigma$ as for $\sigma_0$, so
  the arguments of \cite[\S2.6.3]{michel-2009} apply directly.
\end{proof}

\subsection{Sobolev conductor}\label{sec:sobolev-conductor}
We define, as in \cite[\S2.6.5]{michel-2009},
the ``Sobolev conductor''
\begin{equation}
  C'(\sigma) := \min \{N_{\sigma \nu} :
  \nu \in K^\wedge \text{ with } \sigma^\nu \neq \{0\}\}
\end{equation}
The following result is (the local analogue of) \cite[Lem 2.6.6]{michel-2009}.
\begin{lemma}\label{lem:compare-conductors}
  Suppose that $G = \GL_r(F)$ and that $\sigma$ is irreducible
  and generic.  Then
  \begin{equation}\label{eqn:log-conductor-comparison}
    \log(1 + C(\sigma)) \asymp \log(1 + C'(\sigma)).
  \end{equation}
\end{lemma}
Note that the assumptions of \S\ref{sec:uniformity} imply that
$\sigma$ is $\vartheta$-tempered for some fixed $\vartheta > 0$.

The proof of lemma \ref{lem:compare-conductors} is attributed in
\emph{loc. cit.}, to a personal communication from W.T. Gan, but
no proof is recorded.
The case $r=1$ is straightforward
(see Lemma
\ref{lem:gl1-character-analytic-conductor-invariance}
and \cite[\S1.1]{JN19a} for related discussion). 
We explain
the case $r = 2$ relevant to this paper (and also to
\cite{michel-2009}).
The archimedean
case can be verified
by comparing the Casimir eigenvalue and analytic
conductor of each irreducible representation.
In the non-archimedean
case, the main input is the following:
\begin{lemma}
  Assume $F$ non-archimedean.
  Let $\sigma$ be an irreducible representation of $\GL_2(F)$
  and $\omega$ a character of $F^\times$.
  Then
  \begin{equation}\label{eqn:conductor-inequality-twists}
    C(\sigma \otimes \omega)
    \leq
    \max
    ( C(\sigma) C(\omega), C(\omega)^2 )
  \end{equation}
\end{lemma}
\begin{proof}
  This is a special case of \cite[Thm 1]{MR1462836};
  see also \cite[Prop
  3.4]{xoMR0476703} and \cite{MR1606410}
  for sharper bounds in the supercuspidal case.
\end{proof}
We now deduce \eqref{eqn:log-conductor-comparison} from \eqref{eqn:conductor-inequality-twists}.
Set $K := \GL_2(\mathfrak{o})$ and
\[
  K_1[n] :=
  \left\{\begin{pmatrix}
      a & b \\
      c & d
    \end{pmatrix} \in K : c, d- 1 \in \mathfrak{p}^n \right\},
  \quad K[n] := \left\{ k \in K : k \equiv 1 (\mathfrak{p}^n)
  \right\}.
\]
If $\sigma$ contains a nonzero vector invariant by $K_1[n]$,
then the same vector is invariant by $K[n]$, and so
$C'(\sigma) \leq C(\sigma)$.  Conversely, write
$C'(\sigma) = q^n$, so that $\sigma$ contains a nonzero vector
invariant by $K[n]$.  The translate of that vector by a suitable
diagonal element of $G$
is then invariant by the following conjugate of
$K[n]$:
\[
  K \cap
  (1 + \begin{pmatrix}
    \mathfrak{p}^n & \mathfrak{o}  \\
    \mathfrak{p}^{2 n} & \mathfrak{p}^n
  \end{pmatrix}).
\]
Let $\omega_\sigma$ denote the central character of $\sigma$.
By Fourier decomposition with respect to the diagonal subgroup
of $K$, we may find a character $\chi$ of $F^\times$ with
$C(\chi) \leq q^n$ and a nonzero vector in $\sigma$ transforming
under $J_{1}[2n]$ by the character $\left(
  \begin{smallmatrix}
    a&b\\
    c&d
  \end{smallmatrix}
\right) \mapsto \omega_{\sigma}(d) \chi(a/d)$.  Then the twisted
representation $\sigma \otimes \chi^{-1}$ contains a nonzero
vector invariant by $J_1[2 n]$.  By newvector theory
\cite{MR0337789}, it follows that
$C(\sigma \otimes \chi^{-1}) \leq q^{2 n}$.
By \eqref{eqn:conductor-inequality-twists}
we deduce the sufficient estimate
\begin{equation}
  C(\sigma) = C( (\sigma \otimes \chi^{-1}) \otimes \chi)
  \leq \max(q^{2 n} C(\chi), C(\chi)^2)
  \leq q^{3 n} = C'(\sigma)^3.
\end{equation}

\subsection{Reduction to isotypic vectors}\label{sec:reduct-isotyp-vect}
Let
$\sigma^{\flat}$ denote the subspace of $K$-finite vectors in
$\sigma$.
\begin{lemma}
  Let
  $\ell$ be a linear functional on $\sigma^{\flat}$ such that the estimate
  $\ell(v) \ll \mathcal{S}_d(v)$ holds for some fixed $d$ and all
  $K$-isotypic vectors $v \in \sigma$.  Then $\ell$ extends to a
  continuous linear functional on $\sigma$ for which the estimate
  $\ell(v) \ll \mathcal{S}_{d'}(v)$ holds for some fixed $d'$ and
  all $v \in \sigma$.
\end{lemma}
\begin{proof}
  For $K$-finite $v$,
  our hypothesis gives
  \[
    |\ell(v)|
    \leq \sum_{\nu} |\ell(v_\nu)|
    \ll \sum_{\nu} \mathcal{S}_d(v_\nu).
  \]
  On the other hand, for any $d_0 \geq 0$, we have
  \[
    \mathcal{S}_d(v_\nu)
    =
    \mathcal{S}_{d+d_0}(v_\nu)
    N_{\sigma \nu}^{-d_0}
    \leq
    \mathcal{S}_{d+d_0}(v)
    N_{\sigma \nu}^{-d_0}.
  \]
  Choosing $d_0$ as in lemma
  \ref{lem:uniform-bounds-N-sigma-nu},
  the required estimate follows
  (with $d' := d + d_0$) from
  \eqref{eq:sum-N-sigma-nu-neg-d-0}.
\end{proof}

\subsection{Integration by parts}\label{sec:integration-parts}
Suppose given a pair of groups $G$ and $H$ as above, with
$H$
contained in $G$.

Let $\pi$ and $\sigma$ be representations of $G$ and $H$,
respectively, of the sort considered in
\S\ref{sec:norms-reps-setting}.  We assume that the conditions
of \S\ref{sec:uniformity} are satisfied for both $(G,\pi)$ and
$(H,\sigma)$, with $\mathbf{H}$ an algebraic subgroup of
$\mathbf{G}$.

We will refer subsequently to the following lemma
as \emph{integration
  by parts with respect to $\sigma$}.
\begin{lemma}
  Suppose given a linear functional
  $\ell : \pi \otimes \sigma \rightarrow \mathbb{C}$
  that is invariant by the diagonal action of $H$
  and satisfies
  \begin{equation}\label{eq:int-by-parts-hyp}
    \ell(u \otimes v)
    \ll \mathcal{S}_{d_0}(u) \mathcal{S}_{d_0}(v)
  \end{equation}
  for some fixed $d_0$.
  Then for fixed $d, d'$ with $d'$ large in terms of $d$,
  \begin{equation}\label{eq:int-by-parts-conc-better}
    \ell(u \otimes v)
    \ll \mathcal{S}_{d'}(u) \mathcal{S}_{-d}(v) C'(\sigma)^{-d}.
  \end{equation}
\end{lemma}
\begin{proof}
  By the definition of $C'(\sigma)$,
  we have $\mathcal{S}_{-2 d}(v) \ll \mathcal{S}_{-d}(v)
  C'(\sigma)^{-d}$,
  so by replacing $d$ with $2 d$
  it will suffice to show that
  \begin{equation}\label{eq:int-by-parts-conc}
    \ell(u \otimes v)
    \ll \mathcal{S}_{d'}(u) \mathcal{S}_{-d}(v).
  \end{equation}
  We choose $C \geq 1$ fixed but large enough.
  For any $b \in \mathbb{Z}_{\geq 0}$,
  the $H$-invariance of $\ell$
  and the initial estimate
  \eqref{eq:int-by-parts-hyp}
  give
  \begin{align*}
    \ell(u \otimes v)
    &= \ell( (C + \Delta_H)^b u \otimes (C + \Delta_H)^{-b} v)
    \\
    &\ll
      \mathcal{S}_{d_0}( (C + \Delta_H )^b u)
      \mathcal{S}_{d_0}( ( C + \Delta_H)^{-b} v).
  \end{align*}
  By \eqref{eqn:sobolev-norms-via-C-plus-delta},
  we have
  \begin{equation}
    \mathcal{S}_{d_0}( ( C + \Delta_H)^{-b} v)
    \asymp
    \mathcal{S}_{d_0 - 2 b}(v).
  \end{equation}
  On the other hand,
  we claim that
  \begin{equation}\label{eq:estimate-S-d0-C-plus-delta-H-b-u}
    \mathcal{S}_{d_0}( ( C + \Delta_H)^{b} u)
    \ll
    \mathcal{S}_{d_0 + 2 b}(v).
  \end{equation}
  In the archimedean case, this last estimate follows from
  \eqref{eq:estimate-for-first-order-diff-op}.  In the
  non-archimedean case, denote by $K_G$ and $K_H$ the
  corresponding maximal compact subgroups.  Then
  $K_H[m + \O(1)] \subseteq K_G[m]$ in general, and
  $K_H[m] \subseteq K_G[m]$ if $q$ is sufficiently large.  Thus
  the eigenvalues of $\Delta_H$ on the $K_G$-isotypic subspaces
  $\pi^\nu$ of $\pi$ are $\O(N_{\pi \nu}^2)$.  The claim
  \eqref{eq:estimate-S-d0-C-plus-delta-H-b-u} follows.  We
  conclude that \eqref{eq:int-by-parts-conc} holds with
  $d' := d_0 + 2 b$ for any $b \geq (d + d_0)/2$.
\end{proof}

\subsection{Reduction to pure
  tensors}\label{sec:reduct-pure-tens}
Let $\sigma_1$ and $\sigma_2$ be representations of a
pair of groups $G_1$ and $G_2$ as above.  Let
$\sigma_i^{\flat} \subseteq \sigma_i$ denote the subspace of
$K_i$-finite vectors.
\begin{lemma}
  Suppose given a bilinear form $\ell$ on
  $\sigma_1^{\flat} \times \sigma_2^{\flat}$ satisfying the
  estimate
  \begin{equation}\label{eqn:hypothesis-ell-v1-otimes-v2-bound}
    \ell(v_1, v_2)
    \ll
    \mathcal{S}_d(v_1) \mathcal{S}_d(v_2)
  \end{equation}
  for some $d$.
  Then $\ell$ extends to a continuous linear functional
  on the completed tensor product $\sigma_1 \otimes \sigma_2$
  satisfying the estimate
  \begin{equation}\label{eqn:ell-v-bounded-by-S-d-prime-v}
    \ell(v) \ll \mathcal{S}_{d'}(v).
  \end{equation}
\end{lemma}
\begin{proof}
  By \S\ref{sec:reduct-isotyp-vect},
  we may assume
  that
  $v$ is $K_1 \times K_2$-isotypic,
  say transforming under
  $\nu_1 \otimes \nu_2$
  for some $\nu_j \in K_j$.
  We
  may decompose $v$ with respect
  to an orthonormal basis of its isotypic component as $v = \sum_i v_i$,
  where the number of summands is
  $\dim(\sigma_1^{\nu_1}) \dim(\sigma_2^{\nu_2})$
  and $\mathcal{S}_d(v)^2 = \sum_i \mathcal{S}_d(v_i)^2$.
  By \eqref{eqn:hypothesis-ell-v1-otimes-v2-bound}
  and Cauchy--Schwarz,
  we obtain
  \[
    \ell(v)
    \ll
    \sum_i \mathcal{S}_d(v_i)
    \leq
    \sqrt{\dim(\sigma_1^{\nu_1}) \dim(\sigma_2^{\nu_2})}
    \mathcal{S}_d(v).
  \]
  Choosing $d_0$ large enough that
  $\dim(\sigma_j^{\nu_j}) \leq N_{\sigma_j \nu_j}^{2 d_0}$ and
  taking $d' := d + d_0$ then gives
  \eqref{eqn:ell-v-bounded-by-S-d-prime-v}.
\end{proof}

In particular, a functional $\ell$ on $\sigma_1 \otimes \sigma_2$
satisfies
$\ell(v_1 \otimes v_2) \ll \mathcal{S}(v_1) \mathcal{S}(v_2)$
if and only if it satisfies $\ell(v) \ll \mathcal{S}(v)$.

\subsection{Adelic setting}\label{sec:repr-adel-groups}
Let $F$ be a number field and
$\mathbf{G} \hookrightarrow \mathbf{G} \mathbf{L} _r$  a linear
reductive $F$-group, as in \S\ref{sec:uniformity}.  Take now for
$G$ the set of points of $\mathbf{G}$ over the adele ring
$\mathbb{A}$ of $\mathbf{F}$.  Let $\sigma$ be a representation
of $G$ arising as a restricted tensor product
$\sigma = \otimes_{\mathfrak{p}} \sigma_\mathfrak{p}$, where
each $\sigma_\mathfrak{p}$ arises as in
\S\ref{sec:norms-reps-setting}.  Let
$K = \prod_{\mathfrak{p}} K_\mathfrak{p}$ be a maximal compact
subgroup of $G$, with $K_\mathfrak{p}$ satisfying the
assumptions of \S\ref{sec:norms-reps-setting}.  Each irreducible
representation $\nu$ of $K$ is a tensor product
$\otimes \nu_\mathfrak{p}$ of irreducible representations
$\nu_\mathfrak{p}$ of $K_\mathfrak{p}$.  We define
$N_{\sigma \nu}$ to be the product of the local quantities
$N_{\sigma_\mathfrak{p} \nu_\mathfrak{p}}$ defined in
\S\ref{sec:constr-sobol-norms}, and then define Sobolev norms
$\mathcal{S}_d$ on $\sigma$ by the same formula
\eqref{eq:defn-S-d-of-v-squared} as in the local case.  We
define the operator $\Delta_G$ on $\sigma$ to be the tensor
product $\otimes \Delta_{G_\mathfrak{p}}$ of the operators
$\Delta_{G_\mathfrak{p}}$ on $\sigma_\mathfrak{p}$ defined in
\eqref{eqn:Delta-G-defn} and \eqref{eqn:defn-Delta-nonarch}.

We may similarly adapt the definitions of
\S\ref{sec:norms-schw-spac}
to adelic Schwartz spaces.

As in \cite[\S2]{michel-2009}, many of the local results given
above extend readily to the adelic setting.  Assume that for
each archimedean place $\mathfrak{p}$, $\sigma_\mathfrak{p}$
satisfies the uniformity conditions specified in
\S\ref{sec:uniformity}.  Then the estimate
\eqref{eqn:sobolev-norms-via-C-plus-delta} and the results of
\S\ref{sec:uniform-trace-class}, \S\ref{sec:sobolev-conductor},
\S\ref{sec:reduct-isotyp-vect}, \S\ref{sec:integration-parts}
and \S\ref{sec:reduct-pure-tens} hold.  The reduction to the
local setting is explained in
\cite[\S2.6.3, Proof of (S1d)]{michel-2009}, and follows
ultimately from the divisor bound.

\section{Local preliminaries}
\label{sec-4}
Let $F$ be a local field, and let $\psi$ be a nontrivial unitary
character of $F$.

\subsection{Norms}\label{sec:norms-whit-type-reps}
Let $\sigma$ be a Whittaker type representation of
$G$.
Let $\sigma$ be a Whittaker type representation of $G$.
We obtain from \S\ref{sec:norms-repr} a system of Sobolev norms
$\mathcal{S}_d$ ($d \in \mathbb{R}$), and in particular a norm
$\|.\| = \mathcal{S}_0$ and inner product $\langle , \rangle$.
More precisely:
\begin{itemize}
\item If $\sigma$ is square-integrable, then we regard it as
  coming equipped with an invariant inner product and apply the
  construction of \S\ref{sec:norms-repr} with $M = G$.
\item If $\sigma = \mathcal{I}(\chi)$, then we write
  $\chi = |.|^c \chi_0$ with $\chi_0$ unitary and apply the
  construction of \S\ref{sec:norms-repr} with $M = A$.  Thus for
  $f \in \mathcal{I}(\chi)$ we have
  $\|f\|^2 = \int_K |f|^2$.
  This norm is invariant if $\chi$ is unitary.
\end{itemize}
We choose an orthonormal basis $\mathcal{B}(\sigma)$
for $\sigma$
consisting of $K$-isotypic elements.

\subsection{Whittaker intertwiners and duality}\label{sec:whitt-intertw-dual}
Let $\sigma$ be a Whittaker type
representation of $G$.
We denote by $\mathcal{W}(\sigma,\psi)$ the $\psi$-Whittaker
model, which consists of smooth functions
$W : G \rightarrow \mathbb{C}$ satisfying
$W(n(x) g) = \psi(x) W(g)$.  We similarly define
$\mathcal{W}(\sigma, \bar{\psi})$.  We normalize intertwiners
\[
  \sigma \rightarrow
  \mathcal{W}(\sigma,\psi),
  \quad 
  \sigma \rightarrow
  \mathcal{W}(\sigma,\bar{\psi})
\]
\[
  f \mapsto W_f,
  \quad
  f \mapsto \tilde{W}_f
\]
as follows.
If $\sigma$ is square-integrable,
then we require that
\[
  \|f\|^2 = 
  \int_A |W_f|^2
  =
  \int_A |\tilde{W}_f|^2
\]
(see \cite[\S6.4]{MR748505}
and \cite[\S10.2]{MR1999922}).
The remainder of this section
treats the
case $\sigma = \mathcal{I}(\chi)$.

We define
\begin{equation}\label{eq:W-f-of-g-definition}
  W_f(g) :=
  \int_{x \in F}^{\reg} f(w n(x) g) \psi(-x) \, d x,
  \quad 
  \tilde{W}_f(g) :=
  \int_{x \in F}^{\reg} f(w n(x) g) \psi(x) \, d x.
\end{equation}
The integrals \eqref{eq:W-f-of-g-definition}
converge absolutely for $\Re(\chi)$ sufficiently
large and may be defined in general by
analytic
continuation \cite[Thm 15.4.1]{MR1170566},
or equivalently, by regularization
with respect to the action
of $F^\times$.
For instance, for any
$\kappa \in C_c^\infty(F^\times)$
with $\int_u \kappa(u) \, \frac{d u}{|u|} = 1$,
we have
\begin{equation}\label{eq:whittaker-function-explicated-defn}
  W_f(g)
  =
  \int_{x \in F}
  \left(
    \int_u
    \kappa(u)
    f(w n(x u) g)
    \psi(- x u)
    \, d u
  \right)
  \, d x,
\end{equation}
first for $\Re(\chi)$ sufficiently large, then in general by
analytic continuation with respect to a holomorphic family of
vectors.  The point is that the integral
\eqref{eq:whittaker-function-explicated-defn},
taken in the indicated order, converges
absolutely for all $\chi$.

We note the formula
\begin{equation}\label{eq:W-f-viz-f}
  W_f(y)
  :=
  W_f(a(y))
  = |y|^{1/2}
  \chi^{-1}(y)
  \int_{x \in F}^{\reg}
  f(w n(x)) \psi(- y x) \, d x,
\end{equation}
which follows by a simple rearrangement of
\eqref{eq:W-f-of-g-definition}.
If $\Re(\chi) > -1/2$,
then we have moreover the absolutely convergent
inversion formula
\begin{equation}\label{eq:fourier-inversion-for-whittaker-intertwiner}
  f(w n(x))
  = \int_{y \in F^\times}
  W_f(y)
  |y|^{-1/2} \chi(y)
  \psi(y x)
  \, d y,
\end{equation}
as follows from Fourier inversion on $C_c^\infty(F)$ together
with the absolute convergence of the double integral
$\int_{y \in F^\times} \int_{x \in F} |\int_u \kappa(u) f(w n(x
u)) \psi(- x y u) \, d u| \, d x \, d y$.

We note that if $\chi$ is unitary (i.e., $c = 0$), then
by
\eqref{eqn:local-parseval-f-vs-W-f} and Parseval,
\begin{equation}
  \|f\|^2 = \int_K |f|^2 = \int_{B \backslash G} |f|^2 =
  \int_A |W_f|^2
  =
  \int_A |\tilde{W}_f|^2.
\end{equation}
On the other hand, the unitary structures on complementary
series representations
play no role in this paper.

For general $\chi$,
there is a natural duality between $\mathcal{I}(\chi)$ and
$\mathcal{I}(\chi^{-1})$, given by
$(f,\tilde{f}) \mapsto \int_{B \backslash G} f \tilde{f}$.  Here
we equip $B \backslash G$ with the quotient measure
corresponding to the left Haar on $B$ described by the map
$A \times N \rightarrow B$, $(a,n) \mapsto a n$.  In view of our
choice (\S\ref{sec-4-1}) of Haar measures on $G$ and on $K$, we
may write $\int_{B \backslash G} f \tilde{f}$ as
$\int_{n \in N} f (w n) \tilde{f}(w n)$ or as
$\int_K f \tilde{f}$.
We then have
\begin{equation}\label{eqn:local-parseval-f-vs-W-f}
  \int_K \tilde{f} f =
  \int_{B \backslash G} \tilde{f}f  = \int_{A}^{\reg} \tilde{W}_{\tilde{f}} W_f,
\end{equation}
first when
$-1/2 < \Re(\chi) < 1/2$
(so that the latter integral converges absolutely, see \cite[\S1.9, \S1.10]{MR3889963}),
then in general by analytic continuation along a flat family.

\subsection{Holomorphic families of vectors}\label{sec:families-vectors}
By a holomorphic family of vectors
$f[\chi] \in \mathcal{I}(\chi)$
defined for $\chi$ in some open subset $X$ of $A^\wedge$,
we mean a continuous map
\[
  X \times G \rightarrow \mathbb{C} 
\]
\[
  (\chi,g) \mapsto f[\chi](g)
\]
with the following properties.
\begin{itemize}
\item $f[\chi] \in \mathcal{I}(\chi)$ for each $\chi \in X$.
\item The map $X \ni \chi \mapsto f[\chi](g) \in \mathbb{C}$ is
  holomorphic for each $g \in G$.
\item If $F$ is non-archimedean,
  then for each $\chi_0 \in X$
  there is a compact open subgroup $U$ of $G$ so
  that the vector $f[\chi]$ is $U$-invariant
  for all $\chi$ in some neighborhood of $\chi_0$.
\item If $F$ is archimedean, then for every continuous seminorm
  $\mathcal{N}$ on $C^\infty(K)$, the composition
  $X \ni \chi \mapsto \mathcal{N}(f[\chi])$ is locally bounded;
  equivalently,
  $\chi \mapsto \mathcal{S}_d(f[\chi])$
  is locally bounded for each $d \in \mathbb{R}$
  (\S\ref{sec:norms-repr}).
\end{itemize}
By a flat family of vectors $f[\chi] \in \mathcal{I}(\chi)$ we
mean one for which $f[\chi](g)$ is locally constant in $\chi$
for each $g \in K$.
We note that
any holomorphic family $f[\chi]$ may be written as a sum
$\sum_v c_v(\chi) v[\chi]$ of flat families $v[\chi]$ consisting
of $K$-isotypic vectors with holomorphic coefficients
$c_v(\chi)$.  For $F$ non-archimedean, the sum over $v$ is
finite.  For $F$ archimedean, it converges rapidly in the sense
that if $v[\chi]$ has $K$-type $\nu$, then
$\|v[\chi]\| \ll_{d,\chi} N_{\mathcal{I}(\chi),\nu}^{-d}$ for
any fixed $d$ (see \S\ref{sec:constr-sobol-norms}), locally
uniformly in $\chi$.

\subsection{Standard intertwining operator}
The standard intertwining operator
is the $G$-equivariant map
$M : \mathcal{I}(\chi) \rightarrow \mathcal{I}(\chi^{-1})$ is
given for $\chi$ of large real part by
$M f(g) := \int_{n \in N} f(w n g)$ and in general
by meromorphic continuation.
One may check
(see e.g. \cite[Prop 4.5.9]{MR1431508})
that
\begin{equation}\label{eqn:W-Mf-vs-W-f}
  W_{M f}
  =
  \gamma(\psi, \chi^{-2},1)
  W_f.
\end{equation}

\subsection{Normalized elements
and unramified calculations}\label{sec:normalized-spherical-induced-rep}
Suppose that $F$ is non-archimedean and $\chi$ is unramified.
The \emph{normalized
  spherical element} $f \in \mathcal{I}(\chi)$ is defined to be
the unique $K$-invariant vector with $f(1) = 1$, thus
$f (n(x) a(y) k) = |y|^{1/2} \chi(y)$ for
$x \in F, y \in F^\times, k \in K$.

Suppose $(F,\psi)$ is unramified.
Let $\sigma$ be an unramified
generic irreducible representation.
Then the nonzero $K$-invariant vectors
$W \in \mathcal{W}(\sigma,\psi)$
satisfy $W(1) \neq 0$
\cite[Thm 4.6.5]{MR1431508}.
The \emph{normalized spherical Whittaker function} $W \in
\mathcal{W}(\sigma,\psi)$
is defined to be the unique $K$-invariant vector
with $W(1) = 1$.
Such an element exists
(see \emph{loc. cit.}).
If $W \in \mathcal{W}(\sigma,\psi)$ and
$\tilde{W} \in \mathcal{W}(\sigma,\bar{\psi})$ are normalized
spherical, then
$\int_{A}^{\reg}\tilde{W} W = L(\ad \sigma, 1) / \zeta_F(2)$
(see \cite[Proof of Prop 3.8.1]{MR1431508}).

For any local field $F$
and character $\chi$ of $F^\times$ for which $L(\chi^2,1)$
is finite, we attach to each
$f \in \mathcal{I}(\chi)$ its multiple
\begin{equation}\label{eq:defn-f-asterisk}
  f^* := L(\chi^2,1) f.
\end{equation}
Suppose now that $F$ is non-archimedean
and $\chi$ is unramified,
so that $\mathcal{I}(\chi)$ is unramified.
Let $f \in \mathcal{I}(\chi)$
and $\tilde{f} \in \mathcal{I}(\chi^{-1})$
denote the normalized spherical elements.
Then:
\begin{itemize}
\item For unramified $\psi$, the Whittaker function $W_{{f}^*}$
  and $\tilde{W}_{\tilde{f}^*}$ are normalized spherical (i.e.,
  $f(1) = 1 \implies W_{f^*}(1) = 1$).  (When
  $\mathcal{I}(\chi)$ is reducible, we define $W_{f^*}$ and
  $\tilde{W}_{\tilde{f}^*}$ by analytic continuation.)
\item $M ( f^*) = (\tilde{f} ) ^*$;
\item
  $\int_{B \backslash G} \tilde{f} f = \zeta_F(1)/\zeta_F(2)$.
\end{itemize}
For proofs of these facts we refer
respectively to \cite[Thm 4.6.5]{MR1431508},
\cite[Prop 4.6.7]{MR1431508}
and \cite[\S3.1.6]{michel-2009}.

\subsection{Hecke integrals}
\label{sec-4-3}
\label{sec:local-hecke-basic}
We recall
some special cases
of the results of \cite{MR701565,MR0401654}.
Let $\sigma$ be a Whittaker type representation of $G$, and let
$W \in \mathcal{W}(\sigma,\psi)$.  For almost all $\omega \in A^\wedge$, the
\emph{local Hecke integral} $\int_{A}^{\reg} W \omega$ is
defined, and meromorphic in $\chi$.  The integral converges for
$\Re(\omega)$ large enough, and the ratio
\begin{equation}\label{eqn:local-hecke-ratio}
  \frac{\int_A^{\reg} W \omega }{ L(\sigma \otimes \omega,1/2)/\zeta_F(1)}
\end{equation}
extends to a holomorphic function of $\omega$.

For a character $\omega$ of $A$, we denote by
$\mathbb{C}(\omega)$ the corresponding one-dimensional
representation.  For any representation $V$ of $A$, we use the
map $v \otimes 1 \mapsto v$ to identify
$V \otimes \mathbb{C}(\omega)$ with $V$, but equipped with the
twisted $A$-action.  When $\omega = |.|^s$, we write simply
$\mathbb{C}(s)$.  (We will never use potentially ambiguous
notation such as $\mathbb{C}(1)$.)
Thus
after choosing an identification
$\sigma \cong \mathcal{W}(\sigma,\psi)$,
the (normalized) Hecke integral
defines an $A$-invariant functional
\[
  \sigma \otimes \mathbb{C}(\omega) \rightarrow \mathbb{C}.
\]

If $(F,\psi)$ and $\sigma$ are unramified and
$W \in \mathcal{W}(\sigma,\psi)$ is the normalized spherical
vector, then the ratio
\eqref{eqn:local-hecke-ratio}
vanishes unless $\omega$ is
unramified, in which case
it evaluates to $1$
(see \cite[Prop 3.5.3]{MR1431508} and
\eqref{eq:volume-of-unit-group}).

\subsection{Rankin--Selberg integrals}
\label{sec-4-5}
\label{sec:local-rs-basic}
Let $\sigma_1, \sigma_2$ be Whittaker type representations
of $G$.  Let $\chi$ be a character of $A$.  Let
$W_1 \in \mathcal{W}(\sigma_1, \psi), W_2 \in
\mathcal{W}(\sigma_2, \bar{\psi})$ and
$f \in \mathcal{I}(\chi)$.  If $\Re(\chi)$ is large enough, then
the \emph{local Rankin--Selberg integral}
$\int_{N \backslash G} W_1 W_2 f$ converges absolutely.  It
extends meromorphically.  We denote that extension by
$\int_{N \backslash G}^{\reg} W_1 W_2 f$.
The notation reflects that the extension may be defined
equivalently by regularized integration on the quotient
$N \backslash G$, regarded as a left $A$-space,
as follows from lemma \ref{lem:unif-polyn-type-hecke}
and \eqref{eqn:regularized-integral-via-mellin-analysis}.
Letting $f \in \mathcal{I}(\chi)$ vary over a
holomorphic family, the ratio
\begin{equation}\label{eqn:normalized-local-RS}
  \frac{\int_{N \backslash G}^{\reg} W_1 W_2 f^* }{ L(\sigma_1 \otimes
    \sigma_2 \otimes \chi,1/2)/\zeta_F(2)}
\end{equation}
extends to a holomorphic function of $\chi$.
To deduce this last assertion from the cited references,
we use
(see the proof of \cite[Lemma, p6]{MR871663})
that any holomorphic
family $f[\chi] \in \mathcal{I}(\chi)$
is locally of the form
$f[\chi] = f_{\chi,\Phi}$ for some
$\Phi \in \mathcal{S}(F^2)$,
where $f_{\chi,\Phi}$
is defined by the local Tate integral
\begin{equation}\label{eq:godement-parametrization-induced-rep}
  f_{\chi,\Phi}(g)
  :=
  \frac{|\det g|^{1/2} \chi(\det g)}{L(\chi^2,1)}
  \int_{r \in F^\times}^{\reg}
  \Phi(r e_2 g)
  |r|  \chi^2(r) 
  \, \frac{d r}{|r|}.
\end{equation}

For each $\chi$, we may find $f$ so that the ratio
\eqref{eqn:normalized-local-RS} (defined in general by analytic
continuation) is nonzero.  Thus
when $\sigma_1, \sigma_2$ and $\mathcal{I}(\chi)$ are irreducible,
\eqref{eqn:normalized-local-RS} defines a basis element for the
one-dimensional space of $G$-invariant trilinear functionals on
$\sigma_1 \otimes \sigma_2 \otimes \mathcal{I}(\chi)$ (see
\cite{MR1198305, MR1810211}).

If $(F,\psi),\sigma_1, \sigma_2$ and $\chi$ are unramified and
$W_1, W_2,f$ are normalized spherical, then the ratio
\eqref{eqn:normalized-local-RS} evaluates to $1$ (see \cite[Prop
3.8.1]{MR1431508}, \cite[\S3.1.6]{michel-2009} and
\eqref{eq:volume-of-unit-group}).

\subsection{Uniform polynomial-type
  estimates}\label{sec:local-hecke-rs-continuity}
We record here some crude polynomial estimates for local Tate,
Hecke and Rankin--Selberg integrals, as well as some related
estimates for Whittaker functions, that will be used later to
verify some technical assertions (concerning convergence,
continuity, etc.)  in the archimedean case and to verify the
polynomial dependence of our estimates upon auxiliary parameters
(e.g., upon $\sigma$ in Theorem \ref{thm:CI}).

Recall
(\S\ref{sec:norms-schw-spac})
that we have equipped
the Schwartz space $\mathcal{S}(F)$
with Sobolev norms $\mathcal{S}_d$
indexed by real numbers $d$,
and that $\mathcal{S}$ denotes a Sobolev norm of some fixed
index.
\begin{lemma}\label{lem:crude-tate-integral-estimate}
  Let $\chi$ be a character
  of $F^\times$.
  Suppose that the real part of $\chi$ is
  $\O(1)$
  and that $\chi$ is at least some fixed positive
  distance away from any pole of $L(\chi,0)$.
  Then for each fixed real number $d$
  and any $\phi \in \mathcal{S}(F)$,
  we have the
  local Tate integral
  estimate
  \begin{equation}
    \int_{x \in F^\times}^{\reg} \phi(x) \chi(x) \, \frac{d x}{|x|}
    \ll \mathcal{S}(\phi) C(\chi)^{-d}.
  \end{equation}
\end{lemma}
\begin{proof}
  Suppose first that $\Re(\chi) \geq 2$.
  Then the indicated integral converges
  absolutely, and the required estimate follows readily: In the
  archimedean case, we appeal to the Sobolev lemma and partial
  integration.  In the non-archimedean case, we may assume
  as in \S\ref{sec:reduct-isotyp-vect}
  that
  $\phi(y) = 1_{x + \mathfrak{o}}(y) \psi_0(\xi y)$ for some
  unramified character $\psi_0$ of $F$
  and
  $x,\xi \in F$.
  Then $\mathcal{S}_{d'}(\phi) \asymp \max(1,|x|,|\xi|)^{d'}$.
  The integral $\int_{x \in F^\times} \phi(x) \chi(x) \,
  \frac{d x}{|x|}$
  vanishes unless one of the following
  conditions holds,
  in which case its magnitude is $\O(1)$:
  \begin{itemize}
  \item $x \in \mathfrak{o}, \xi \in \mathfrak{o}$,
    $C(\chi) = 1$.
  \item $x \in \mathfrak{o}, \xi \notin \mathfrak{o}$,
    $C(\chi) = 1, |\xi| \leq q$.
  \item $x \in \mathfrak{o}, \xi \notin \mathfrak{o},
    C(\chi) > 1, |\xi| = C(\chi)$.
  \item
    $x \notin \mathfrak{o}$,
    $|x| \max(1,|\xi|) \geq C(\chi)$.
  \end{itemize}
  In each case the required estimate follows.
  
  The required estimate follows then for $\Re(s) \leq -1$ by the
  Tate local functional equation
  \begin{equation}\label{eqn:Tate-local-func-eqn}
    \int_{x \in F^\times}^{\reg}
    \phi(x)
    \chi(x) |x|^s \, \frac{d x}{|x|}
    = 
    \frac{\int_{x \in F^\times}^{\reg}
      (\int_{y \in F} \phi(y) \psi(-y x) \, d y)
      \chi^{-1}(x) |x|^{1-s} \, \frac{d x}{|x|}}{  \gamma(\psi,\chi,s)}
  \end{equation}
  and the
  Stirling consequence \eqref{eq:stirling-for-general-RS-gamma}, and
  finally for general $s$ by Phragmen--Lindel{\"o}f.
\end{proof}

\begin{lemma}\label{lem:unif-polyn-type-hecke}
  Let $\sigma$ be a Whittaker type representation of $G$.
  Assume $\sigma$ is $\vartheta$-tempered
  for some $\vartheta = \O(1)$.
  For each
  $f \in \sigma$,
  we have the following estimates:
  \begin{enumerate}[(i)]
  \item Let $\omega \in A^\wedge$
    be
    such that $\Re(\omega) = \O(1)$
    and the distance between $\omega$
    and each pole of $L(\sigma \otimes \omega,1/2)$
    is $\gg 1$.  Then for each fixed $d$,
    \begin{equation}\label{eqn:upper-bound-Mellin-transform-W}
      \int_A^{\reg} W_f \omega \ll
      \mathcal{S}(f)
      C(\omega)^{-d} q^{\O(1)}.
    \end{equation}
  \item
    For $|y| \geq 1$ and fixed $d$,
    \begin{equation}\label{eqn:estimate-W-near-infinity}
      W_f(y) \ll \mathcal{S}(f) |y|^{-d} q^{\O(1)}.
    \end{equation}
  \item
    For $|y| \leq 1$,
    \begin{equation}\label{eq:estimate-W-near-0-with-eps}
      W_f(y) \ll \mathcal{S}(f)
      |y|^{1/2-\vartheta-\eps} q^{\O(1)}.
    \end{equation}
  \item
    There is a finite function
    $\phi$ on $F^\times$
    so that for $|y| \leq 1$ and fixed $d$,
    \begin{equation}\label{eqn:estimate-W-near-zero}
      W_f(y) - \phi(y) \ll \mathcal{S}(W_f) |y|^{d} q^{\O(1)}.
    \end{equation}
    Moreover, we may choose $\phi$ so that:
    \begin{enumerate}[(a)]
    \item The dimension of the span of the $A$-translates
      of $\phi$ is $\O(1)$.
    \item We may write
      $\phi(y) = \sum c_i \chi_i(y) (\log |y|)^{m_i}$,
      corresponding to the principal part of
      $\int_A^{\reg} W_f \omega$ for $\Re(\omega) > -d$.  Thus the
      $\chi_i \in A^\wedge$ are poles of
      $L(\sigma \otimes \chi_i^{-1}, 1/2)$, the $m_i$ are
      nonnegative integers with $m_i + 1$ bounded by the
      multiplicity of the corresponding pole, and the
      coefficients $c_i$ depend linearly upon $W$.
    \item If $\sigma = \mathcal{I}(\chi)$ and we take
      $f = f[\chi]$ for some holomorphic family
      $f[\chi] \in \mathcal{I}(\chi)$ defined for $\chi$ in a
      small open subset of $A^\wedge$, then
      $\phi = \phi_{f[\chi]}$ varies pointwise holomorphically
      with $\chi$.
    \end{enumerate}
  \end{enumerate}
\end{lemma}
\begin{proof}
  We loosely follow an argument of Jacquet (see \cite[Thm 2.3,
  \S12.2]{MR2533003} and also \cite[\S3.2.3]{michel-2009}).  The
  estimates \eqref{eqn:estimate-W-near-infinity},
  \eqref{eq:estimate-W-near-0-with-eps} and
  \eqref{eqn:estimate-W-near-zero}, together with the noted
  refinement of \eqref{eqn:upper-bound-Mellin-transform-W} and
  the noted properties of $\phi$, follow readily from
  \eqref{eqn:upper-bound-Mellin-transform-W}, Mellin expansion
  and Cauchy's theorem.  By Phragmen--Lindel{\"o}f, it suffices
  to establish \eqref{eqn:upper-bound-Mellin-transform-W} when
  $\Re(\omega)$ is sufficiently positive or negative in terms of
  $\vartheta$.  Recall the local functional equation (see
  \S\ref{sec:local-gamma-factors}):
  \[
    \gamma(\psi,\sigma
    \otimes \omega,1/2)
    \int^{\reg}_A W \omega
    =
    \int^{\reg}_A w W |.| \omega^{-1}.
  \]
  For $\Re(\omega)$ sufficiently positive in terms of
  $\vartheta$, the the Stirling asymptotics
  \eqref{eq:stirling-for-general-RS-gamma} imply
  that $\gamma(\psi,\sigma \otimes \omega,1/2) \ll q^{\O(1)}$.  We
  thereby reduce to establishing
  \eqref{eqn:upper-bound-Mellin-transform-W} (both for $W$ and
  its Weyl translate $w W$)
  when $\Re(\omega)$ is sufficiently positive.
  In that case, $\int_A W \omega$
  converges absolutely (as follows either from standard
  estimates for individual $W$ or from the uniform estimates
  verified below).  Using the $A$-equivariance of the Hecke
  integral
  and integration by parts with respect to $\omega$
  (\S\ref{sec:integration-parts}), we may reduce further to establishing that
  \begin{equation}\label{eqn:very-weak-bound-for-Hecke-integral-uniform-0}
    \int_A W \omega \ll  \mathcal{S}(W)
    C(\omega)^{\O(1)} q^{\O(1)}.
  \end{equation}
  Since $\Re(\omega)$ is large enough in terms
  of $\vartheta$ but of size $\O(1)$,
  we reduce further to verifying for each fixed $d$ that
  \begin{equation}\label{eqn:very-weak-bound-Whittaker-upper-bound-all-over}
    W(y)
    \ll
    \mathcal{S}(W)
    |y|^{-d} q^{\O(1)}.
  \end{equation}
  If $\sigma$ is square-integrable, the required
  estimate
  \eqref{eqn:very-weak-bound-Whittaker-upper-bound-all-over}
  follows (in stronger form) from \cite[\S3.2.3]{michel-2009}.
  We may thus suppose that $\sigma = \mathcal{I}(\chi)$,
  so that $W = W_f$ for some $f \in \mathcal{I}(\chi)$.
  
  Define $\kappa \in C_c^\infty(F^\times)$
  in the non-archimedean case as the normalized characteristic
  function of $\mathfrak{o}^\times$
  and in the archimedean case as a fixed bump function,
  with the normalization in either case so that
  $\int \kappa(u) \, \frac{d u }{|u|} = 1$.
  We may then express $W_f$
  as the absolutely-convergent integral
  \begin{equation}\label{eqn:W-f-actual-defn-via-convolution}
    W_f(y)
    =
    |y|^{1/2} \chi^{-1}(y)
    \int_{x \in F}
    I_f(x,y)
    \, d x,
  \end{equation}
  where
  \begin{equation}
    I_f(x,y) :=
    \int_{u \in F^\times}
    \kappa(u)
    f(w n(x u)) \psi(- y x u)
    \, \frac{d u}{|u|}.
  \end{equation}
  We thereby reduce to verifying that
  \begin{equation}\label{eqn:estimate-I-f-small-x}
    \int_{x : |x| \leq 1}
    I_f(x,y) \, d x
    \ll |y|^{-d}
    \mathcal{S}(f) q^{\O(1)},
  \end{equation}
  and that
  \begin{equation}\label{eqn:estimate-I-f-large-x}
    I_f(x,y)
    \ll |x y|^{-d}
    \mathcal{S}(f) q^{\O(1)}
    \text{ for } |x| \geq 1.
  \end{equation}

  To that end, we first verify some pointwise estimates for $f$.
  We may assume that $f$ is a unit vector in $\sigma^{\nu}$ for
  some $\nu \in K^\wedge$.  Let us say that a quantity is
  \emph{bounded polynomially} if it is of the form
  $N_{\sigma \nu}^{\O(1)} q^{\O(1)}$.  By the Sobolev lemma on
  $K$, we see that the $L^\infty$-norms on $K$ of $f$ and any
  archimedean derivatives of fixed degree are bounded polynomially.
  (In the non-archimedean case,
  we use that $f$ is invariant by $K[n]$
  with $N_{\sigma \nu} = q^n$
  to pass from an $L^2$ bound to an $L^\infty$ bound.)
  By lemma
  \ref{lem:compare-conductors}, $C(\chi)$ is
  bounded polynomially.  By the definition of
  $\mathcal{I}(\chi)$, it follows that the $L^\infty$-norms on
  $\{x : |x| \leq 10\}$ of the functions
  $x \mapsto f(w n(x)), x \mapsto f(n'(x))$ and any archimedean
  derivatives of fixed degree are bounded polynomially.
  
  The estimate \eqref{eqn:estimate-I-f-small-x}
  follows by partial integration.
  (In the archimedean case,
  this carries the usual meaning.
  In the non-archimedean case, we use that $x \mapsto f(w n(x))$ is invariant
  under translation by $\mathfrak{p}^n$ with $N_{\sigma \nu} =
  q^n$
  to see that
  $\int_{x : |x| \leq 1}
  I_f(x,y) \, d x$ vanishes unless $|y| \ll N_{\sigma \nu}$,
  in which case it is bounded polynomially.)

  The estimate \eqref{eqn:estimate-I-f-large-x}
  follows similarly,
  by partial integration
  and the identity
  \begin{equation}\label{eqn:identity-f-w-n-x-large-x-useful}
    f(w n(x)) = |x|^{-1} \chi^{-2}(x) f(n'(x)).  
  \end{equation}
\end{proof}

\begin{lemma}\label{lem:polyn-bound-rs}
  Let $\sigma_1, \sigma_2$
  be Whittaker type representations
  of $G$.
  Let $\chi$ be a character
  of $A$.
  Assume that $\sigma_1, \sigma_2$ are $\vartheta$-tempered
  for some $\vartheta = \O(1)$.
  Assume that $\chi$ has real part $\O(1)$ and is some fixed
  positive distance away from the poles of
  $L(\sigma_1 \otimes \sigma_2 \otimes \chi, 1/2)$.
  Then for $f_1 \in \sigma_1, f_2 \in \sigma_2, f_3 \in \mathcal{I}(\chi)$,
  \begin{equation}\label{eq:sobolev-bound-RS}
    \int_{N \backslash G}^{\reg} W_{f_1} \tilde{W}_{f_2} f_3 \ll
    \mathcal{S}(f_1) \mathcal{S}(f_2) \mathcal{S}(f_3) q^{\O(1)}.
  \end{equation}
\end{lemma}
\begin{proof}
  Suppose first that $\Re(\chi)$ is sufficiently large
  in terms of $\vartheta$.
  Then the integral converges
  absolutely,
  and the Whittaker function estimates
  \eqref{eqn:estimate-W-near-infinity},
  \eqref{eq:estimate-W-near-0-with-eps}
  and the Sobolev lemma consequence
  $\|f\|_{L^\infty(K)} \ll \mathcal{S}(f)$
  (see the paragraph after \eqref{eqn:estimate-I-f-large-x})
  imply the required estimate \eqref{eq:sobolev-bound-RS}.
  Integrating by parts with respect
  to $\mathcal{I}(\chi)$
  (\S\ref{sec:integration-parts})
  yields for any fixed $d \geq 0$ the refined estimate
  \begin{equation}\label{eqn:refined-RS-sobolev-large-real-part-chi}
    \int_{N \backslash G} W_1 W_2 f \ll
    \mathcal{S}(W_1) \mathcal{S}(W_2) \|f\| C(\chi)^{-d} q^{\O(1)}
  \end{equation}
  (We have used here
  that
  $C(\mathcal{I}(\chi)) \asymp C(\chi)^2$.)

  The following local functional equation may be
  derived from \cite[Thm 2.7]{MR701565} and \cite[Thm
  2.1]{MR2533003}, using
  \eqref{eq:godement-parametrization-induced-rep}:
  \begin{equation}\label{eq:RS-local-func-eqn}
    \int_{N \backslash G}^{\reg} W_1 \cdot W_2
    \cdot f
    =
    \frac{
      \gamma(\psi,\chi^2,0)
    }
    {
      \gamma(\psi, \sigma_1 \otimes \sigma_2 \otimes \chi, 1/2)
    }
    \int_{N \backslash G}^{\reg} W_1 \cdot W_2
    \cdot M f.
  \end{equation}
  (We have used also that $\sigma_1, \sigma_2$ are
  representations of $G = \PGL_2(F)$, hence have trivial central
  character when regarded as representations of $\GL_2(F)$.)  If
  $\Re(\chi)$ is sufficiently \emph{negative} in terms of
  $\vartheta$ (but chosen, as indicated, to avoid the poles of
  $L(\sigma_1 \otimes \sigma_2 \otimes \chi,1/2)$), then
  we have by \eqref{eq:stirling-for-general-RS-gamma}
  the crude Stirling-type estimates
  $1/\gamma(\psi, \sigma_1 \otimes \sigma_2 \otimes \chi, 1/2)
  \ll q^{\O(1)}$ and $\gamma(\psi,\chi^2,0) \ll (q C(\chi))^{\O(1)}$.
  
  We claim that $\|M f\| \ll \mathcal{S}(f)$.  To see this, it
  suffices (by the $K$-invariance of the Sobolev norms) to
  verify that
  $M f(1) = \int_{x \in F} f(w n(x)) \, d x \ll \mathcal{S}(f)$.
  To that end, we smoothly decompose the $x$-integral according
  as $|x| \leq 2$ and $|x| \geq 1$.  The contribution from the
  former range is estimated as in the proof of lemma
  \ref{lem:unif-polyn-type-hecke}.  After the change of
  variables $x \mapsto 1/x$ and the identity
  \eqref{eqn:identity-f-w-n-x-large-x-useful}, the contribution
  from the latter range is a Tate integral for which adequate
  bounds follow from lemma
  \ref{lem:crude-tate-integral-estimate}.
  
  We now embed $f$ in a flat family $f[\chi]$.
  Note that $\|f[\chi]\|$ is locally constant in such a family.
  By the local functional equation
  \eqref{eq:RS-local-func-eqn}
  and the noted estimates for the $\gamma$-factors and $M f[\chi]$,
  we see that the estimate
  \eqref{eqn:refined-RS-sobolev-large-real-part-chi}
  holds also when $\chi$ has sufficiently \emph{negative} real part.
  By Phragmen--Lindel{\"o}f,
  we deduce
  \eqref{eqn:refined-RS-sobolev-large-real-part-chi}
  for all indicated $\chi$,
  and in particular its consequence \eqref{eq:sobolev-bound-RS}.
\end{proof}

\begin{remark}
  The stated estimates
  that require separation from some pole
  may be generalized
  via Cauchy's
  theorem.
  For illustration,
  we record the general form of
  \eqref{eqn:upper-bound-Mellin-transform-W}.
  For any $\omega \in A^\wedge$, let
  $P(s)$ be the smallest monic polynomial such that
  $P(s) L(\sigma \otimes \omega, 1/2 + s)$ is holomorphic for
  $|s| < 2$.  Then $P(s) \int_A^{\reg} W_f \omega |.|^s$ is
  holomorphic for $|s| < 2$ and satisfies the analogue of
  \eqref{eqn:upper-bound-Mellin-transform-W} for $|s| \leq 1$.
\end{remark}

\section{Global preliminaries}
\label{sec-5}
\subsection{Groups, measures, etc}
\label{sec-5-1}
Let $F$ be a number field with adele ring $\mathbb{A}$.  We fix
a nontrivial unitary character
$\psi$ of $\mathbb{A}/F$.
We write $\mathfrak{p}$ for a place of $F$,
finite or infinite.
We write $\mathbf{G}$
for the $F$-algebraic groups given by $\PGL_2$.  We define
subgroups $\mathbf{B}, \mathbf{N}, \mathbf{A}$ of $\mathbf{G}$
by analogy to the local case.  We set $G := \mathbf{G}(F)$,
$G_\mathbb{A} := \mathbf{G}(\mathbb{A}), G_\mathfrak{p} :=
\mathbf{G}(F_\mathfrak{p})$ and
$[G] := G \backslash G_\mathbb{A}$.  We similarly define $A$,
$A_{\mathbb{A}}, A_\mathfrak{p}$ and $[A]$, and likewise for
$\mathbf{N}$ and $\mathbf{B}$.  Note that
$[A] \cong \mathbb{A}^\times/F^\times$ is an abelian group.  We
define $K = \prod_{\mathfrak{p}} K_\mathfrak{p}$ with
$K_\mathfrak{p} \leq G_\mathfrak{p}$ as in the local case.  We
identify $N$ and $A$ with $F$ and $F^\times$ as in the local
case, and similarly for $N_\mathbb{A}, A_\mathbb{A}$, etc.

We equip $N_\mathbb{A}$ with the product of the local measures
and $[N]$ with the quotient measure
(of volume $1$).
The
product of the local measures on $A_\mathfrak{p}$ does not
literally converge to a measure on
$A_{\mathbb{A}} \cong \mathbb{A}^\times$, but may be regularized
as follows: for $f \in C_c^\infty(A_\mathbb{A})$ of the form
$f(a) = \prod_{\mathfrak{p}} f_\mathfrak{p}(a_\mathfrak{p})$,
with $f_\mathfrak{p}$ the characteristic function of the unit
group for almost all finite $\mathfrak{p}$, we set
$\int_{A_\mathbb{A}} f := \xi_F^*(1)^{-1} \prod_\mathfrak{p}
\zeta_{F_\mathfrak{p}}(1) \int_{A_\mathfrak{p}} f_\mathfrak{p}$.
We equip $[A] \cong \mathbb{A}^\times/F^\times$ with the
quotient measure.
The pushforward of this measure under the
idelic norm
$\mathbb{A}^\times/F^\times \xrightarrow{|.|}
\mathbb{R}^\times_+$ is then $\tfrac{d t}{t}$, with $d t$
Lebesgue measure.

As in the local case,
the character group $[A]^\wedge$ of $[A]$
is a complex manifold
with charts $\chi |.|^s \mapsto s$.
We equip the group of unitary characters of
$[A]$, hence also its cosets consisting of characters of given
real part, with the measure dual to the chosen measure
on $[A]$.
Then for a real number $c$ and $f : [A]^\wedge \rightarrow \mathbb{C}$,
\begin{equation}\label{eq:measure-on-dual-omega}
  \int_{\chi : \Re(\chi) = c}
  f(\chi)
  =
  \sum_{\chi_0}
  \int_{s : \Re(s) = c}
  f(\chi_0 |.|^s) \, \frac{d s}{2 \pi i },
\end{equation}
where $\chi_0$ runs over a set of representatives for the group
of unitary characters of $[A]$ modulo the subgroup
$\{|.|^{i t} : t \in \mathbb{R} \}$.  The normalization
\eqref{eq:measure-on-dual-omega}
is well-suited for the applications of Cauchy's
theorem and contour shift arguments
given in \S\ref{sec-12}.

We define the measure on $G_\mathbb{A}$
using the measures on $N_\mathbb{A}$ and $A_\mathbb{A}$
as in the local case,
and equip $[G]$ with the quotient measure,
which is then the Tamagawa measure of volume $2$.


\subsection{Induced representations and Eisenstein series}
\label{sec-5-2}
For a character $\chi$ of $[A]$, we denote by
$\mathcal{I}(\chi)$ the corresponding smooth induction, defined
by analogy to the local case and denoted $\mathcal{I}(s)$
when $\chi = |.|^s$.
The factorizable vectors $f \in \mathcal{I}(\chi)$
have the form
$f = \otimes_\mathfrak{p} f_\mathfrak{p} :
g
\mapsto \prod_\mathfrak{p} 
f_\mathfrak{p}(g_\mathfrak{p})$,
where $f_\mathfrak{p}$ belongs to
$\mathcal{I}(\chi_\mathfrak{p})$
for all $\mathfrak{p}$ and is the normalized
spherical element for almost all finite $\mathfrak{p}$.
As in the local case, we may speak of holomorphic families of vectors.

For $\chi^2 \neq |.|^{\pm 1}$,
the standard intertwining operator
$M : \mathcal{I}(\chi) \rightarrow \mathcal{I}(\chi^{-1})$
is defined on factorizable vectors by
\[
  M f(g)
  =
  \int_{x \in \mathbb{A}}
  f(w n(x) g)
  \, d x
  :=
  \frac{\Lambda(\chi^2,0)}{\Lambda(\chi^2,1)}
  \prod_{\mathfrak{p}}
  \frac{L(\chi_\mathfrak{p}^2,1)}{L(\chi_\mathfrak{p}^2,0)}
  M f_{\mathfrak{p}}(g_\mathfrak{p}),
\]
with the product really a finite product.
The duality between
$\mathcal{I}(\chi)$ and $\mathcal{I}(\chi^{-1})$
is given on factorizable vectors by
the (finite) product
\[
  (f,\tilde{f})
  \mapsto
  \int_{B_\mathbb{A} \backslash G_\mathbb{A}}
  \tilde{f} f 
  :=
  \frac{\xi_F^*(1)}{\xi_F(2)
  }
  \prod_{\mathfrak{p}}
  \frac{\zeta_{F_\mathfrak{p}}(2)}{\zeta_{F_\mathfrak{p}}(1)}
  \int_{B_\mathfrak{p}  \backslash G_\mathfrak{p} }
  \tilde{f}_\mathfrak{p} f_\mathfrak{p}.
\]

Suppose now that $\Re(\chi) > -1/2$ and $\chi^2 \neq |.|$.  We
denote by $\Eis : \mathcal{I}(\chi) \rightarrow C^\infty([G])$
the Eisenstein intertwiner, defined for $\Re(\chi)$ sufficiently
large by the convergent sum
$\Eis(f)(g) := \sum_{\gamma \in B \backslash G} f(\gamma g)$ and
in general by meromorphic continuation.  As $f$ varies in a
holomorphic family, $\Eis(f)$ vanishes as $\chi$ tends to any
quadratic character (i.e., for $\chi^2$ trivial).  For $\chi$
non-quadratic and $f \in \mathcal{I}(\chi)$, we set
$f^* := \Lambda(\chi^2,1) f$ and $\Eis^*(f) := \Eis(f^*)$.
Although $f^*$ itself is not defined when $\chi$ is quadratic,
we may interpret $\Eis^*(f)$ by analytic continuation.
For $f \in \mathcal{I}(\chi)$
varying holomorphically, $\Eis^*(f)$
extends holomorphically to all $\chi$
except for possible simple poles at  $\chi^2 = |.|^{\pm 1}$.

\subsection{Generic representations and Fourier expansions}
\label{sec-5-3}
By a \emph{generic automorphic representation} $\sigma$ (always
of $G$), we mean either an (irreducible) cuspidal automorphic
subrepresentation of $L^2([G])$ or an Eisenstein representation
of the form $\Eis^*(\mathcal{I}(\chi))$ for some character
$\chi$ of $[A]$ with $\chi^2 \neq |.|^{\pm 1}$.
The local components $\sigma_\mathfrak{p}$
of any such representation are of Whittaker type.
We say
that $\sigma$ is \emph{standard} if it is either cuspidal or of
the form $\Eis^*(\mathcal{I}(\chi))$ with $\chi$ unitary.

Let $\sigma$ be a generic automorphic representation.  Each
$\varphi \in \sigma$ admits a Fourier expansion
$\varphi(g) = \varphi_N(g) + \sum_{\alpha \in F^\times}
W_\varphi(a(\alpha) g)$, where $\varphi_N = 0$ unless $\sigma$
is Eisenstein and
$W_\varphi(g) := \int_{x \in \mathbb{A}/F} \varphi(n(x) \psi(-x) 
g) \, d x$.  If $\varphi$ is factorizable, then
we may write
$W_\varphi = \otimes_\mathfrak{p} W_{\varphi,\mathfrak{p}}
: g \mapsto \prod_{\mathfrak{p}}
W_{\varphi,\mathfrak{p}}(g_\mathfrak{p})$, where
$W_{\varphi,\mathfrak{p}}$
belongs to
$\mathcal{W}(\sigma_\mathfrak{p},\psi_\mathfrak{p})$
for all $\mathfrak{p}$
and is the normalized spherical
Whittaker function
for almost all finite $\mathfrak{p}$.
We denote by $\tilde{W}_\varphi$
the Whittaker function defined analogously,
but using the opposite character $\bar{\psi}$ in place of
$\psi$.
If $\sigma$ is irreducible,
then it is isomorphic
to its own contragredient;
the duality
may be given by the pairing
defined for
factorizable vectors by
\[
  (\varphi, \tilde{\varphi})
  \mapsto
  \int_{A_\mathbb{A}}^{\reg}
  \tilde{W}_{\tilde{\varphi}}
  W_{\varphi}
  :=
  \frac{\Lambda^*(\ad \sigma,1)}{\xi(2)}
  \prod_{\mathfrak{p}}
  \frac{\zeta_{F_\mathfrak{p}}(2)}{L(\ad \sigma_\mathfrak{p},1)}
  \int_{A_{\mathfrak{p}}}^{\reg} 
  \tilde{W}_{\tilde{\varphi}} W_{\varphi},
\]
so that the product is really a finite product.

We specialize now to the case
that
$\sigma$ is the Eisenstein representation
$\mathcal{I}(\chi)$
attached to some character
$\chi$ of $[A]$
with $\chi^2 \neq |.|^{\pm }$.
Let $f \in \mathcal{I}(\chi)$
be factorizable,
and write $f =\otimes_{\mathfrak{p}} f_\mathfrak{p}$ as above.
Let $S$ be a finite set of places of $F$
large enough that
for $\mathfrak{p} \notin S$,
$(F_\mathfrak{p},\psi_\mathfrak{p})$ is unramified
and $f_\mathfrak{p}$ is the normalized spherical element.
Then $W_{f_\mathfrak{p}^*}$ is normalized spherical
for $\mathfrak{p} \notin S$
(see \S\ref{sec:normalized-spherical-induced-rep}).
Thus for each
$g \in G_\mathbb{A}$,
we have
$W_{f_\mathfrak{p}^*}(g_\mathfrak{p}) = 1$
for almost all $\mathfrak{p}$.
Thus
\[
  W_{\Eis^*(f)}(g) = \prod_\mathfrak{p}
  W_{f_\mathfrak{p}^*}(g_\mathfrak{p}),
\]
with all but finitely many factors in the product equal to $1$.
(This identity follows first for $\chi$ non-quadratic,
then in general by continuity.)
The full Fourier expansion of the corresponding Eisenstein
series
reads
\begin{equation}\label{eqn:eis-FE}
    \Eis^*(f)(g)
  =
  f^*(g)
  +
  M f^*(g)
  +
  \sum_{\alpha \in F^\times}
  W_{\Eis^*(f)}(a(\alpha) g).
\end{equation}
(When $\chi$ is quadratic,
the sum of the terms $f^*(g) + M f^*(g)$
is interpreted via continuity from the non-quadratic case.)
By standard estimates for local Whittaker functions
(lemma \ref{lem:unif-polyn-type-hecke}),
it follows that for
$g = n(x) a(y) k$ ($x \in \mathbb{A}, y \in \mathbb{A}^\times, k
\in K)$
and fixed $B \geq 0$,
we have
\begin{equation}\label{eq:estimate-eisenstein-series-near-cusp}
  \Eis^*(f)(g) =
  |y|^{1/2} \chi(y) f^*(k) +
  |y|^{1/2} \chi^{-1}(y) M f^*(k)
  + \O_{f,d}(|y|^{-d}),
\end{equation}
and similarly for archimedean derivatives.  Since $w \in G$,
we have
$\Eis^*(f)(a(y)) = \Eis^*(f)(w a(y)) = \Eis^*(f)(a(1/y) w)$, so
the estimate \eqref{eq:estimate-eisenstein-series-near-cusp} is
useful both as $|y| \rightarrow \infty$ and as
$|y| \rightarrow 0$.

We note finally that
if $-1/2 < \Re(\chi) < 1/2$
and $\chi^2 \neq 1$,
then
for $f \in \mathcal{I}(\chi)$ and $\tilde{f} \in
\mathcal{I}(\chi^{-1})$,
we have by \eqref{eqn:local-parseval-f-vs-W-f} that
\begin{equation}\label{eqn:Eis-induced-vs-Whittaker-global-L-2}
    \int_{B_\mathbb{A} \backslash G_\mathbb{A}}
     \tilde{f}^*f^*
  =
  \int_{A_\mathbb{A}}^{\reg}
  \tilde{W}_{\Eis^*(\tilde{f})}
  W_{\Eis^*(f)}.
\end{equation}

\subsection{Hecke integrals}
\label{sec-5-4}
Let $\omega$ be a character of $[A]$.
We denote by
$\mathbb{C}(\omega)$ the corresponding one-dimensional
representation,
and adopt the same conventions as in the local case.

Let $\sigma$ be a generic automorphic representation of $G$,
$\omega$ a character of $[A]$, and $\varphi \in \sigma$.
The \emph{global Hecke integral} $\int_{[A]}^{\reg} \varphi
\omega$
is defined initially for $\Re(\omega)$ large enough
and extends meromorphically.
It unfolds to a product of local Hecke integrals:
if $\varphi \in \sigma$
is factorizable,
then
\begin{equation}\label{eqn:defn-global-hecke-integral-unfold}
  \int_{[A]}^{\reg}
  \varphi \omega
  =
  \int_{A_\mathbb{A}}^{\reg}
  W_{\varphi} \omega
  :=
  \frac
  {
    \Lambda(\sigma \otimes \omega, 1/2)
  }
  {
    \xi_F^*(1)
  }
  \prod_\mathfrak{p}
  \frac{
    \zeta_{F_\mathfrak{p}}(1)
  }
  {
    L(\sigma_\mathfrak{p}  \otimes \omega_\mathfrak{p}, 1/2)
  }
  \int_{A_\mathfrak{p}}^{\reg}
  W_{\varphi,\mathfrak{p}}
  \omega_\mathfrak{p},
\end{equation}
with almost all local factors
equal to $1$.
The ratio
\begin{equation}\label{eqn:global-hecke-integral-ratio}
  \frac{\int_{[A]}^{\reg} \varphi \omega}{\Lambda(\sigma \otimes
    \omega, 1/2)}
\end{equation}
extends to a holomorphic function of $\omega$.
If $\sigma$ is cuspidal,
then the regularization is not needed.
If $\sigma = \Eis^*(\mathcal{I}(\chi))$
is Eisenstein,
then we can define $\int_{[A]}^{\reg} \varphi \omega$
as the meromorphic continuation
from $\omega$ with large real part
of the convergent integral $\int_{[A]} (\varphi - \varphi_N)
\omega$.
If $\varphi = \Eis^*(f)$
where $f = f_\chi \in \mathcal{I}(\chi)$
varies holomorphically with respect to $\chi$,
then the ratio
\eqref{eqn:global-hecke-integral-ratio}
is jointly holomorphic in $\omega$ and $\chi$.

\subsection{Height function}\label{sec:height-function}
Recall the height function
$\htt : [G] \rightarrow \mathbb{R}_{>0}$ defined by
$\htt(g) := \sup |y|$, the supremum taken over all ways to write
$g = \gamma n a(y) k$ with
$\gamma \in G, n \in N_\mathbb{A}, y \in \mathbb{A}^\times, k
\in K$.  The map $\htt$ is proper, its image is bounded below by
a positive quantity (see \cite[\S8]{MR0191899}), and we have
$\int_{[G]}\htt^\alpha < \infty$ precisely when $\alpha < 1$.
For instance, the estimate
\eqref{eq:estimate-eisenstein-series-near-cusp}
implies that if $\chi$ is a unitary character of $[A]$ and
$f \in \mathcal{I}(\chi)$, then $\Eis^*(f)$ is bounded by a
multiple of $\htt^{\alpha}$ for each $\alpha > 1/2$.

\subsection{Spectral decomposition}
\label{sec-5-5}
Let $\Psi_1, \Psi_2,\Psi  \in C_c^\infty([G])$.
The Parseval relation for
$L^2([G])$ may be written
\begin{align*}
  \int_{[G]} \Psi_1 \Psi_2
  &=
    \sum_{\omega^2 = 1}
    \frac{
    \int_{[G]}\Psi_1 \omega^{-1}(\det)
    \int_{[G]}\omega(\det) \Psi_2
    }
    {
    \vol([G])
    }
  \\
  &\quad +
    \sum_{\sigma:\text{cuspidal}}
    \sum_{\varphi \in \mathcal{B}(\sigma)}
    \frac{
    \int_{[G]} \Psi_1 \tilde{\varphi}
    \int_{[G]} \varphi \Psi_2
    }
    {
    \int_{[G]} \tilde{\varphi} \varphi 
    }
  \\
  &\quad
    +
    \frac{1}{2}
    \int_{\chi:\text{unitary}}
    \sum_{f \in \mathcal{B}(\mathcal{I}(\chi))}
    \frac{\int_{[G]} \Psi_1 \Eis(\tilde{f})
    \int_{[G]} \Eis(f) \Psi_2
    }
    {
    \int_{B_\mathbb{A} \backslash G_\mathbb{A}}
     \tilde{f} f
    },
\end{align*}
where
\begin{itemize}
\item $\omega$ runs over the quadratic characters
  of $[A]$,
  with $\omega(\det)(g) := \omega(\det(g))$,
\item
  $\sigma$ runs over the cuspidal automorphic representations,
\item
  $\chi$ runs over the unitary dual of $[A] \cong
  \mathbb{A}^\times/F^\times$
  with respect to the measure dual to the chosen measure on
  $[A]$, and
\item
  $\varphi$ (resp. $f$) runs over an orthonormal basis
  consisting of $K$-isotypic vectors
  for $\sigma$
  (resp. $\mathcal{I}(\chi)$),
  with inner product defined using $L^2([G])$
  if $\sigma$ is cuspidal
  and using $L^2(K)$ if $\sigma = \mathcal{I}(\chi)$.
  We write $\tilde{\varphi}$ (resp. $\tilde{f}$)
  for the corresponding element of the dual basis for
  $\tilde{\sigma} = \sigma$ (resp.  $\mathcal{I}(\chi^{-1})$),
  with the duality normalized via the pairing given in the
  respective denominators.
\end{itemize}
In other words,
the Fourier inversion formula
\begin{align}\label{eqn:fourier-inversion-L-2}
  \Psi(g)
  &=
    \sum_{\omega^2 = 1}
    \frac{
    \int_{[G]}\Psi \omega^{-1}(\det)
    }
    {
    \vol([G])
    }
    \omega(\det(g))
  \\ \nonumber
  &\quad +
    \sum_{\sigma:\text{cuspidal}}
    \sum_{\varphi \in \mathcal{B}(\sigma)}
    \frac{
    \int_{[G]} \Psi \tilde{\varphi}
    }
    {
    \int_{[G]} \tilde{\varphi}\varphi 
    }
    \varphi(g)
  \\ \nonumber
  &\quad
    +
    \frac{1}{2}
    \int_{\chi:\text{unitary}}
    \sum_{f \in \mathcal{B}(\mathcal{I}(\chi))}
    \frac{\int_{[G]} \Psi \Eis(\tilde{f})
    }
    {
    \int_{B_\mathbb{A} \backslash G_\mathbb{A}} \tilde{f} f
    } \Eis(f)(g)
\end{align} 
holds in an $L^2$-sense; the sums in
\eqref{eqn:fourier-inversion-L-2}
need not converge pointwise.

An obvious necessary condition for the inversion formula
\eqref{eqn:fourier-inversion-L-2} to hold pointwise in the
traditional sense is that for each unitary $\chi$ and
each $f \in \mathcal{I}(\chi)$, the
integral $\int_{[G]} \Psi \Eis(\tilde{f})$ converges absolutely.
In our applications, the slightly stronger condition that
$\Psi$ is smooth and every
archimedean derivative of $\Psi$ is bounded by a multiple of
$\htt^{\alpha}$ for some $\alpha < 1/2$ holds;
in particular, $\Psi \in L^2([G])$.
Arguing as in \cite[\S2.6.5]{michel-2009}
using the trace
class property on $L^2([G])$ of some inverse power of the
product of local Laplacians, one can show then that the RHS of
\eqref{eqn:fourier-inversion-L-2} converges uniformly on
compacta to a continuous function.  On the other hand, the
$L^2$ theory implies that the difference between the two sides
of \eqref{eqn:fourier-inversion-L-2} is orthogonal to every
element of $C_c^\infty([G])$, hence vanishes.  Thus
\eqref{eqn:fourier-inversion-L-2} gives a pointwise defined and
convergent expansion of such $\Psi$.

\subsection{Regularized adelic integrals}
\label{sec-9-2}
We recall, following
\cite{MR656029}, \cite[\S4.3]{michel-2009}
and \cite{MR3977317} with some minor modifications,
how to define
a regularized integral 
of certain functions on $[G]$.
By a \emph{finite} function
$\phi : N_\mathbb{A} A \backslash G_\mathbb{A} \rightarrow
\mathbb{C}$ we mean a smooth function on $G_\mathbb{A}$ that
is finite in the sense of \S\ref{sec-9-1}
with respect to the left translation action of $[A]$;
equivalently, $\phi$
may be written in the form
\begin{equation}\label{eqn:phi-decompose-finite-N-A-K}
  \phi(n a(y) k) = \sum_{i=1}^n \chi_i(y) \log^{m_i-1}|y|
  \mathcal{K}_i(k)
\end{equation}
for some characters $\chi_i$ on $\mathbb{A}^\times/F^\times$,
positive integers $m_i$, and smooth functions $\mathcal{K}_i$ on
$K$.  We may and shall assume that the $\mathcal{K}_i$ are not
identically zero and that each pair $(\chi_i,m_i)$ shows up at
most once, so that the decomposition
\eqref{eqn:phi-decompose-finite-N-A-K} of $\phi$ is unique up to
rearrangement.  We say that a smooth function
$\Psi : [G] \rightarrow \mathbb{C}$ is of \emph{controlled
  increase} if we can find a finite function $\phi$ on
$N_\mathbb{A} A \backslash G_\mathbb{A}$ so that for
$d \geq 0$ and all
$n \in N_\mathbb{A}, k \in K$ and $y \in \mathbb{A}^\times$ with
$|y| \geq 1$, the estimate
$\Psi(n a(y) k) = \phi(n a(y) k) + \O(|y|^{-d})$ holds, together
with the analogous estimates involving all archimedean
derivatives.  The function $\phi$ is then uniquely determined;
we refer to it as the \emph{asymptotic part} of $\Psi$.
The characters $\chi_i$
that arise in the
decomposition \eqref{eqn:phi-decompose-finite-N-A-K}
are the exponents of $\phi$.
We will refer to these simply as the exponents
of $\Psi$.
For example, if $\chi$ is a character of $[A]$ with
$\chi^2 \neq |.|^{\pm}$ and $f$ is a nonzero element of
$\mathcal{I}(\chi)$, then the Eisenstein series $\Eis^*(f)$ is
of controlled increase with exponents
$|.|^{1/2} \chi$ and $|.|^{1/2} \chi^{-1}$.

Suppose $\Psi$ has controlled increase and that 
the exponent $|.|$ does not occur.  Following
\cite{MR656029} and \cite[\S4.3]{michel-2009},
we may then
define the regularized integral $\int_{[G]}^{\reg} \Psi$.
It may be characterized
as the unique $G$-invariant functional on the space of such $\Psi$
that extends integration on the subspace of integrable
functions.  It may be defined using convolution or truncation
as in \S\ref{sec-9-1}, or  Eisenstein series as in
\cite{MR656029} and \cite[\S4.3.5]{michel-2009}.
We briefly recall the last of these definitions.  Split
the asymptotic part $\phi$ of $\Psi$ as the sum
$\phi_{>1/2} + \phi_{\leq 1/2}$ of two terms corresponding to
the exponents with real part indicated by the subscript.  Using $\phi_{>1/2}$, we may construct
a linear combination of (derivatives of)
Eisenstein series $\mathcal{E}$ for which the difference
$\Psi - \mathcal{E}$ has controlled increase with exponents of
real part $\leq 1/2$.
That difference is then integrable,
and we take $\int_{[G]}^{\reg} \Psi := \int_{[G]}(\Psi - \mathcal{E})$.

\subsection{Regularized spectral decomposition}
\label{sec-9-3}
If $\Psi$ is a smooth function on $[G]$ of controlled increase
with exponents of real part strictly less than $1/2$, then
$\Psi$ and its archimedean derivatives are bounded by constant
multiples of $\htt^{\alpha}$ for some $\alpha <1/2$,
so the expansion
\eqref{eqn:fourier-inversion-L-2} is pointwise defined and
convergent, uniformly on compacta.

Suppose now that $\Psi$ has controlled increase and that
each exponent $\chi$ satisfies $\Re(\chi) \neq 1/2$,
hence in particular
$\chi^2 \neq |.|$.  We may then split the asymptotic part $\phi$
as a sum $\phi_{>1/2} + \phi_{<1/2}$ and define
$\mathcal{E}$ using $\phi_{>1/2}$ as before.  The difference
$\Psi - \mathcal{E}$ then has controlled increase with exponents
of real part $< 1/2$, hence admits a pointwise spectral
expansion, uniformly on compacta, with coefficients given by the
convergent integrals
$\int_{[G]} (\Psi - \mathcal{E}) \omega^{-1}(\det)$ and
$\int_{[G]} (\Psi - \mathcal{E}) \Eis(\tilde{f})$ and the
analogous integrals involving cusp forms.  In fact, it follows
from representation-theoretic considerations
that
$\int_{[G]}^{\reg} \mathcal{E} \omega^{-1}(\det)
= \int_{[G]}^{\reg} \mathcal{E} \Eis(\tilde{f}) = 0$,
so that the spectral expansion of $\Phi$
may be written
\begin{align}\label{eqn:fourier-inversion-L-2-reg}
  \Psi(g)
  &=
    \mathcal{E}(g)
    +
    \sum_{\omega^2 = 1}
    \frac{
    \int_{[G]}^{\reg}\Psi \omega^{-1}(\det)
    }
    {
    \vol([G])
    }
    \omega(\det(g))
  \\ \nonumber
  &\quad
    +
    \sum_{\sigma:\text{cuspidal}}
    \sum_{\varphi \in \mathcal{B}(\sigma)}
    \frac{
    \int_{[G]} \Psi \tilde{\varphi}
    }
    {
    \int_{[G]} \varphi \tilde{\varphi}
    }
    \varphi(g)
  \\ \nonumber
  &\quad
    +
    \frac{1}{2}
    \int_{\chi:\text{unitary}}
    \sum_{f \in \mathcal{B}(\mathcal{I}(\chi))}
    \frac{\int_{[G]}^{\reg} \Psi \Eis(\tilde{f})
    }
    {
    \int_{B_\mathbb{A} \backslash G_\mathbb{A}} \tilde{f} f
    } \Eis(f)(g).
\end{align}
Suppose now moreover
that $\Psi$ is orthogonal to
the one-dimensional subrepresentations of $L^2([G])$,
so that the sum over $\omega$ may be omitted.
The resulting expansion then involves only the standard \emph{generic}
automorphic representations.
It will be convenient
to rewrite that expansion in terms of Whittaker norms.
Using \eqref{eqn:Eis-induced-vs-Whittaker-global-L-2}
and the formula
$\int_{[G]}
\tilde{\varphi} \varphi = 2\int_{A_{\mathbb{A}}}^{\reg} 
\tilde{W}_{\tilde{\varphi}} W_{\varphi}$
for cuspidal $\sigma$
(see \cite[\S3.2.2]{nelson-variance-II}
or \cite[Lem 2.2.3]{michel-2009})
gives the required identity,
which we restate in full for convenience:
\begin{theorem}\label{thm:reg-spect-pointwise}
  Let $\Psi \in C^\infty([G])$
  be of controlled increase,
  with each exponent
  $\chi$ satisfying
  $\Re(\chi) \neq 1/2$
  and $\chi^2 \neq |.|$,
  and orthogonal to
  the one-dimensional subrepresentations
  $\mathbb{C} \omega(\det)$ of $L^2([G])$.
  Let $\mathcal{E}$ denote the linear combination
  of (derivatives) of Eisenstein series
  attached to the
  summand $\phi_{>1/2}$ of the asymptotic part $\phi$
  of $\Psi$ obtained by collecting terms
  involving characters of real part $> 1/2$.
  We then have a pointwise defined and normally convergent
  expansion
  \begin{align*}
    \Psi(g)
    -
    \mathcal{E}(g)
    &=
      \frac{1}{2} 
      \sum_{\sigma:\text{cuspidal}}
      \sum_{\varphi \in \mathcal{B}(\sigma)}
      \frac{
      \int_{[G]} \Psi \tilde{\varphi}
      }
      {
      \int_{A_{\mathbb{A}}}^{\reg} \tilde{W}_{\tilde{\varphi}} W_{\varphi} 
      }
      \varphi(g)
    \\
    &\quad
      +
      \frac{1}{2}
      \int_{\chi:\text{unitary}}
      \sum_{f \in \mathcal{B}(\mathcal{I}(\chi))}
      \frac{\int_{[G]}^{\reg} \Psi \Eis^*(\tilde{f})
      }
      {
      \int_{A_{\mathbb{A}}}^{\reg}
      \tilde{W}_{\Eis^*(\tilde{f})} W_{\Eis^*(f)} 
      }
      \Eis^*(f)(g).
  \end{align*}
\end{theorem}
We will often abbreviate the RHS of the above
decomposition to
\begin{equation}\label{eqn:spectral-decomp-reg-generic}
  \int_{\sigma:\text{generic}}
  \sum_{\varphi \in \mathcal{B}(\sigma)}
  \frac{
    \int_{[G]}^{\reg} \Psi \tilde{\varphi}
  }
  {
    \int_{A_{\mathbb{A}}}^{\reg}  \tilde{W}_{\tilde{\varphi}}W_{\varphi}
  }
  {
    \varphi(g)
  }
\end{equation}
or further to simply
\begin{equation}
  \int_{\sigma:\text{generic}}
  \Psi_{\sigma}.
\end{equation}
Note that the measures implicit in the integrals over $\sigma$
in \eqref{eqn:fourier-inversion-L-2-reg} and
\eqref{eqn:spectral-decomp-reg-generic} differ on the cuspidal
spectrum by the factor $1/2$.

It is useful to describe the rate of convergence a bit more
precisely.  First
of all, for $\Psi$ as in Theorem
\ref{thm:reg-spect-pointwise}, the constant term
$\Psi_N(g) := \int_{x \in \mathbb{A}/F} \Psi(n(x) g) \, d x$
enjoys the normally convergent expansion
\begin{equation}
  \Psi_N =
  \mathcal{E}_N +
  \int_{\sigma:\text{generic}}
  \Psi_{\sigma,N},
\end{equation}
and thus likewise
\begin{equation}
  \Psi = \Psi_N + \mathcal{E} - \mathcal{E}_N
  +
  \int_{\sigma:\text{generic}} (\Psi_\sigma - \Psi_{\sigma,N}).
\end{equation}
This last integral converges
uniformly near the cusp;
more precisely,
for $g = a(y) k$ with $y \in \mathbb{A}^\times, k \in K$,
we have
\begin{equation}\label{eqn:estimate-Psi-sigma-minus-constant-term}
  \Psi_\sigma(g) - \Psi_{\sigma,N}(g)
  \ll_{\Psi,d}
  \frac{\min(|y|^{-1/2-\eps}, |y|^{-d})}{C(\sigma)^d}
\end{equation}
for fixed $d \geq 0$.  The bound
\eqref{eqn:estimate-Psi-sigma-minus-constant-term} may be
established first for $B$ sufficiently negative, then using
\cite[\S2.6.6]{michel-2009} and arguing as in
\S\ref{sec:integration-parts} or \cite[Lem 3.5.2]{michel-2009}
to save with respect to $C(\sigma)$.  For Eisenstein $\sigma$,
we estimate each term of the Fourier expansion using
lemma \ref{lem:unif-polyn-type-hecke}
or \cite[(3.2.3)]{michel-2009} and
\cite[\S2, (S1d)]{michel-2009}.  For cuspidal $\sigma$, we apply
the same argument for large $|y|$, while for small $|y|$ we
invoke the crude $L^\infty$-norm bound for $\Psi_{\sigma}$
following from \cite[\S2, (S2a), (S3b)]{michel-2009}.


\subsection{Rankin--Selberg integrals}
\label{sec-5-6}
For generic
automorphic representations
$\sigma_1, \sigma_2$ and characters $\chi$ of $[A]$ of large
enough real part, we define for
$\varphi_1 \in \sigma_1, \varphi_2 \in \sigma_2$ and
$f \in \mathcal{I}(\chi)$ the global Rankin--Selberg integral
$\int_{[G]}^{\reg} \varphi_1 \varphi_2 \Eis^*(f)$.  If either
$\sigma_1$ or $\sigma_2$ is cuspidal, then the integral
converges absolutely.  The Eisenstein case requires
regularization, for the details of which we refer to
\S\cite[\S4.4]{michel-2009}.  The important feature for our
purposes is that in either case, the integral unfolds as the
Eulerian integral
$\int_{N_\mathbb{A} \backslash G_\mathbb{A}}^{\reg}
W_{\varphi_1} \tilde{W}_{\varphi_2} f^*$
(interpreted, 
using the local unramified calculation of
\eqref{eqn:normalized-local-RS},
by analogy to
\eqref{eqn:defn-global-hecke-integral-unfold}).
It follows then from the corresponding
local results of \S\ref{sec-4-5} that as $f$ varies in
a holomorphic family, the ratio
\[
  \frac{\int_{[G]}^{\reg} \varphi_1 \varphi_2 \Eis^*(f)
  }{
    \Lambda(\sigma_1 \otimes \sigma_2 \otimes \chi, 1/2)
  }
\]
extends holomorphically
to all $\chi$.

\part{The basic identity}\label{part:basic-identity}
Here we formulate Theorem \ref{thm:basic-spectr-ident} precisely
(see \S\ref{sec:summary-main-results}) and give the proof.

\section{Local invariant functionals\label{sec:local-inv-func}}
\label{sec-10}
Let $F$ be a local field,
with notation as in \S\ref{sec-4}.
Let $s =  (s_1,s_2,s_3) \in \mathbb{C}^3$.
We will eventually
take the limit as $s$ approaches the origin.

For each such $s$, we obtain a representation
\begin{equation}\label{eqn:Is1-Is2-Cs3}
    \mathcal{I}(s_1) \otimes \mathcal{I}(s_2) \otimes
  \mathbb{C}(s_3)
\end{equation}
of $G \times G \times A$.
(The notation $\mathcal{I}(s_j)$
is defined in \S\ref{sec:intro-representations},
$\mathbb{C}(s_3)$
in \S\ref{sec-4-3}.)
In the archimedean case, we take the completed tensor product,
as in \S\ref{sec:representation-pi-g};
in either case,
this representation identifies
with
the space $\mathcal{I}(s_1) \otimes \mathcal{I}(s_2)$ via the map
$f_1 \otimes f_2 \otimes 1 \mapsto f_1 \otimes f_2$,
which may in turn be defined as the
space of smooth functions $f : G \times G \rightarrow
\mathbb{C}$
satisfying $f(n(x_1) a(y_1) g_1, n(x_2) a(y_2) g_2)
= |y_1|^{1/2+s_1} |y_2|^{1/2+s_2}$.

We may speak, as in \S\ref{sec:families-vectors}, of holomorphic
families of vectors
$f[s] \in \mathcal{I}(s_1) \otimes \mathcal{I}(s_2) \otimes
\mathbb{C}(s_3)$ indexed by $s$ in an open subset on
$\mathbb{C}^3$.

We will define two
families of $A$-invariant functionals on the representation \eqref{eqn:Is1-Is2-Cs3},
corresponding to the ``strong Gelfand triple''
\cite{MR2373356}
formed by the two
sequences $G \times G \geq G \geq A$ and
$G \times G \geq A \times A \geq A$ of strong Gelfand pairs.
The special case $s = 0$ was sketched
in \S\ref{sec-6-3}.

\subsection{$G \times G \geq G \geq A$}
\label{sec-10-1}
\subsubsection{Definition and estimates}\label{sec:defn-ell-sigma-s-makes-sense}
Let $\sigma$ be a generic irreducible representation of $G$,
realized in its Whittaker model $\mathcal{W}(\sigma,\psi)$.  
For almost all $s$, we
aim to define a continuous $A$-invariant map
\[
  \ell_{\sigma,s} :
  \mathcal{I}(s_1) \otimes \mathcal{I}(s_2)
  \otimes \mathbb{C}(s_3) \rightarrow \mathbb{C}
\]
by the formula
\begin{equation}\label{eqn:defn-l-omega-s}
  f_1 \otimes f_2
  \mapsto
  \sum_{W \in \mathcal{B}(\sigma)}
  \frac{\int_{N \backslash G}^{\reg} \tilde{W}  W_{f_1}  f_2 \,
    \int_A^{\reg} W |.|^{s_3}
  }
  {
    \int_A^{\reg} \tilde{W} W
  },
\end{equation}
with the integrals interpreted as in \S\ref{sec-4-3},
\S\ref{sec-4-5} and the sum as in \S\ref{sec-6-3}.  More
precisely, we write $\sum_{W \in \mathcal{B}(\sigma)}$ as
shorthand for
$\sum_{f \in \mathcal{B}(\sigma), W := W_f \tilde{W} :=
  \tilde{W}_{\tilde{f}}}$, where the orthonormal basis
$\mathcal{B}(\sigma)$ is as defined in
\S\ref{sec:norms-whit-type-reps}, and $\tilde{f}$ runs over the
corresponding dual basis for $\sigma$
(resp. $\mathcal{I}(\chi^{-1})$) for $\sigma$ square-integrable
(resp. $\sigma =\mathcal{I}(\chi)$), with the duality prescribed
by the denominator of \eqref{eqn:defn-l-omega-s} (or in the
induced case, equivalently by integration over $B \backslash G$
or $K$ -- see \eqref{eqn:local-parseval-f-vs-W-f}).

We pause to make sense of this definition.
By the local theory of Hecke and
Rankin--Selberg integrals recalled in
\S\ref{sec:local-hecke-basic} and \S\ref{sec:local-rs-basic},
together with the formula
\eqref{eqn:multiplicativity-L-factors-wrt-induction}
for $L$-factors of induced
representations,
we see
that the integrals
in the numerator of \eqref{eqn:defn-l-omega-s} are defined away
from the poles of the numerator of
\[
  \mathcal{L}(\sigma,s)
  :=
  \frac{
    L(\sigma, 1/2 + s_1 + s_2)L(\sigma, 1/2 - s_1 + s_2) L(\sigma,
    1/2 + s_3)
  }
  {
    L(\sigma \times \sigma, 1)
  }.
\]
In the archimedean case, these integrals moreover define
continuous functionals.

It remains to make sense of the sum in
\eqref{eqn:defn-l-omega-s} over $W$.
We suppose henceforth that $\sigma \in G^\wedge_{\gen}$ is
$\vartheta$-tempered for some fixed $\vartheta > 0$, that $s$
lies in a fixed compact subset of $\mathbb{C}^3$
and
that $s$ is some fixed positive distance away from any pole of
$\mathcal{L}(\sigma,s)$.
\begin{lemma}
  Let $f_1 \in \mathcal{I}(s_1), f_2 \in \mathcal{I}(s_2)$.
  Then
  for each fixed $d$,
  \begin{equation}\label{eqn:sum-of-sobolev-norms-of-W-weighted-by-RS-int}
    \sum_{W \in \mathcal{B}(\sigma)}
    \left\lvert
      \frac{\int_{N \backslash G}^{\reg} \tilde{W}  W_{f_1}  f_2
      }
      {
        \int_A^{\reg} \tilde{W} W
      }
    \right\rvert
    \mathcal{S}_d(W)
    \ll
    C(\sigma)^{-d} \mathcal{S}(f_1) \mathcal{S}(f_2) q^{\O(1)}.
  \end{equation}
\end{lemma}
\begin{proof}
  Let $W, \tilde{W}$ be as in the sum.
  Let $f_3 \in \mathcal{B}(\sigma)$
  be such that $W = W_{f_3}$,
  and let $\tilde{f}_3$ denote the corresponding
  dual basis element,
  belonging to $\sigma$ in the square-integrable
  case and to $\mathcal{I}(\chi^{-1})$ if $\sigma =
  \mathcal{I}(\chi)$,
  so that $\tilde{W} = \tilde{W}_{\tilde{f}_3}$.
  We 
  recall from
  \eqref{eqn:upper-bound-Mellin-transform-W}
  that
  \begin{equation}
    \int_{N \backslash G}^{\reg}
    \tilde{W} W_{f_1} f_2
    \ll
    \mathcal{S}_d(f_1)
    \mathcal{S}_d(f_2) \mathcal{S}_d(\tilde{f}_3)  q^{\O(1)}.
  \end{equation}
  We may strengthen this estimate
  by integrating by parts with respect
  to $\sigma$ (\S\ref{sec:integration-parts}),
  giving for fixed $d,d'$ with $d'$ large in terms of $d$ that
  \begin{equation}
    \int_{N \backslash G}^{\reg}
    \tilde{W} W_{f_1} f_2
    \ll
    C(\sigma )^{-d}
    \mathcal{S}_{d'}(f_1)
    \mathcal{S}_{d'}(f_2)
    \mathcal{S}_{-2 d}(\tilde{f}_3)
    q^{\O(1)}.
  \end{equation}
  By the construction of $\mathcal{B}(\sigma)$, we have
  $\int_{A}^{\reg} \tilde{W} W = \int_{B \backslash G} \tilde{f}_3 f_3$,
  $\mathcal{S}_{-2 d}(\tilde{f}_3) = N_{\sigma \nu}^{-2 d}$ and
  $\mathcal{S}_d(f_3) = N_{\sigma \nu}^d$.  
  Thus
  the LHS of
  \eqref{eqn:sum-of-sobolev-norms-of-W-weighted-by-RS-int}
  is majorized by
  \[
    C(\sigma)^{-d}
    \mathcal{S}_{d'}(f_1)
    \mathcal{S}_{d'}(f_2) q^{\O(1)} \sum_{\nu \in K^\wedge}
    \dim(\sigma^\nu) N_{\sigma \nu}^{-d}.
  \]
  By the uniform trace class property
  \eqref{eq:sum-N-sigma-nu-neg-d-0},
  we may assume that $d$ is large enough that this last sum over $\nu$ is
  $\O(1)$.
  This completes the proof.
\end{proof}
We may thus define a linear map
\begin{equation}
  \rho_{\sigma,s} : \mathcal{I}(s_1) \otimes \mathcal{I}(s_2) \rightarrow \mathcal{W}(\sigma,\psi)
\end{equation}
\begin{equation}
  f_1 \otimes f_2
  \mapsto
  \sum_{W \in \mathcal{B}(\sigma)}
  \frac{\int_{N \backslash G}^{\reg} \tilde{W}  W_{f_1}  f_2
  }
  {
    \int_A^{\reg} \tilde{W} W
  }
  W
\end{equation}
satisfying $\mathcal{S}_d(\rho_{\sigma,s}(f)) \ll \mathcal{S}(f)$
for each fixed $d$.
This map may be characterized
by the property:
for any $\tilde{W} \in \mathcal{W}(\sigma,\bar{\psi})$,
\begin{equation}
  \int_{A}^{\reg}
  \tilde{W}
  \rho_{\sigma,s}(f_1 \otimes f_2)
  =
  \int_{N \backslash G}^{\reg} \tilde{W}  W_{f_1}  f_2.
\end{equation}
From this property and the diagonal $G$-invariance of the
Rankin--Selberg integral, we see that $\rho_{\sigma,s}$ is
$G$-equivariant.
We now compose $\rho_{\sigma,s}$
with the Hecke integral
\begin{equation}
  \mathcal{W}(\sigma,\psi) \otimes \mathbb{C}(s_3)
  \rightarrow \mathbb{C} 
\end{equation}
\begin{equation}
  W \mapsto \int_A^{\reg} W |.|^{s_3}
\end{equation}
to obtain the required continuous $A$-invariant map
$\ell_{\sigma,s}$.  We may use
\eqref{eqn:sum-of-sobolev-norms-of-W-weighted-by-RS-int} to
estimate $\rho_{\sigma,s}$; by combining with the estimate
\eqref{eqn:upper-bound-Mellin-transform-W} for local Hecke
integrals, we deduce that $\ell_{\sigma,s}$ satisfies for each
fixed $d$ the estimate
\begin{equation}\label{eq:crude-estimate-for-ell-sigma-s}
  \ell_{\sigma,s}(f)
  \ll
  C(\sigma)^{-d}
  \mathcal{S}(f) q^{\O(1)}.
\end{equation}

\subsubsection{}\label{sec:local-ell-sigma-s-where-defined-win}
It follows from the above discussion that if $\sigma$ is
$\vartheta$-tempered (\S\ref{sec:bounds-towards-raman}), then
$\ell_{\sigma,s}$ is defined whenever the real parts of
$s_1 + s_2, - s_1 + s_2$ and $s_3$ are at least
$-1/2 + \vartheta$.  In particular, if $\sigma$ is
$1/6$-tempered, then $\ell_{\sigma,s}$ is defined for all $s$
satisfying
\begin{equation}\label{eqn:s1-s2-s3-one-sixth-initial-bounds}
  |\Re(s_1)| < 1/6, \quad |\Re(s_2)| < 1/6, \quad \Re(s_3) > -1/6.
\end{equation}
Since $7/64 < 1/6$ (see \S\ref{sec:bounds-towards-raman}),
the poles of $\mathcal{L}(\sigma,s)$
will not play a significant role in our analysis.

\subsubsection{}
Recall from \eqref{eq:defn-f-asterisk} the notation
$f_i^* := \zeta_F(1 + 2 s_i) f_i$.
We define a normalized variant of $\ell_{\sigma,s}$
by the formula
\begin{equation}
  \ell_{\sigma,s}^*(f_1 \otimes f_2)
  :=
  \ell_{\sigma,s}(f_1^* \otimes f_2^*).
\end{equation}
By the corresponding property
of local Hecke and Rankin--Selberg integrals,
the ratio $\ell_{\sigma,s}(f)/\mathcal{L}(\sigma,s)$ extends to
a holomorphic function of $s$, hence the ratio
$\ell_{\sigma,s}^*(f)/\mathcal{L}(\sigma,s)$ to a meromorphic
function of $s$.
\begin{lemma}\label{lem:unram-calc-ell-sigma-s}
  If $(F,\psi)$ is unramified and $f = f_1 \otimes f_2$
  with $f_1, f_2$ normalized spherical
  (\S\ref{sec:normalized-spherical-induced-rep}),
  then
  the ratio $\ell_{\sigma,s}^*(f)/\mathcal{L}(\sigma,s)$
  vanishes unless $\sigma$
  is unramified, in which case it evaluates to $1$.
\end{lemma}
\begin{proof}
  The vanishing is clear when $\sigma$ is ramified:
  the image of $f$ in $\sigma$
  is then $K$-invariant,
  but $\sigma$ contains no $K$-invariant vectors.
  Suppose that
  $\sigma$ is unramified.  We may assume that
  $\mathcal{B}(\sigma)$ contains a $K$-invariant vector $W$,
  with $\tilde{W}$ also $K$-invariant.  The remaining basis
  elements then do not contribute to the sum defining
  $\ell_{\sigma,s}(f_1 \otimes f_2)$.  Since the quantities
  $W(1)$ and $\tilde{W}(1)$ are nonzero (see
  \S\ref{sec:normalized-spherical-induced-rep}) and the ratio in question is
  invariant under scaling $W$ and $\tilde{W}$, we may assume for
  computational convenience that $W(1) = \tilde{W}(1) = 1$.
  Then by the unramified calculations
  recorded in
  \S\ref{sec:normalized-spherical-induced-rep},
  \S\ref{sec-4-3}
  and
  \S\ref{sec-4-5},
  we have
  $\int_A^{\reg} \tilde{W} W = L(\ad \sigma, 1)/\zeta_F(2)$
  and
  $\int_A^{\reg} W |.|^{s_3} = L(\sigma, 1/2 + s_3)/ \zeta_F(1)$
  and
  $W_{f_2}^*(1) = 1$ and
  $\int_{N \backslash G}^{\reg} \tilde{W}  W_{f_1^*} f_2^* =
  L(\sigma \otimes \mathcal{I}(s_1), 1/2 + s_2) / \zeta_F(2)$.
  The conclusion follows upon noting that
  $L(\sigma \otimes \mathcal{I}(s_1), 1/2 + s_2) L(\sigma, 1/2 +
  s_3)/ \zeta_F(1) L(\ad \sigma,1) = \mathcal{L}(\sigma,s)$.
\end{proof}

\subsubsection{Holomorphy of
  $\ell_{\sigma,s}(f)$}\label{sec:holom-ell_s-sf}
In this section we verify that $\ell_{\sigma,s}(f)$ varies
holomorphically with respect to both $\sigma$ and $s$ away from
the poles of $\mathcal{L}(\sigma,s)$.  We first record a
technical lemma related to the holomorphy of local Bessel
functions with respect to their index.
\begin{lemma}\label{lem:holom-bessel-index}
  Fix $\nu \in K^\wedge$.
  For $\chi \in A^\wedge$,
  write $\sigma(\chi)$ for the generic irreducible
  subquotient of $\mathcal{I}(\chi)$
  and
  $\mathcal{B}(\sigma(\chi)^\nu)$ for the set of all
  $W \in \mathcal{B}(\sigma(\chi))$ having $K$-type $\nu$.
  Then for $g_1, g_2 \in G$,
  the sum
  \begin{equation}\label{eqn:bessel-sigma-nu}
    \sum_{W \in \mathcal{B}(\sigma(\chi)^\nu)}
    \frac{\tilde{W}(g_1) W(g_2)}{\int_A ^{\reg} \tilde{W} W}
  \end{equation}
  varies holomorphically with respect to $\chi$.
\end{lemma}
\begin{proof}
  Fix a component $\mathcal{X}$ of $A^\wedge$, i.e., a coset of
  the subgroup $\{|.|^s : s \in \mathbb{C} \}$ of unramified
  characters.  Let $\chi_K$ denote the common restriction to
  $A \cap K$ of elements $\chi$ of $\mathcal{X}$.  Let
  $\{\phi_i\}$ be an orthonormal basis for the $\nu$-isotypic
  subspace of $\Ind_{A \cap K}^K(\chi_K) \subseteq L^2(K)$.  For
  each $\chi \in \mathcal{X}$, let $f_i[\chi]$
  (resp. $\tilde{f}_i[\chi]$) denote the element of
  $\mathcal{I}(\chi)$ (resp. $\mathcal{I}(\chi^{-1})$) whose
  restriction to $K$ is $\phi_i$ (resp. the complex conjugate of
  ${\phi}_i$).  For $g_1, g_2 \in G$, the sum
  \begin{equation}\label{eqn:bessel-induced-chi}
    \sum_{i}
    \tilde{W}_{\tilde{f}_i[\chi]}(g_1) W_{f_i[\chi]}(g_2)
  \end{equation}
  is independent of the choice of basis, and varies
  holomorphically with $\chi$.  We claim that
  \eqref{eqn:bessel-induced-chi} and \eqref{eqn:bessel-sigma-nu}
  are equal.  The claim implies the required conclusion.

  The claim
  clearly holds when $\mathcal{I}(\chi)$ is
  irreducible, so suppose otherwise.
  Then
  for some choice of sign $\pm$,
  $\sigma(\chi)$ is
  isomorphic to  (see \cite{MR3889963})
  \begin{itemize}
  \item the quotient $\mathcal{I}(\chi^{\mp})/ L$, where
    $L$ denotes the finite-dimensional kernel of
    $\tilde{f} \mapsto \tilde{W}_{\tilde{f}}$, and also to 
  \item the
    subspace $L^\perp \subseteq \mathcal{I}(\chi^{\pm})$.
  \end{itemize}
  Suppose
  for instance that $\chi^{\pm} = \chi$.  
  We may
  assume the basis $\{\phi_i\}$
  chosen so that $\tilde{f}_1[\chi], \dotsc, \tilde{f}_{\dim(L)}[\chi] \in L$
  and
  $f_{\dim(L)+1}[\chi],f_{\dim(L)+2}[\chi], \dotsc \in L^\perp$.
  Then the
  $i=1..\dim(L)$ summands
  in \eqref{eqn:bessel-induced-chi} vanish,
  while the remaining summands give \eqref{eqn:bessel-sigma-nu}.
  The claim follows.  A
  similar argument applies if $\chi^{\pm} = \chi^{-1}$.
\end{proof}

\begin{lemma}
  Let $U$ be an open subset of $\mathbb{C}^3$.
  Let $f[s] \in \mathcal{I}(s_1) \otimes \mathcal{I}(s_2)$ be a
  holomorphic family of vectors defined for $s \in U$.  Let
  $\mathcal{D} \subseteq G^\wedge_{\gen} \times U$ denote the
  complement of the closure of the set of all pairs $(\sigma,s)$
  for which $\mathcal{L}(\sigma,s)$ is infinite.  Then
  $\ell_{\sigma,s}(f[s])$ defines a holomorphic function on
  $\mathcal{D}$.
\end{lemma}
Here the holomorphy means as in
\S\ref{sec-4-1} that
on the indicated domain $\mathcal{D}$,
\begin{itemize}
\item $\ell_{\sigma,s}(f[s])$ is holomorphic in $s$ for each
  $\sigma$, and
\item $\ell_{\sigma(\chi),s}(f[s])$ is holomorphic in
  $(\chi,s)$.
\end{itemize}
\begin{proof}
  We may assume that $U = \mathbb{C}^3$ and that $f[s]$ is a
  flat family.  The Sobolev norms $\mathcal{S}_{d}(f[s])$ are
  locally bounded in $s$, so the proof of the estimate
  \eqref{eq:crude-estimate-for-ell-sigma-s} shows that the sum
  defining $\ell_{\sigma,s}(f[s])$ converges locally uniformly.
  Since each term varies holomorphically in $s$, we obtain the
  required holomorphy of $\ell_{\sigma,s}(f[s])$ with respect to
  $s$.

  It remains to verify that $\ell_{\sigma(\chi),s}(f[s])$ is
  holomorphic in $(\chi,s)$.
  Let us fix
  $\chi_0 \in A^\wedge$
  and write $\chi = \chi_0 |.|^w$
  with respect to the local coordinate $w \in \mathbb{C}$,
  and fix $\nu \in K^\wedge$.
  For $s \in \mathbb{C}^3$
  and $w \in \mathbb{C}$,
  we define $\Phi_0(s,w)$
  like $\ell_{\sigma(\chi_0 |.|^w),s}(f_1 \otimes f_2)$,
  but restricting
  the sum to those $W$ having
  $K$-type $\nu$.
  In view of the locally uniform convergence
  of the sum defining $\ell_{\sigma,s}$,
  it is enough to verify
  that $\Phi_0(s,w)$ is holomorphic
  away from the poles
  of $\mathcal{L}(\sigma(\chi_0 |.|^w), s)$.

  To that end, we choose a sufficiently large real number $c$,
  denote by $\Omega$ the set of all pairs $(s,w)$ for which
  the real parts of $s_1,s_2,s_3,w$ have magnitude less than
  $c$, choose a polynomial $P(s,w)$ so that
  $P(s,w) \mathcal{L}(\mathcal{I}(\chi_0 |.|^w),s)$ is holomorphic on
  $\Omega$, and set $\Phi(s,w) := P(s,w) \Phi_0(s,w)$.  It is
  enough then to verify that $\Phi(s,w)$ is holomorphic on
  $\Omega$.

  We know by \S\ref{sec:local-hecke-basic} and
  \S\ref{sec:local-rs-basic} that on $\Omega$, the map
  $s \mapsto \Phi(s,w)$ is holomorphic for each $w$.  Moreover,
  if $\Re(s_2)$ and $\Re(s_3)$ are large enough in terms of
  $|\Re(s_1)|$ and $|\Re(w)|$, then the integral defining
  $\Phi_0(s,w)$ converges absolutely, and so lemma
  \ref{lem:holom-bessel-index} implies that $\Phi(s,w)$ is
  holomorphic in both variables.  Using the local functional
  equations as in the proofs of lemmas
  \ref{lem:unif-polyn-type-hecke} and \ref{lem:polyn-bound-rs},
  we deduce that $\Phi(s,w)$ is holomorphic whenever
  $|\Re(s_2)|$ and $|\Re(s_3)|$ are large enough in terms of
  $|\Re(s_1)|$ and $|\Re(w)|$.  In the archimedean case, we see
  moreover by integrating by parts with respect to
  $\mathcal{I}(s_2)$ and $\mathcal{I}(s_3)$
  (\S\ref{sec:integration-parts})
  that $\Phi(s,w)$
  decays rapidly in vertical strips with respect to $s_2$ and
  $s_3$.  We may thus use Cauchy's theorem to express
  $\Phi(s,w)$ for general $s_2,s_3$ as an alternating sum of
  four contour integrals in which each of $\Re(s_2)$ and
  $\Re(s_3)$ is sufficiently positive or negative.  We thereby
  deduce the holomorphy of $\Phi(s,w)$ for general arguments
  $s_2, s_3$ from the case in which those arguments have large
  real parts.
\end{proof}

\subsubsection{Holomorphy of pre-Kuznetsov
  weights}\label{sec:holomorphy-of-pre-K-wts}
Let
$h : G_{\gen}^\wedge \rightarrow \mathbb{C}$ be a pre-Kuznetsov
weight (\S\ref{sec:pre-kuzn-weights}) with kernel
$\phi \in C_c^\infty(G)$.
We verify here that the formula defining
$h$ converges absolutely
and that $h$ is holomorphic
(in the sense of \ref{sec-4-1}).

Write $J_{\sigma,\nu}(g_1,g_2)$ for
the sum \eqref{eqn:bessel-sigma-nu}.
By regrouping the definition, we have
$h(\sigma) = \sum_{\nu \in K^\wedge} h^\nu(\sigma)$, where
$h^\nu(\sigma) := \int_{g \in G} J_{\sigma,\nu}(g,1) \phi(g)$.
By lemma \ref{eqn:bessel-sigma-nu},
we see that each
$h^\nu$ is holomorphic, so it suffices to verify for every compact
subset $\Sigma$ of $G^\wedge_{\gen}$ that
$\sup_{\sigma \in \Sigma} \sum_{\nu \in K^\wedge}
|h^\nu(\sigma)| < \infty$.  We have
$h_\nu(\sigma) = \int_{g \in G} J_{\sigma,\nu}(g,1)
\phi_\nu(g)$, where $\phi_\nu$ denotes the $\nu$-isotypic
component of $\phi$ under right translation.  In the
non-archimedean case, we have $\phi_\nu = 0$ for all but
finitely many $\nu$, so the required absolute convergence
follows from the continuity of $J_{\sigma,\nu}$.  In the
archimedean case, the smoothness of $\phi$ implies that
$\|\phi_\nu \|_{L^\infty(G)} \ll_d (1 + c_\nu)^{-d}$ for each
fixed $d \geq 0$, where $c_\nu \geq 0$ denotes the Casimir
eigenvalue.  By expanding $g = n(x) a(y) k$ in Iwasawa
coordinates, we reduce to showing that
$J_{\sigma,\nu}(a(y) k, 1) \ll (1 + c_\nu)^{\O(1)}$ for all
$k \in K$ and all $y$ belonging to a fixed compact subset of
$F^\times$.  This last estimate follows (in stronger form) from
lemma \ref{lem:unif-polyn-type-hecke}.

An identical argument gives the locally uniform
convergence
of the sums \eqref{eqn:V-sig-proj-defn}.

\subsection{$G \times G \geq A \times A \geq A$}
\label{sec-10-2}
Let $\omega$ be a character of $A$.
For almost all $s$, we may define
an $A$-invariant map
\[
  \ell_{\omega,s}
  : \mathcal{I}(s_1) \otimes \mathcal{I}(s_2)
  \otimes \mathbb{C}(s_3) \rightarrow \mathbb{C} 
\]
\begin{equation}\label{eq:ell-omega-s-initial-defn}
  f_1 \otimes f_2
  \mapsto
  \int_A^{\reg} W_{f_1} \omega 
  \int_A^{\reg} W_{f_2} \omega^{-1} |.|^{s_3}.
\end{equation}
The definition makes sense: for given $f_1, f_2$, the RHS of
\S\ref{eq:ell-omega-s-initial-defn} is defined away from the
poles of
the numerator of
\begin{align*}
  \mathcal{L}(\omega,s)
  &:= 
    \frac{
    L(\mathcal{I}(s_1) \otimes \omega,1/2)
    L(\mathcal{I}(s_2) \otimes \omega^{-1},1/2 + s_3)
    }
    {
    \zeta_F(1)^2
    }
  \\
  &=
    \frac{ \prod_{\pm}
    L(\omega,1/2 \pm s_1)
    L(\omega^{-1},1/2 \pm s_2 + s_3)
    }{
    \zeta_F(1)^2
    }
\end{align*}
(see \S\ref{sec:local-hecke-basic}).  Moreover, in the
archimedean case, the integrals appearing on the RHS of
\eqref{eq:ell-omega-s-initial-defn} define continuous functionals
on $\mathcal{I}(s_1)$ and $\mathcal{I}(s_2)$, hence their
product extends to a continuous functional on the (completed)
tensor product.
In general, the Hecke integral estimate
\eqref{eqn:upper-bound-Mellin-transform-W}
and the reduction to pure tensors of
\S\ref{sec:reduct-pure-tens}
implies that if
$\sigma$  is $\O(1)$-tempered,
$s$ lies in a fixed compact subset of $\mathbb{C}^3$
and $s$ is some fixed positive distance
away from any pole of $\mathcal{L}(\omega,s)$,
then for each fixed $d$,
\begin{equation}\label{eq:crude-ell-omega-s-estimate}
  \ell_{\omega,s}(f) \ll C(\omega)^{-d}
\mathcal{S}(f).
\end{equation}

As we did for $\ell_{\sigma,s}$, we define the normalized variant
\begin{equation}
  \ell_{\omega,s}^*(f_1 \otimes f_2)
  :=
  \ell_{\omega,s}(f_1^* \otimes f_2^*).
\end{equation}
The ratio $\ell_{\omega,s}(f)/\mathcal{L}(\omega,s)$ extends to
a holomorphic function of $s$, the ratio
$\ell_{\omega,s}^*(f)/\mathcal{L}(\omega,s)$ to a meromorphic
one, and we have the expected unramified calculation, which
follows immediately from the corresponding calculation of
\S\ref{sec:local-hecke-basic}:
\begin{lemma}
  If $(F,\psi)$ is unramified and $f = f_1 \otimes f_2$ with
  $f_1,f_2$ normalized spherical, then the ratio
  $\ell_{\omega,s}^*(f) / \mathcal{L}(\omega,s)$
  vanishes
  unless $\omega$ is unramified, in which case it evaluates to
  $1$.
\end{lemma}

We note finally that if $\omega$ is
unitary, then $\ell_{\omega,s}$ is
defined whenever each of $\pm s_1$ and $\pm s_2 + s_3$ have real
parts at least $-1/2$,
as
follows from \eqref{eqn:s1-s2-s3-one-sixth-initial-bounds}.


\subsection{Variation with respect to the additive character}\label{sec:vari-with-resp}
It
is often convenient to assume that $\psi$ has a particular form
(e.g., unramified when $F$ is non-archimedean), so we record
here how the above definitions vary with $\psi$.
For $b \in F^\times$, set $\psi^b(x) := \psi(b x)$
and
let $\ell_{\sigma,s}^b$, $\ell_{\omega,s}^b$ be defined as
above, but using $\psi^b$.
We verify readily that
\begin{equation}
  \ell_{\sigma,s}^b
  = |b|^{3/2 - s_1 + s_2 - s_3}
  \ell_{\sigma,s}, \quad 
  \ell_{\omega,s}^b
  = |b|^{1 + s_1 + s_2 - s_3}
  \ell_{\omega,s}.
\end{equation}

\section{Global invariant functionals\label{sec:global-inv-func}}
\label{sec-11}
Let $F$ be a number field,
with accompanying notation as in \S\ref{sec-5}.
Let $s \in \mathbb{C}^3$.
Then
$\mathcal{I}(s_1) \otimes \mathcal{I}(s_2) \otimes
\mathbb{C}(s_3)$ is an irreducible automorphic representation of
$G_\mathbb{A} \times G_\mathbb{A} \times A_\mathbb{A}$, given by
the restricted tensor product of the analogous local
representations.  We will define two families of
$A_\mathbb{A}$-invariant functionals, the global analogues of
those of \S\ref{sec:local-inv-func}, and record their
factorizations into local functionals.

\subsection{$G \times G \geq G \geq A$}
\label{sec-11-1}
Let $\sigma$ be a generic automorphic representation.
For almost all $s$,
we define a functional
\[
  \ell_{\sigma,s} : \mathcal{I}(s_1) \otimes \mathcal{I}(s_2)
  \otimes 
  \mathbb{C}(s_3) \rightarrow \mathbb{C}
\]
\[
  f_1 \otimes f_2
  \mapsto \sum _{\varphi \in 
    \mathcal{B}(\sigma)}
  \frac{
    \int _{[G]}^{\reg} \Eis(f_1^*) \Eis(f_2^*) \tilde{\varphi}
    \int _{[A]}^{\reg}
    \varphi |.|^{s_3}
  }{
    \int_{A_{\mathbb{A}}}^{\reg}
    \tilde{W}_{\tilde{\varphi}}
    W_{\varphi}
  },
\]
with the integrals interpreted as above.
Here
$\mathcal{B}(\sigma)$
is obtained by tensoring 
the local bases
defined in \S\ref{sec:norms-whit-type-reps},
and as $\varphi$ traverses
$\mathcal{B}(\sigma)$
we let $\tilde{\varphi}$
traverse the  corresponding dual basis,
with the duality normalized by the pairing given
in the denominator.

If $\sigma = \Eis^*(\mathcal{I}(\eta))$ with
$\eta$ quadratic, set $\mathcal{L}(\sigma, s) := 0$;
otherwise, set
\[
  \mathcal{L}(\sigma,s)
  :=
  \frac{
    \Lambda(\sigma, 1/2 + s_1 + s_2)
    \Lambda (\sigma, 1/2 - s_1 + s_2) \Lambda (\sigma,
    1/2 + s_3)
  }
  {
    \Lambda^*(\sigma \times \sigma, 1)
  }.
\]
\begin{remark}
If $\sigma$ is Eisenstein,
say
$\sigma = \Eis^*(\mathcal{I}(\chi))$,
then the quantity
\begin{equation}
  \Lambda^*(\sigma \times \sigma, 1)
  = \begin{cases}
    \xi_F^*(1)^4 & \text{ if $\chi$ is quadratic,}\\
    \xi_F^*(1)^2
    \Lambda(\chi^2, 1) & \text{ otherwise}
  \end{cases}
\end{equation}
does not vary continuously with respect to
$\chi \in [A]^\wedge$.
(Compare with \cite[\S2.2.2]{michel-2009}.)
The set of discontinuities consists of
the quadratic $\chi$, hence has measure zero.  If $\chi = \eta
|.|^{i t}$ with
$\eta$
quadratic and
$t$ is a small nonzero real number, then
\begin{equation}
  \Lambda^*(\sigma \times \sigma, 1)
  \asymp
  t^{-2}.
\end{equation}
These considerations motivate the
definition of $\mathcal{L}(\sigma,s)$,
which varies meromorphically with $\chi$
for $\sigma = \Eis^*(\mathcal{I}(\chi))$.
\end{remark}

We define $\mathcal{L}^{(S)}(\sigma,s)$
like $\mathcal{L}(\sigma,s)$,
but with the Euler products
over places not in $S$.

By the
unfolding of global Hecke and Rankin--Selberg integrals noted in
\S\ref{sec-5-4} and \S\ref{sec-5-6}, we see that for
factorizable factors,
\begin{equation}
  \ell_{\sigma,s}(f_1 \otimes f_2)
  =
  \mathcal{L}(\sigma,s)
  \prod_{\mathfrak{p}}
  \mathcal{L}(\sigma_\mathfrak{p},s)^{-1}
  \ell_{\sigma_\mathfrak{p},s}^*(f_{1 \mathfrak{p}} \otimes f_{2 \mathfrak{p}}),
\end{equation}
with the product really a finite product.
As $f$ varies
holomorphically, the ratio
$\ell_{\sigma,s}(f)/\mathcal{L}(\sigma,s)$ thus extends
holomorphically
to all $s$.

\subsection{$G \times G \geq A \times A \geq A$}
\label{sec-11-2}
Let $\omega$ be a character of $[A]$.
For almost all $s$, we may define
a functional
\[
  \ell_{\omega,s} : \mathcal{I}(s_1) \otimes \mathcal{I}(s_2)
  \otimes \mathbb{C}(s_3)
  \rightarrow \mathbb{C} 
\]
\[
  f_1 \otimes f_2
  \mapsto
  \int_{[A]}^{\reg} \Eis^*(f_1) \omega 
  \int_{[A]}^{\reg} \Eis^*(f_2) \omega^{-1} |.|^{s_3},
\]
which unfolds
and then factors
on factorizable vectors
as the (finite) product
\[
  \ell_{\omega,s}(f_1 \otimes f_2)
  =
  \mathcal{L}(\omega,s) \prod_{\mathfrak{p}}
  \mathcal{L}(\omega_\mathfrak{p},s)^{-1}
  \ell_{\omega_\mathfrak{p},s}^*(f_{1 \mathfrak{p}} \otimes f_{2 \mathfrak{p}}),
\]
where
\[
  \mathcal{L}(\omega,s)
  := 
  \frac{ \prod_{\pm}
    \Lambda (\omega,1/2 \pm s_1)
    \Lambda (\omega^{-1},1/2 \pm s_2 + s_3)
  }
  {
    (\xi_F(1)^*)^2
  }.
\]
As $f$ varies holomorphically,
the ratio $\ell_{\omega,s}(f)/\mathcal{L}(\omega,s)$
thus extends
holomorphically to all
$\omega \in [A]^\wedge$ and $s$.

\section{Decompositions of global periods}
\label{sec-12}
We consider here
$s \in \mathbb{C}^3$ satisfying
the condition
\begin{equation}\label{eqn:assumptions-s-2}
  \eps_1 s_1 + \eps_2 s_2 + \eps_3 s_3
  \notin \{0, \pm 1/2, \pm 1\}
  \text{ for all }
  0 \neq \eps\in \{-1,0,1\}^3.
\end{equation}
We deduce from the estimate
\eqref{eq:estimate-eisenstein-series-near-cusp}
for the Eisenstein series
that the expression
\begin{equation}\label{eqn:prod-eis-and-y-s3}
  \Eis^*(f_1)(a(y)) \Eis^*(f_2)(a(y)) |y|^{s_3}
  \quad
  (y \in \mathbb{A}^\times/F^\times)
\end{equation}
may be
approximated as $|y| \rightarrow \infty$
(resp. $|y| \rightarrow 0$)
by a linear combination
of the characters $|y|^{1 \pm s_1 \pm s_2 + s_3}$
(resp.
$|y|^{-1 \pm s_1 \pm s_2 + s_3}$).
Since
$\pm s_1 \pm s_2 \pm s_3 \neq 1$,
the trivial character never occurs,
so \eqref{eqn:prod-eis-and-y-s3} defines a strongly regularizable
function on $[A]$ (\S\ref{sec-9-1}).
We obtain
an $A_\mathbb{A}$-invariant functional
\[
  \ell_s : \mathcal{I}(s_1) \otimes \mathcal{I}(s_2) \otimes
  \mathbb{C}(s_3)
  \rightarrow \mathbb{C} 
\]
\[
  f_1 \otimes f_2 \mapsto \int_{[A]}^{\reg} \Eis^*(f_1)
  \Eis^*(f_2) |.|^{s_3},
\]
which we proceed to decompose in two ways.


\begin{remark}
  The regularized integral
  over $[A]$ considered above may be
  defined concretely as the absolutely convergent integral
  \begin{equation}
    \int_{y \in [A]}
    \left(
      \int_{u \in \mathbb{A}^\times}
      \Eis^*(f_1) \Eis^*(f_2) (a(y u))
      \, d \nu(u)
    \right)
    |y|^{s_3}
    \, \frac{d y}{|y|}
  \end{equation}
  for any finite measure $\nu$ on $\mathbb{A}^\times$ whose Mellin
  transform vanishes at the characters
  $|.|^{\pm 1 \pm s_1 \pm s_2}$ (with all $2^3$ possible sign
  combinations).
\end{remark}

\subsection{$G \times G \geq G \geq A$}
\label{sec-12-1}
We phrase some estimates below in terms of ``generic complex
lines in $\mathbb{C}^3$ containing the origin.''  In each case,
one could take the line
$\{ (10^{-6} s_0, 10^{-12} s_0, 10^{- 18} s_0) \in \mathbb{C}^3
: s_0 \in \mathbb{C} \}$, for instance.  Those estimates may be
formulated alternatively as bounds for the total polar multiplicity
at $0 \in \mathbb{C}^3$ of certain meromorphic functions.

\begin{theorem}\label{thm:GG-G-A}
  Let $s \in \mathbb{C}^3$
  satisfy $|\Re(s_i)| < 1/6$ for $i=1,2,3$
  and
  \eqref{eqn:assumptions-s-2}.
  Then
  \[
    \ell_s
    =
    \int_{\sigma:\text{generic}}
    \ell_{\sigma,s}
    +
    \sum_{i=1}^{7}
    \ell^{\deg}_{i,s},
  \]
  where the integral
  is taken over
  standard generic automorphic
  representations $\sigma$ as in \eqref{eqn:spectral-decomp-reg-generic}
  and the ``degenerate maps'' $\ell^{\deg}_{i,s}$ 
  are meromorphic families of functionals with the
  following
  properties:
  \begin{itemize}
  \item For $1 \leq i \leq 4$, $\ell^{\deg}_{i,s}$ factors
    $\Delta G \times A$-equivariantly through
    $\mathcal{I}(1/2 \pm s_1 \pm s_2) \otimes \mathbb{C}(s_3)$.
  \item
    $\ell^{\deg}_{5,s}$ is $N_\mathbb{A}$-invariant.
  \item For $i=6,7$, $\ell^{\deg}_{i,s}$ factors
    $\Delta G \times A$-equivariantly through
    $\mathcal{I}(\pm (s_3 - 1/2)) \otimes \mathbb{C}(s_3)$.
  \end{itemize}
  Moreover,
  for any holomorphic family $f[s] \in \mathcal{I}(s_1) \otimes
  \mathcal{I}(s_2)
  \otimes \mathbb{C}(s_3)$
  defined for $s$ near zero,
  we have
  \begin{equation}\label{eqn:polar-bounds-for-degen-func}
    \ell^{\deg}_{i,s}(f[s]) =
    \begin{cases}
      \O(|s|^{-4}) & \text{ for }i=1,2,3,4, \\
      \O(|s|^{-5}) & \text{ for }i=6,7
    \end{cases}
  \end{equation}
  for small $s \in \mathbb{C}^3$ in a generic complex line
  containing the origin.
\end{theorem}

\begin{proof}
  Let $f = f_1 \otimes f_2 \in \mathcal{I}(s_1)
  \otimes \mathcal{I}(s_2)$.
  Set $\Psi := \Eis^*(f_1) \Eis^*(f_2)$.
  The regularizing Eisenstein series
  as in \S\ref{sec-9-2}
  is given by
  \begin{equation}
    \mathcal{E} := \Eis(f_1^* \cdot f_2^*) + \Eis(M f_1^* \cdot f_2^*) +
    \Eis(f_1^* \cdot M f_2^*) + \Eis(M f_1^* \cdot M f_2^*).
  \end{equation}
  By Theorem \ref{thm:reg-spect-pointwise},
  we have the normally convergent pointwise expansion
  \[
    \Psi - \mathcal{E} =
    \int_{\sigma:\text{generic}}
    \Psi_{\sigma},
  \]
  say.
  We obtain
  \[
    \ell_s(f)
    =
    \sum_{i=1}^4 \ell^{\deg}_{i,s}(f) +
    \int_{[A]}^{\reg}
    \int_{\sigma:\text{generic}}
    \Psi_{\sigma} |.|^{s_3},
  \]
  where the $\ell^{\deg}_{i,s}(f)$ are given by the regularized
  integral over $[A]$ of the terms in the definition of
  $\mathcal{E}$.  The functionals $\ell^{\deg}_{i,s}$ 
  have the required factorization property.  The estimate
  \eqref{eqn:polar-bounds-for-degen-func} for $i=1..4$ follows
  by writing, e.g.,
  \[
    \Eis(f_1^* \cdot f_2^*) =
    \frac{
      \xi_F(1 + 2 s_1) \xi_F(1 + 2 s_2)
    }
    {
      \xi_F(2 + 2 s_1 + 2 s_2)
    }
    \Eis^*(f_1 \cdot f_2),
  \]
  and recalling from
  \S\ref{sec-5-4}  that
  \[
    \int_{[A]}^{\reg} \Eis^*(f_1 \cdot f_2) |.|^{s_3}/\prod_{\pm }
    \xi_F(1/2 \pm (1/2 + s_1 + s_2) + s_3)
  \]
  is holomorphic.

  Continuing to assume
  that $|\Re(s_1)|, |\Re(s_2)| < 1/6$,
  we now change our assumptions
  temporarily by supposing
  that $\Re(s_3) > 1/2$.
  Then for any standard generic automorphic representation
  $\sigma$,
  we have
  \begin{equation}\label{eqn:ell-sigma-expand-with-large-s3}
    \ell_{\sigma,s}(f) = \int_{[A]}
    (\Psi_\sigma -
    \Psi_{\sigma,N}) |.|^{s_3}.
  \end{equation}
  By \eqref{eqn:estimate-Psi-sigma-minus-constant-term},
  the integrand in \eqref{eqn:ell-sigma-expand-with-large-s3}
  is bounded for each fixed $d \geq 0$
  by  $C(\sigma)^{-d}$ times a fixed convergent integrand.
  Taking
  $d$
  sufficiently large that
  $\int_{\sigma} C(\sigma)^{-d} < \infty$
  (see \cite[\S2.6.5]{michel-2009}),
  we deduce
  by interchanging summation with integration that
  \begin{equation}\label{eqn:ell-with-five-degens}
    \ell_s(f)
    =
    \sum_{i=1}^5 \ell^{\deg}_{i,s}(f)
    + \Phi(s)
  \end{equation}
  where, for $\Re(s_3) > 1/2$,
  \begin{equation}\label{eqn:Phi-defn-integral-generic-stuff}
    \Phi(s) := \int_{\sigma:\text{generic}} \ell_{\sigma,s}(f)
  \end{equation}
  and
  \begin{equation}
    \ell^{\deg}_{5,s}(f)
    :=
    \int_{[A]}^{\reg} \int_{\sigma:\text{generic}} \Psi_{\sigma,N} |.|^{s_3}.
  \end{equation}
  We may evaluate the functional $\ell^{\deg}_{5,s}$ explicitly
  using truncation and that $\Psi_{\sigma,N} = 0$ unless
  $\sigma$ is Eisenstein, but it suffices for our purposes to
  note as claimed that this functional is
  $N_\mathbb{A}$-invariant.

  We now freeze the variables $s_1, s_2$
  and view the above identity as one of meromorphic functions
  of $s_3$ defined
  initially for $\Re(s_3) > 1/2$.
  We aim to meromorphically continue
  $\Phi$ to the range
  $-1/6  < \Re(s_3) < 1/2$
  and to verify that a modified form of the identity  \eqref{eqn:Phi-defn-integral-generic-stuff}
  holds there.
  
  We expand
  \begin{equation}\label{eqn:expand-first-I}
    \Phi(s)
    =
    \frac{1}{2}
    \sum_{\sigma:\text{cuspidal}}
    \ell_{\sigma,s}(f)
    +
    \frac{1}{2}
    \int_{\chi:\text{unitary}}
    \ell_{\mathcal{I}(\chi),s}(f).
  \end{equation}
  For cuspidal $\sigma$, the $L$-function
  $\mathcal{L}(\sigma,s)$ is entire; by the convexity bound and
  the estimate \eqref{eq:crude-estimate-for-ell-sigma-s}, we see
  that $\ell_{\sigma,s}(f)$ decays faster than any power of
  $C(\sigma)$, locally uniformly in $s$, and
  extends to an entire function of $s$.  Analogous assertions hold for
  the individual Eisenstein contributions
  $\ell_{\mathcal{I}(\chi), s}(f)$ except when $\chi$ is of the
  form $|.|^t$, in which case we may encounter poles in the
  indicated range.  Indeed, we may write
  $\ell_{\mathcal{I}(t),s}(f) = \Xi(t,s) h(t,s)$, where
  \[
    \Xi(t,s) :=
    \prod_{\pm}
    \frac{
      \xi_F(\tfrac{1}{2} + s_1 + s_2 \pm
      t) \xi_F(\tfrac{1}{2} - s_1 + s_2 \pm t) \xi_F(\tfrac{1}{2} +
      s_3 \pm t)
    }
    {
    \xi_F(1 \pm 2 t)
    }
  \]
  and $h$ is an entire function
  defined by a finite product of normalized
  local integrals.
  The poles of $\Xi(t,s)$
  with $s_1, s_2$ as usual, $\Re(s_3) > -1/6$
  and $\Re(t) = 0$
  are at
  $t = \pm (s_3 - 1/2)$.

  We are led to the problem of determining the meromorphic
  continuation of the integral
  \[
    I(s) :=
    \int_{t \in i \mathbb{R}}
    \ell_{\mathcal{I}(t),s}(f)
    \, \frac{d t}{2 \pi i},
  \]
  defined initially for $\Re(s_3) > 1/2$.  Let us first do this
  under the assumption of GRH
  (to avoid possible poles of $\xi_F(1 \pm 2 t)^{-1}$)
  and working formally.
  We fix $\eps > 0$ sufficiently small, and consider $s_3$ of real part
  in the interval $(1/2,1/2+\eps)$.  We shift the contour to
  $\Re(t) = 2 \eps$, passing a pole at $t = s_3 - 1/2$.  The
  resulting integral has no poles for $s_3$ of real part in
  $(1/2-\eps,1/2+\eps)$.  We obtain in this way the analytic
  continuation of $I(s)$ to $1/2 - \eps < \Re(s) < 1/2 + \eps$.
  We suppose next that $1/2 - \eps < \Re(s) < 1/2$ and shift the
  contour back to $\Re(t) = 0$, passing a pole at
  $t = 1/2 - s_3$.  We obtain in this way the analytic
  continuation of $I(s)$ to $s_3$ of real part greater than
  $-1/2$, and for $s_3$ of real part in the interval
  $(-1/2,1/2)$ the formula
  \begin{equation}\label{eq:formula-I-of-s-after-shift}
    I(s)
    =
    \sum_{\pm}
    \pm
    \res_{t \rightarrow \pm(1/2 - s_3)}
    \ell_{\mathcal{I}(t),s}(f)
    +
    \int_{t \in i \mathbb{R}}
    \ell_{\mathcal{I}(t),s}(f)
    \, \frac{d t}{2 \pi i}.
  \end{equation}

  The argument of the previous paragraph was not quite rigorous,
  because we assumed GRH and did not justify the contour shifts.  To
  give a rigorous argument, we fix $T \geq 1$ sufficiently large
  and restrict to $s_3$ with $|\Im(s_3)| < T/2$, say.
  By the prime number theorem,
  we may choose
  $\eps$ small enough that shifting $t$ from $i \mathbb{R}$ to
  the piecewise-linear contour $\mathcal{C}$ with endpoints
  $- i \infty, - i T, - i T + 2 \eps, i T + 2 \eps , i T, i\infty$
  does not encounter any zeroes of $\xi_F(1 - 2 t)$.
  We then argue as before but with $\{t : \Re(t) = 2 \eps\}$
  replaced
  by $\mathcal{C}$.

  Combining the formula \eqref{eq:formula-I-of-s-after-shift}
  with our earlier remarks gives the required meromorphic continuation
  of $\Phi$ and the identity, for $s_3$ of real part in $(-1/6,1/2)$,
  \begin{equation}\label{eqn:expand-second-I}
    \Phi(s)
    =
    \int_{\sigma:\text{generic}}
    \ell_{\sigma,s}(f)
    +
    \sum_{i=6,7}
    \ell^{\deg}_{i,s}(f),
  \end{equation}
  where the $\ell^{\deg}_{i,s}$
  have the required factorization properties.
  Since the estimate
  \begin{equation}
    \res_{t=\pm (s_3 - 1/2)} \Xi(t,s) \ll |s|^{-5}
  \end{equation}
  holds for small $s \in \mathbb{C}^3$ in a generic complex line
  containing the origin,
  the estimate \eqref{eqn:polar-bounds-for-degen-func}
  follows for $i=6,7$.

  We obtain the required identity by combining
  \eqref{eqn:ell-with-five-degens}
  and \eqref{eqn:expand-second-I}.
\end{proof}

\subsection{$G \times G \geq A \times A \geq A$}
\label{sec-12-2}
\begin{theorem}\label{thm:GG-AA-A}
  Let $s \in \mathbb{C}^3$
  satisfy $|\Re(s_i)| < 1/6$ for $i=1,2,3$
  and
  \eqref{eqn:assumptions-s-2}.
  Then
  \[
    \ell_s
    = \int_{\omega:\text{unitary}}
    \ell_{\omega,s}
    + \sum_{i=8}^{15}
    \ell^{\deg}_{i,s},
  \]
  where the integral is taken
  over unitary characters $\omega$ of $[A]$
  as in  \eqref{eq:measure-on-dual-omega}
  and the degenerate functionals
  are given on $f = f_1 \otimes f_2$
  by
  \begin{equation}\label{deg:8}
    f_1^*(1)
    \int_{A_{\mathbb{A}}}^{\reg}
    W_{\Eis^*(f_2)} |.|^{1/2+s_1+s_3},
  \end{equation}
  \begin{equation}\label{deg:9}
    M f_1^*(1)
    \int_{A_{\mathbb{A}}}^{\reg}
    W_{\Eis^*(f_2)} |.|^{1/2-s_1+s_3},
  \end{equation}
  \begin{equation}\label{deg:10}
    f_1^*(w) \int_{A_\mathbb{A}}^{\reg}
    W_{\Eis^*(f_2)} |.|^{-1/2-s_1+s_3},
  \end{equation}
  \begin{equation}\label{deg:11}
    M f_1^*(w) \int_{A_\mathbb{A}}^{\reg}
    W_{\Eis^*(f_2)} |.|^{-1/2+s_1+s_3},
  \end{equation}
  together with the analogous quantities
  obtained by swapping
  $(f_1,s_1)$ with $(f_2,s_2)$.
\end{theorem}
\begin{proof}
  Set $\varphi_i(y) := \Eis^*(f_i)(a(y))$.
  By estimate \eqref{eq:estimate-eisenstein-series-near-cusp}
  for the Eisenstein series,
  we see that
  $\varphi _i$
  admits the finite expansions
  \begin{equation}
    \varphi_i(y)
    \sim
    \begin{cases}
      \phi_i^\infty(y)
      & \text{ as $|y| \rightarrow \infty$,}
      \\
      \phi_i^0(y)
      & \text{ as $|y| \rightarrow 0$}
    \end{cases}
  \end{equation}
  where
  \begin{align*}
    \phi_i^\infty(y)  &:= |y|^{1/2+s_i} f_i^*(1)
                        + |y|^{1/2-s_i} M f_i^*(1),
    \\    \phi_i^0(y) &:=
                        |y|^{-1/2-s_i} f_i^*(w)
                        + |y|^{-1/2+s_i} M f_i^*(w).
  \end{align*}
  Since regularizable finite functions
  have vanishing regularized integral,
  we have
  \begin{align}\label{eqn:subtract-off-asymptotics-at-infi}
    \int_{\mathbb{A}^\times/F^\times }^{\reg}
    \varphi_1 \varphi_2 |.|^{s_3}
    &= 
      \int_{\mathbb{A}^\times/F^\times}^{\reg}
      (\varphi_1 - \phi_1^{\infty})
      (\varphi_2 - \phi_2^\infty ) |.|^{s_3}
    \\ \nonumber
    &\quad
      +
      \int_{\mathbb{A}^\times/F^\times}^{\reg}
      \phi_1^{\infty} \varphi_2 |.|^{s_3}
      +
      \int_{[A]}^{\reg}
      \phi_2^{\infty} \varphi_1 |.|^{s_3}.
  \end{align}
  The second and third terms on the RHS of
  \eqref{eqn:subtract-off-asymptotics-at-infi}
  contribute the degenerate terms \eqref{deg:8},
  \eqref{deg:9}
  listed above plus their analogues with $(f_1,s_1)$ and $(f_2,s_2)$ swapped.
  The first term
  may be defined for (say) $\Re(s_3) \geq 20$
  as an absolutely convergent integral,
  then in general by meromorphic continuation.
  Similarly,
  the Mellin transform
  (cf. \S\ref{sec-5-4})
  \begin{equation}
    \hat{\varphi}_i(\omega)
    :=
    \int_{y \in \mathbb{A}^\times/F^\times}^{\reg}
    \varphi_i(\omega) \omega(y)
    =
    \int_{A_{\mathbb{A}}}^{\reg}
    W_{\Eis^*(f_i)}
    \omega
  \end{equation}
  may be defined for (say) $\Re(\omega) \geq 10$ via the
  absolutely convergent integral of
  $(\varphi_i - \phi_i^{\infty})\omega$.
  By Mellin inversion,
  we have
  \[
    \varphi_1(y)
    -
    \phi_1^{\infty}(y)
    = \int_{\omega : \Re(\omega) = 10}
    \hat{\varphi}_1(\omega)
    \omega^{-1}(y).
  \]
  (The
  crude estimate \eqref{eq:crude-ell-omega-s-estimate}
  and the convexity bound for the $L$-values
  give adequate decay at infinity
  for $\hat{\varphi}_i$
  to justify this expansion and subsequent contour shifts.)
  Thus for $\Re(s_3) = 20$,
  \begin{equation}
    \int_{\mathbb{A}^\times/F^\times}^{\reg}
    (\varphi_1 - \phi_1^{\infty})
    (\varphi_2 - \phi_2^\infty ) |.|^{s_3}
    =
    \int_{\omega : \Re(\omega) = 10}
    \hat{\varphi}_1(\omega)
    \hat{\varphi}_2(\omega^{-1} |.|^{s_3}).
  \end{equation}
  We shift the $\omega$-contour to $\Re(\omega) = 0$,
  passing poles
  at $\omega = |.|^{1/2 \pm s_1}$
  whose residues contribute
  the degenerate terms
  \eqref{deg:10},
  \eqref{deg:11}.
  We then take $\Re(s_3)$ nearly as small as we can
  without passing a pole of the integrand
  and shift to
  $\Re(\omega) = \eps$,
  passing two more poles that contribute
  the remaining degenerate terms.
  We then take $\Re(s_3)$ close to $0$
  and shift back to $\Re(\omega) = 0$,
  giving the required identity.
  
  
\end{proof}

\subsection{Summary}
\label{sec:summary-main-results}
\begin{theorem}\label{thm:basic-identity-summary}
  Suppose $s \in \mathbb{C}^3$
  satisfies the hypotheses of
  Theorem \ref{thm:GG-AA-A}.
  Let $f = \otimes f_{\mathfrak{p}} \in \mathcal{I}(s_1) \otimes \mathcal{I}(s_2)$
  be a factorizable vector.
  Let $S$ be a finite set of places
  of $F$
  such that for each $\mathfrak{p} \notin S$,
  we have that $(F_\mathfrak{p}, \psi_\mathfrak{p})$
  is unramified and
  $f_\mathfrak{p} = f_{1 \mathfrak{p}} \otimes f_{2
    \mathfrak{p}}$
  with $f_{1 \mathfrak{p}}, f_{2 \mathfrak{p}}$
  normalized spherical.
  Then
  \begin{align}
    \label{eqn:basic-moment-identity-2}
    &\int_{\substack{
      \sigma:\text{generic}, \\
    \text{unram. outside $S$}
    }}
    \mathcal{L}^{(S)}(\sigma,s)
    \prod_{\mathfrak{p} \in S}
    \ell_{\sigma_\mathfrak{p},s}^*(f_\mathfrak{p})
    + \sum_{i=1}^{7} \ell^{\deg}_{i,s}(f_\mathfrak{p})
    \\ \nonumber
    &\quad
      \int_{\substack{
      \omega:\text{unitary}, \\
    \text{unram. outside $S$}
    }}
    \mathcal{L}^{(S)}(\omega,s)
    \prod_{\mathfrak{p} \in S}
    \ell_{\omega_\mathfrak{p},s}^*(f)
    +
    \sum_{i=8}^{15} \ell^{\deg}_{i,s}(f),
  \end{align}
  where $\mathcal{L}^{(S)}(\sigma,s)$
  and $\mathcal{L}^{(S)}(\omega,s)$
  denote the corresponding partial Euler products.
\end{theorem}
\begin{proof}
  We decompose $\ell_s(f)$ via Theorems \ref{thm:GG-G-A} and
  \ref{thm:GG-AA-A} and unfold as in
  \S\ref{sec:global-inv-func}.
\end{proof}
\begin{corollary}\label{cor:basic-identity-summary}
  Let $f = \otimes f_{\mathfrak{p}} \in \mathcal{I}(0) \otimes \mathcal{I}(0)$
  be a factorizable vector,
  and let
  $S$ be as in Theorem \ref{thm:basic-identity-summary}.
  Then
  the difference
  \begin{equation}\label{eq:difference-two-moments}
    \int_{\substack{
        \sigma:\text{generic}, \\
        \text{unram. outside $S$}
      }}
    \mathcal{L}^{(S)}(\sigma,0)
    \prod_{\mathfrak{p} \in S}
    \ell_{\sigma_\mathfrak{p},0}^*(f_\mathfrak{p})
    -
    \int_{\substack{
        \omega:\text{unitary}, \\
        \text{unram. outside $S$}
      }}
    \mathcal{L}^{(S)}(\omega,0)
    \prod_{\mathfrak{p} \in S}
    \ell_{\omega_\mathfrak{p},0}^*(f_\mathfrak{p})
  \end{equation}
  is equal to the limit
  \begin{equation}\label{eq:difference-degen-terms}
    \lim_{s \rightarrow 0} (\sum_{i=8}^{15} - \sum_{i=1}^7)
    \ell^{\deg}_{i,s}(f[s])
  \end{equation}
  for any holomorphic family
  $f[s] \in \mathcal{I}(s_1) \otimes \mathcal{I}(s_2)$, defined
  for $s \in \mathbb{C}^3$ near the origin, with $f[0] = f$.
\end{corollary}
By combining the remarks of \S\ref{sec:vari-with-resp} with the
unramified calculations of \S\ref{sec-10-1} and
\S\ref{sec-10-2}, we see that the above results extend with
minor modification to the case that $\psi_\mathfrak{p}$ is
ramified for some finite $\mathfrak{p} \notin S$.

\begin{remark}
\label{rmk:}
One could likely evaluate the limit
\eqref{eq:difference-degen-terms} more explicitly as in
Motohashi's work, but it is more convenient in our experience
to work with directly with the individual degenerate functionals
$\ell_{i,s}^{\deg}$, each of which has clear
representation-theoretic significance.
\end{remark}

\subsection{Holomorphy}
\label{sec:holomorphy-global-functionals}
We record, for future reference,
some holomorphy properties implicit in the above arguments.

As in the local setting (\S\ref{sec:families-vectors}), we may
speak of holomorphic families
$f[s] \in \mathcal{I}(s_1) \otimes \mathcal{I}(s_2) \otimes
\mathbb{C}(s_3)$ indexed by $s$ in an open subset $U$ of
$\mathbb{C}^3$; this means that $f[s]$ varies pointwise
holomorphically, the seminorms $\mathcal{S}_d(f[s])$ defined in
\S\ref{sec:repr-adel-groups} are locally bounded in $s$, and,
locally in $s$, $f[s]$ is invariant by a compact open subgroup
of the finite adelic points of $\PGL_2 \times \PGL_2$.  We say
that a family of (continuous) functionals $\rho_s$ indexed by
$s$ in $U$ varies holomorphically if $s \mapsto \rho_s(f[s])$ is
holomorphic for all holomorphic families $f[s]$ defined on an
open subset of $U$, or equivalently, if $\rho_s(f[s])$ varies
holomorphically for each flat family $f[s]$ defined on $U$.
\begin{lemma}
  Each of the following functionals varies holomorphically
  on $\{s \in \mathbb{C}^3 : |\Re(s_j)| < 1/6 \text{ for } j=1,2,3\}$:
  \begin{equation}\label{eq:five-functionals-holomorphic}
    \ell_s, \quad
    \int_{\sigma:\text{generic}}
    \ell_{\sigma,s},
    \quad
    \int_{\omega:\text{unitary}}
    \ell_{\omega,s},
    \quad
    \sum_{i=1}^{7}
    \ell_{i,s}^{\deg},
    \quad
    \sum_{i=8}^{15}
    \ell_{i,s}^{\deg}.
  \end{equation}
\end{lemma}
\begin{proof}
  The holomorphy of $\ell_s$ on the indicated domain follows
  from \eqref{eqn:assumptions-s-2}, lemma
  \ref{lem:reg-int-holom-var}, and the holomorphic variation of
  the asymptotic expansions near $0$ and $\infty$ of
  \eqref{eqn:prod-eis-and-y-s3} as $f_1, f_2$ vary
  holomorphically.  The holomorphy of the next two functionals
  in \eqref{eq:five-functionals-holomorphic} was implicit in the
  proofs of Theorems \ref{thm:GG-G-A} and \ref{thm:GG-AA-A}.
  The holomorphy of the last two then follows from the
  identities proved in those theorems.
\end{proof}

\part{Analysis of local weights and degenerate terms}\label{part:analys-local-weights}

\section{Holomorphic families of
  weights and vectors}

\subsection{Local}\label{sec:holom-famil-local}
Let $F$ be a local field,
with nontrivial unitary character $\psi$.
We fix a small neighborhood $\Omega$ of the origin in
$\mathbb{C}^2$.  We let $s = (s_1,s_2) \in \Omega$ and set
$\pi := \mathcal{I}(s_1) \otimes \mathcal{I}(s_2)$.
For $f \in \mathcal{I}(s_1) \otimes \mathcal{I}(s_2)$
we define $W_f$ and $V_f$
as in \S\ref{sec-6-5}.
We aim to
generalize the results of \S\ref{sec-6} from $s=0$ to all
$s \in \Omega$.
No new ideas are required here,
but the formulas obtained are slightly more complicated.

We first generalize
the results of
\S\ref{sec-6-5}:
\begin{lemma}\label{lem:W_f-vs-V_f-general-s}~
  \begin{enumerate}[(i)]
  \item   The set $\{W_f : f \in \pi \}$
    contains $C_c^\infty(F^\times \times F^\times)$.
  \item   The set $\{V_f : f \in \pi \}$ contains
    $C_c^\infty(N \backslash G, \psi)$.
  \item 
    For $f \in \pi$,
    we have
    \begin{align}\label{eqn:W_f-via-V_f-general-s}
      &W_f(t_1,t_2) \\ \nonumber
      &=
        |t_1|^{1/2-s_1} |t_2|^{1/2-s_2}
        \int_{x \in F}
        |1-x|^{2 s_1} |x|^{2 s_2}
        V^\sharp[s](
        x,
        \frac{(t_1 + t_2) x - t_2 }{x(1-x)}
        )
        \,
        \frac{d x}{|x(1-x)|},
    \end{align}
    where
    \begin{equation}
      V^\sharp[s](x,y)
      :=
      \int_{\xi \in F}
      |\xi|^{-2 s_2}
      V^\wedge[s](\xi,-x/\xi)
      \psi(-\xi y)
      \, d \xi
    \end{equation}
    with
    \begin{equation}
      V^\wedge[s](\xi,z)
      :=
      \int_{y \in F^\times}
      |y|^{s_1-s_2}
      V(a(y) n'(z))
      \psi(\xi y)
      \, \frac{d y}{|y|}.
    \end{equation}
    The integration in \eqref{eqn:W_f-via-V_f-general-s}
    is understood as in lemma \ref{lem:W_f-via-V_f-0}.
  \end{enumerate}
\end{lemma}
\begin{proof}
  The first assertion is again
  a consequence of standard properties
  of the Kirillov model.
  For the remaining assertions,
  we fix a nonzero
  test function $\phi_0 \in C_c^\infty(F^\times)$ supported
  close enough to the identity that
  $\int_{F} \phi_0 |.|^{2 s_1}  \neq 0$ for all
  $s \in \Omega$
  and define $\Phi[s] \in \mathcal{S}(F^3)$ by
  \begin{equation}\label{eq:defn-Phi-of-s-via-V-wedge-of-s}
        \Phi[s](x,y,z)
    :=
    \frac{\phi_0(x-y z)}{ \int_{ F} \phi_0 |.|^{2 s_1} }
    V^\wedge[s] (\frac{y}{x - y z}, z).
  \end{equation}
  and then $f[s] \in \mathcal{I}(s_1) \otimes \mathcal{I}(s_2)$
by
\begin{equation}\label{eq:defn-f-of-s-via-Phi-of-s}
      f[s](g, n'(z)) := |\det g|^{1/2 + s_1} \int_{r \in F}
    \Phi[s]( (0,r) g, z) |r|^{2 s_1} \, d r.
\end{equation}
  The same calculations as in the proof of lemma
  \ref{lem:many-V-f} confirm that $V_{f[s]} = V$, whence the
  second assertion.  The same calculations as in the proof of
  lemma \ref{lem:W_f-via-V_f-0} lead to the required formula
  for $W_f$ in terms of $V_f$.
\end{proof}

We next generalize and slightly refine
the results of
\S\ref{sec-6-6}:
\begin{theorem}\label{thm:refined-constr-adm-weight}
  Let $\phi$, $h$
  be as in Theorem \ref{thm:constr-admiss-weight}.
  Fix a small neighborhood $\Omega$ of the origin in
  $\mathbb{C}^3$.
  Then we may find a holomorphic family
  $f[s] \in \mathcal{I}(s_1) \otimes \mathcal{I}(s_2)$,
  defined for $s = (s_1,s_2,s_3) \in  \Omega$,
  with the following properties:
  \begin{enumerate}[(i)]
  \item \label{item:ell-f-s-recovers-h} $\ell_{\sigma,s}(f[s]) = h(\sigma)$ for all
    $s$.
  \item \label{item:ell-omega-f-o-gives-tilde-h} $\ell_{\omega,0}(f[0]) = \tilde{h}(\omega)$
    is as described in Theorem \ref{thm:constr-admiss-weight}.
  \item \label{item:vanish-on-N-inv-funcs}
    for each $s$
    and every $N$-invariant functional $\ell : \pi \rightarrow
    \mathbb{C}$,
    we have $\ell(f[s]) = 0$.
  \item \label{item:crude-bound-f-of-s-via-tilde-phi}
    Suppose that $\phi$ has the form
    $\phi(n(x) a(y) n'(z)) = \tilde{\phi}(x,y,z)$
    for some $\tilde{\phi} \in \mathcal{S}(F^3)$.
    Then for each fixed $d$ there is a fixed
    $d'$
    so that
    for all $s \in \Omega$,
    \begin{equation}\label{eq:estimate-S-f-s-via-tilde-phi}
      \mathcal{S}_d(f[s]) \ll \mathcal{S}_{d'}(\tilde{\phi})
      q^{\O(1)},
    \end{equation}
    where $q := 1$ if $F$ is archimedean and the Sobolev norms
    $\mathcal{S}_d$ are as defined in
    \S\ref{sec:constr-sobol-norms} and
    \S\ref{sec:norms-schw-spac}.
  \end{enumerate}
\end{theorem}
\begin{proof}
  As in the proof of Theorem
  \ref{thm:constr-admiss-weight},
  we define
  $V \in C_c^\infty(N \backslash G, \psi)$ by
  \eqref{eqn:defn-V-via-phi} and then
  $f_0[s] \in \mathcal{I}(s_1) \otimes \mathcal{I}(s_2)$ using
  lemma \ref{lem:W_f-vs-V_f-general-s}, so that
  $V = V_{f_0[s]}$.
  We choose $\phi_1 \in C_c^\infty(A) \cong C_c^\infty(F^\times)$
  supported close enough to the identity
  that $\int_A \phi_1 |.|^{s_3} \neq 0$ for all $s \in \Omega$,
  and set
  \begin{equation}\label{eqn:f-s-via-phi1-iota-N-psi-f-0}
    f[s] :=
    \frac{1}{\int_A \phi_1 |.|^{s_3}}
    \phi_1^{\iota} \ast N_\psi f_0[s]
    \in \mathcal{I}(s_1) \otimes \mathcal{I}(s_2),
  \end{equation}
  with notation as in \eqref{eqn:defn-f-via-f0-basic-case}.
  Assertion  \eqref{item:ell-f-s-recovers-h}
  holds
  because $V = V_{f_0[s]}$.
  The proof of assertion
  \eqref{item:ell-omega-f-o-gives-tilde-h}
  is the same calculation
  as in the proof of
  Theorem \ref{thm:constr-admiss-weight},
  noting that
  the present definitions specialize to the earlier ones
  upon taking $s = 0$.
  Assertion \eqref{item:vanish-on-N-inv-funcs}
  follows from the definition
  \eqref{eqn:f-s-via-phi1-iota-N-psi-f-0}
  and the fact that $\phi_1$ is supported away from $0$.

  For the proof of assertion
  \eqref{item:crude-bound-f-of-s-via-tilde-phi}, it is
  convenient to use the construction of $f[s]$ given by
  \eqref{eq:defn-Phi-of-s-via-V-wedge-of-s} and
  \eqref{eq:defn-f-of-s-via-Phi-of-s}.  We may assume in the
  non-archimedean case that $\phi_0$ is the normalized
  characteristic function of $\mathfrak{o}^\times$.  Since
  $\Phi[s]$ is essentially a partial Fourier transform of
  $\tilde{\phi}$,
  we see that
  \begin{equation}
    \mathcal{S}_d(\Phi[s]) \ll \mathcal{S}_{d'}(\tilde{\phi}).
  \end{equation}
  with $d,d'$ as above.  Similarly, it follows readily from
  \eqref{eq:defn-f-of-s-via-Phi-of-s} that
  \begin{equation}
    \mathcal{S}_d(f[s]) \ll \mathcal{S}_{d'}(\Phi[s]) q^{\O(1)}.
  \end{equation}
  The required estimate \eqref{eq:estimate-S-f-s-via-tilde-phi}
  follows.
\end{proof}

\subsection{Global}\label{sec:constr-glob-test}
Returning to the global setting, let $F$ be a number field,
let $S$ be a finite set of places of $F$ containing all
archimedean places,
and for each $\mathfrak{p} \in S$, let $h_\mathfrak{p}$
be a pre-Kuznetsov weight defined on the set of generic
irreducible representations $\sigma_\mathfrak{p}$
of $\PGL_2(F_\mathfrak{p})$.

Let $\Omega \subseteq \mathbb{C}^3$
be a small neighborhood of the origin.  For $s \in \Omega$, let
$f[s] = \otimes f[s]_\mathfrak{p} \in \mathcal{I}(s_1) \otimes
\mathcal{I}(s_2) \otimes \mathbb{C}(s_3)$ denote the
factorizable vector such that
\begin{itemize}
\item
  for $\mathfrak{p} \notin S$,
  the local component
  $f[s]_\mathfrak{p}$ is normalized spherical
  (i.e., the unique
  $\PGL_2(\mathfrak{o}_\mathfrak{p})^2$-invariant
  vector with $f[s]_\mathfrak{p}(1) = 1$), while
\item for $\mathfrak{p} \in S$,
  the local component $f[s]_\mathfrak{p}$ is as
  constructed in Theorem \ref{thm:refined-constr-adm-weight}.
\end{itemize}

We refer subsequently to
$(f[s])_{s \in \Omega}$
as \emph{the holomorphic family
attached to the local weights
$(h_\mathfrak{p})_{\mathfrak{p} \in S}$}.

\section{Local estimates for long
  families}\label{sec:crude-local-estim}
The results of this section may be used in place of lemma
\ref{lem:crude-lower-bound-individual} to ensure that our main
estimates depend polynomially upon auxiliary parameters.  We
note that many of the estimates recorded this section may be strengthened
significantly via explicit calculation.

Let $F$ be a local field and $\psi$ a nontrivial unitary
character of $F$.  If $F$ is non-archimedean, then we assume
that $\psi$ is unramified.  Let $Q \geq 1$ be an element of the
value group of $F$.  In the non-archimedean case, we assume that
$Q \geq q$.

We define $\phi \in C_c^\infty(G)$ by applying the recipe of
\S\ref{sec:constr-suit-weight} with $\chi$ the trivial
character.  For example, for $F$ non-archimedean, we let
$J \leq G$ denote the set consisting of $g = n(x) a(y) n'(z)$
with $|x| \leq 1, |y| = 1, |z| \leq 1/Q$ and take for
$\phi \in C_c^\infty(G)$ the characteristic function of $J$.  In
either case, we have $\phi(g) = \tilde{\phi}(x,y,z)$ where
$\tilde{\phi} \in \mathcal{S}(F^3)$ satisfies
\begin{equation}\label{eq:cal-S-tilde-phi-polyn}
  \mathcal{S}_d(\tilde{\phi})
  \ll Q^{\O(1)}
\end{equation}
for fixed $d$.  Indeed, one can check that
$\mathcal{S}_d(\tilde{\phi}) \asymp
q^{\O(1)} Q^{d-1/2}$ for fixed $d \geq 0$.

Let $h$ denote the pre-Kuznetsov weight with kernel $\phi$
(\S\ref{sec:pre-kuzn-weights}).
Let $\tilde{h}$ denote its dual, as given by Theorem
\ref{thm:constr-admiss-weight}.
\begin{lemma}
  Fix $\vartheta \in [0,1/2)$,
  and let $\sigma \in G_{\gen}^\wedge$ be unitary.
  \begin{enumerate}[(i)]
  \item We have $h(\sigma) \geq 0$.
  \item If $\sigma$ is $\vartheta$-tempered
    and $C(\sigma) \leq Q$, then
    $h(\sigma) \gg_{\vartheta} 1/Q$.
  \end{enumerate}
\end{lemma}
\begin{proof}
  The proof is similar to but simpler than that of Theorem
  \ref{thm:lower-bounds-weights}.  We discuss only the
  non-archimedean case in detail, since the modifications
  required for the archimedean case are exactly as in \S\ref{sec:lower-bounds-for-wts}.
  
  We note first that $h(\sigma)$ is the sum of
  $\vol(J) |W(1)|^2$ taken over $W$ in an orthonormal basis for
  the space $\sigma^J$ of $J$-fixed vectors in $\sigma$; in
  particular, $h(\sigma) \geq 0$.  Suppose now that $\sigma$ is
  $\vartheta$-tempered and $C(\sigma) \leq Q$.  By newvector
  theory \cite{MR0337789}, $\sigma$ then contains a $J$-invariant element $W$
  with $W(1) = 1$.  Moreover, as before, we have
  $W(1) \gg_{\vartheta} \|W\|$.  Since $\vol(J) \asymp 1/Q$, the
  required lower
  bound for $h(\sigma)$
  follows.
\end{proof}

\begin{lemma}
Let $f[s] \in \mathcal{I}(s_1) \otimes \mathcal{I}(s_2)$
denote the holomorphic family,
defined for small $s \in \mathbb{C}^3$,
attached to $\phi$ and $h$
by Theorem \ref{thm:refined-constr-adm-weight}.
Then for each fixed $d$,
\begin{equation}
  \mathcal{S}_d(f[s]) \ll Q^{\O(1)}.
\end{equation}
Moreover, for unitary $\omega \in A^\wedge$,
\begin{equation}
  \tilde{h}(\omega) \ll Q^{\O(1)} C(\omega)^{-d}
\end{equation}
\end{lemma}
\begin{proof}
  The first estimate is a consequence of
  \eqref{eq:cal-S-tilde-phi-polyn} and
  \eqref{eq:estimate-S-f-s-via-tilde-phi}.  The second then
  follows from \eqref{eq:crude-ell-omega-s-estimate}.
\end{proof}

\section{The first seven degenerate
  terms}\label{sec:handl-degen-terms-1}
We give conditions on the local
weights $h_\mathfrak{p}$ under which the first seven degenerate
terms may be neglected.  The informal content of these
conditions is that the weights ``vanish adequately near the
trivial representation.''

\begin{theorem}\label{thm:first-seven-degen}
  Let $F$ be a number field.
  Let $S$ be a finite set of places of $F$, containing all
  archimedean places.
  For each $\mathfrak{p} \in S$, let $h_\mathfrak{p}$
  be a pre-Kuznetsov weight for $\PGL_2(F_\mathfrak{p})$.
  Assume that
  there exists either
  \begin{enumerate}[(i)]
  \item \label{item:first-seven-degen-finite-place}
    a
    finite place $\mathfrak{p} \in S$
    such that $h_\mathfrak{p}$ vanishes
    on the subset of unramified representations, or
  \item \label{item:first-seven-degen-arch-place} an infinite place $\mathfrak{p} \in S$
    for which $h_\mathfrak{p}$ is divisible
    by the sixth power of the Casimir operator
    $\mathcal{C}_\mathfrak{p}$
    on $\PGL_2(F_\mathfrak{p})$,
    i.e., there exists a pre-Kuznetsov weight
    $h_\mathfrak{p}^0$ so that
    $h_\mathfrak{p}(\sigma) =    \lambda_{\sigma_\mathfrak{p}}^6
    h_\mathfrak{p}^0(\sigma_\mathfrak{p})$
    for all $\sigma_\mathfrak{p}$,
    where
    $\lambda_{\sigma_\mathfrak{p}}$
    denotes the $\mathcal{C}_\mathfrak{p}$-eigenvalue.
  \end{enumerate}
  Let $\Omega \subseteq \mathbb{C}^3$ be a small neighborhood of
  the origin, and let $(f[s])_{s \in \Omega}$ be the holomorphic
  family attached to $(h_\mathfrak{p})_{\mathfrak{p} \in S}$.
  Then for $i=1..7$, we have
  \begin{equation}\label{eqn:first-seven-degen-tend-to-zero}
    \lim_{s \rightarrow 0} \ell^{\deg}_{i,s}(f[s]) = 0,
  \end{equation}
  with the limit taken along $s$ in a generic complex
  line in $\mathbb{C}^3$
  containing the origin.
\end{theorem}
Note that, by assertion
\eqref{item:assertion-about-unramified-vanishing-of-h} of
Theorem \ref{thm:lower-bounds-weights}, the condition
\eqref{item:first-seven-degen-finite-place} of Theorem
\ref{thm:first-seven-degen} is satisfied (at some finite place)
for the weights relevant to our applications
(see \S\ref{sec-8}).
\begin{proof}
  We observe first, by the stated properties
  of $f[s]$
  and the $N_{\mathbb{A}}$-invariance
  of $\ell^{\deg}_{5,s}$,
  that
  $\ell^{\deg}_{5,s}(f[s]) = 0$.
  It thus suffices to verify
  \eqref{eqn:first-seven-degen-tend-to-zero}
  for $i \in \{1,2,3,4,6,7\}$.
  We will make use
  of the factorization properties
  stated in Theorem \ref{thm:GG-G-A},
  namely,
  that $\ell^{\deg}_{i,s}$
  factors $\Delta G \times A$-equivariantly
  through
  $\mathcal{I}(t) \otimes \mathbb{C}(s_3)$
  for some complex number $t = \pm 1/2 + \O(|s|)$.

  For a complex number $t$
  and a place $\mathfrak{p}$,
  we denote by $\mathcal{I}_\mathfrak{p}(t)$
  the corresponding induced
  representation of $\PGL_2(F_\mathfrak{p})$,
  so that $\mathcal{I}(t) = \otimes_{\mathfrak{p}}
  \mathcal{I}_\mathfrak{p}(t)$.

  Consider first the case that the condition
  \eqref{item:first-seven-degen-finite-place} is satisfied for
  some finite place $\mathfrak{p}$.
  The space of
  $\PGL_2(F_\mathfrak{p})$-invariant functionals
  $\mathcal{I}_{\mathfrak{p}}(s_1) \otimes \mathcal{I}_{\mathfrak{p}}(s_2) \rightarrow
  \mathcal{I}_{\mathfrak{p}}(t)$ is one-dimensional, and described explicitly
  using the local Rankin--Selberg integral
  $\mathcal{I}_{\mathfrak{p}}(s_1) \otimes \mathcal{I}_{\mathfrak{p}}(s_2) \otimes
  \mathcal{I}_{\mathfrak{p}}(-t) \rightarrow \mathbb{C}$ or some
  Laurent coefficient thereof
  (see \S\ref{sec:local-rs-basic}).
  The assumption on $h_\mathfrak{p}$ implies that any such
  Rankin--Selberg integral vanishes at $f[s]_\mathfrak{p}$.  It
  follows that $f[s]_\mathfrak{p}$ lies in the kernel of any
  invariant functional
  $\mathcal{I}_{\mathfrak{p}}(s_1) \otimes \mathcal{I}_{\mathfrak{p}}(s_2) \rightarrow
  \mathcal{I}_{\mathfrak{p}}(t)$, hence that $\ell^{\deg}_{i,s}(f) = 0$.
  
  Consider next the case that the condition
  \eqref{item:first-seven-degen-arch-place} is satisfied for
  some infinite place $\mathfrak{p}$.
  Let $f^0[s]$ be attached to $h^0$
  in the same way that $f[s]$ was to $h$.
  Then for the same reasons as in the previous case,
  we have
  $\ell^{\deg}_{i,s}(f[s])
  = \lambda_t^6 \ell^{\deg}_{i,s}(f^0[s])$,
  where $\lambda_t$ denotes the eigenvalue for
  $\mathcal{C}_\mathfrak{p}$ on $\mathcal{I}(t)$.  If
  $t = \pm 1/2 + \O(|s|)$, then $\lambda_t = \O(|s|)$.
  Since
  $\ell^{\deg}_{i,s}(f^0[s]) = \O(|s|^{-5})$ (see
  \eqref{eqn:polar-bounds-for-degen-func}), it follows that
  $\ell^{\deg}_{i,s}(f[s]) = \O(|s|)$, whence the required
  conclusion.
\end{proof}

\section{The remaining eight degenerate
  terms}\label{sec:handl-degen-terms-2}

\begin{theorem}\label{thm:last-eight-degen-func}
  Let $F$ be a number field, equipped with a nontrivial unitary
  character $\psi$ of $\mathbb{A}/F$.  Let $S = S_0 \cup S_1$,
  $(h_\mathfrak{p})_{\mathfrak{p} \in S}$ and $Q$ be as in
  \S\ref{sec:CI} or \S\ref{sec:PY1}.  Let
  $f[s] \in \mathcal{I}(s_1) \otimes \mathcal{I}(s_2) \otimes
  \mathbb{C}(s_3)$ be the holomorphic family, defined for small
  $s$, attached to $(h_\mathfrak{p})_{\mathfrak{p} \in S}$
  in \S\ref{sec:constr-glob-test}.
  Then
  \begin{equation}\label{eq:last-eight-degen-func}
    \lim_{s \rightarrow 0}
    \sum_{i=8}^{15}
    \ell^{\deg}_{i,s}(f[s])
    \ll
    Q^\eps.
  \end{equation}
  In the setting of
  \S\ref{sec:CI},
  the implied constant
  depends polynomially
  upon $\sigma_0$.
\end{theorem}
The proof is given in \S\ref{sec:proof-theor-degen-2}
after several preliminaries.

\begin{remark}
It should be possible to refine this estimate to an asymptotic
formula for the LHS of \eqref{eq:last-eight-degen-func};
compare with \cite{MR3394377}.
\end{remark}


\subsection{Factorization}\label{sec:main-term-degen-func-factorization}
Take $s \in \mathbb{C}^3$ small
and satisfying \eqref{eqn:assumptions-s-2}.
We introduce the temporary notation
\[
  Z(u_1,u_2,u_3)
  :=
  \frac{\zeta_F^{(S)}( u_1) \zeta_F^{(S)}( u_2)
    \zeta_F^{(S)}( u_3)}{\zeta_F^{(S),*}(1)}.
\]
Take $f = f_1 \otimes f_2$
with $f_i = \otimes f_{i \mathfrak{p}} \in \mathcal{I}(s_i)$
unramified outside $S$.
Then by \S\ref{sec-5-4}, the $\ell^{\deg}_{i,s}(f)$ ($i=8..15$)
factor as
\begin{equation}\label{eq:factor-degeg-8}
  Z(1+ 2 s_1, 1+ s_1 + s_2 + s_3,
  1+ s_1 - s_2 - s_3)
  \prod_{\mathfrak{p} \in S}
  f_{1 \mathfrak{p}}^*(1)
  \int_{A_\mathfrak{p}}
  W_{f_{2 \mathfrak{p}}^*} |.|^{1/2+s_1+s_3}
\end{equation}
\begin{equation}
  Z(1 - 2 s_1,1 - s_1 + s_2 + s_3,1  - s_1 - s_2 + s_3)
  \prod_{\mathfrak{p} \in S}
  M f_{1 \mathfrak{p}}^*(1)
  \int_{A_\mathfrak{p}}
  W_{f_{2 \mathfrak{p}}^*} |.|^{1/2-s_1+s_3}
\end{equation}
\begin{equation}
  Z(1 + 2 s_1, -s_1 + s_2 + s_3,-s_1 - s_2 + s_3)
  \prod_{\mathfrak{p} \in S}
  f_{1 \mathfrak{p}}^*(w)
  \int_{A_\mathfrak{p}}^{\reg}
  W_{f_{2 \mathfrak{p}}^*} |.|^{-1/2-s_1+s_3}
\end{equation}
\begin{equation}\label{eq:factor-degeg-11}
  Z(1 - 2 s_1, s_1 + s_2 + s_3, s_1 - s_2 + s_3)
  \prod_{\mathfrak{p} \in S}
  M f_{1 \mathfrak{p}}^*(w)
  \int_{A_\mathfrak{p}}^{\reg}
  W_{f_{2 \mathfrak{p}}^*} |.|^{-1/2-s_1+s_3},
\end{equation}
together with four similar terms
obtained by swapping $(f_1,s_1)$ with $(f_2,s_2)$.


\subsection{Evaluation of local functionals}
Let $(F,\psi)$ be a local field.
Take $s \in \mathbb{C}^3$ small
and satisfying   \eqref{eqn:assumptions-s-2}.
We consider the eight $A$-invariant functionals on
$\mathcal{I}(s_1) \otimes \mathcal{I}(s_2) \otimes
\mathbb{C}(s_3)$ defined
by sending $f = f_1 \otimes f_2$ to the local
integrals implicit above, namely,
\begin{equation}\label{eqn:four-functionals-degen-local}
  f_1(1) \int_A W_{f_2} |.|^{1/2 + s_1 + s_3},
  \quad  M f_1(1) \int_A W_{f_2} |.|^{1/2 - s_1 + s_3},
\end{equation}
\begin{equation}\label{eqn:four-functionals-degen-local-2}
  f_1(w) \int_A^{\reg} W_{f_2} |.|^{-1/2 - s_1 + s_3},
  \quad 
  M f_1(w) \int_A^{\reg} W_{f_2} |.|^{-1/2 + s_1 + s_3},
\end{equation}
together with the analogous quantities obtained by swapping the
indices $1$ and $2$.  For notational simplicity, we have
replaced $f_i^*$ with $f_i$ (see \eqref{eq:defn-f-asterisk});
this has no effect on the estimation to be carried out because
$\zeta_F(1+2 s_i) \asymp 1$ for small $s_i$.  We aim to evaluate
the quantities \eqref{eqn:four-functionals-degen-local},
\eqref{eqn:four-functionals-degen-local-2} in a manner more
convenient for estimation.

For $s \in \mathbb{C}^3$ near the origin,
$f = \mathcal{I}(s_1) \otimes \mathcal{I}(s_2)$
and $\nu \in \mathbb{C}^2$,
we define the double Hecke integral
\begin{equation}
  D_f(\nu) := \int_{A \times A}^{\reg} W_f |.|^{\nu_1} \otimes |.|^{\nu_2}
\end{equation}
and its normalized variant (cf. \S\ref{sec-4-3})
\begin{equation}\label{eqn:normalized-double-hecke}
  D_f^*(\nu) := 
  \frac{D_f(\nu)}{L(\mathcal{I}(s_1), 1/2 + \nu_1)
    L(\mathcal{I}(s_2), 1/2 + \nu_2)},
\end{equation}


\begin{lemma}\label{lem:eval-local-funcs-via-D-f}
  For $f = f_1 \otimes f_2$,
  the quantities listed in \eqref{eqn:four-functionals-degen-local}
  and \eqref{eqn:four-functionals-degen-local-2}
  are respectively equal to
  \begin{equation}
    \beta_1(s)
    D_f^*(-1/2  - s_1, 1/2 + s_1 + s_3),
  \end{equation}
  \begin{equation}
    \beta_2(s)
    D_f^*(-1/2  + s_1, 1/2 - s_1 + s_3),
  \end{equation}
  \begin{equation}
    \beta_3(s)
    D_f^*(1/2+s_1, -1/2 - s_1 + s_3),
  \end{equation}
  \begin{equation}
    \beta_4(s)     D_f^*(1/2-s_1, -1/2 + s_1 + s_3),
  \end{equation}
  with
  \begin{align*}
    \beta_1(s)
    &:=
      \frac{
      L(\mathcal{I}(s_1), 1 + s_1)
      L(\mathcal{I}(s_2), 1 + s_1 + s_3)
      }{
      \eps(\mathcal{I}(s_1),\psi,1+s_1),
      },
    \\
    \beta_2(s) &:=
                 \gamma(\psi,1 - 2 s_1)
                 \frac{
                 L(\mathcal{I}(s_1), 1 - s_1)
                 L(\mathcal{I}(s_2), 1 - s_1 + s_3)
                 }{
                 \eps(\mathcal{I}(s_1),\psi,1-s_1),
                 }
    \\
    \beta_3(s) &:=    L(\mathcal{I}(s_1), 1 + s_1)
                 L(\mathcal{I}(s_2), - s_1 + s_3),
    \\
    \beta_4(s) &:=
                 \gamma(\psi,1 - 2 s_1)
                 L(\mathcal{I}(s_1), 1 - s_1)
                 L(\mathcal{I}(s_2), s_1 + s_3),
  \end{align*}
\end{lemma}
\begin{proof}
By the formulas
\eqref{eq:W-f-viz-f} and \eqref{eqn:W-Mf-vs-W-f}
and Fourier inversion
(see \eqref{eq:fourier-inversion-for-whittaker-intertwiner}),
we have
\begin{equation}
  f_1(w) = \int_{A}
  W_{f_1} |.|^{1/2 + s_1},
\end{equation}
\begin{equation}
  M f_1(w) =
  \gamma(\psi,1 - 2 s_1)
  \int_{A}
  W_{f_1} |.|^{1/2 - s_1}.
\end{equation}
Thus
\begin{equation}
  f_1(1) = \int_{A}
  W_{w f_1} |.|^{1/2 + s_1},
\end{equation}
\begin{equation}
  M f_1(1) =
  \gamma(\psi,1 - 2 s_1)
  \int_{A}
  W_{w f_1} |.|^{1/2 - s_1}.
\end{equation}
To express these last two Hecke
integrals in terms of $W_{f_1}$, we
invoke the local functional equation
(see \eqref{eq:31})
\begin{equation}
  \int_A W_{w f_1} |.|^{1/2 \pm s_1}
  =
  \lim_{u \rightarrow 0}
  \frac{1}{\gamma(\psi,\mathcal{I}(s_1), 1 + u  \pm s_1)}
  \int_A^{\reg} W_{f_1} |.|^{- 1/2 - u \mp s_1}.
\end{equation}
Here we need to take a limit because the
latter Hecke integral may have a
pole at $u = 0$, compensated for by the pole of the
$\gamma$-factor
\[
  \gamma(\psi,\mathcal{I}(s_1), 1 + u \pm s_1)
  =
  \gamma (\psi,1 + u)
  \gamma (\psi,1 + u \pm 2 s_1).
\]
The required identities follow readily.
\end{proof}


\subsection{Reduction to local estimates}

\begin{definition}\label{defn:D_f-of-nu}
  Let $F$ be a local field.
  For a holomorphic family
  $f[s] \in \mathcal{I}(s_1) \otimes \mathcal{I}(s_2) \otimes
  \mathbb{C}(s_3)$ defined for $s \in \mathbb{C}^3$ near the
  origin and $\alpha > 0$ sufficiently small, define
  \[
    \mathcal{N}_\alpha(f) :=
    \sup |D_{f[s]}^*(\nu)|,
  \]
  with the supremum 
  taken over all $s$ and $\nu$
  such that for some choice of sign $\pm$, each of the quantities
  \[
    s_1, s_2, s_3, \nu_1 \pm 1/2, \nu_2 \mp 1/2
  \]
  is bounded in magnitude by $\alpha$.
\end{definition}

We note, by \S\ref{sec-4-3},
that $\mathcal{N}_\alpha(f)$ is finite.

\begin{proposition}\label{lem:degen-reduce-to-local}
  Let $F$ be a number field.  Fix a nontrivial unitary character
  $\psi$ of $\mathbb{A}/F$.
  Let $S$, $(h_\mathfrak{p})_{\mathfrak{p} \in S}$
  and $f[s]$
  be as in \S\ref{sec:constr-glob-test}.
  Then for any small $\alpha > 0$,
  \begin{equation}\label{eqn:reduction-bounbd-last-eight-degen-funcs}
    \lim_{s \rightarrow 0}
    \sum_{i=8}^{15}
    \ell^{\deg}_{i,s}(f[s])
    \ll_{\alpha}
    \exp(\O_{\alpha}(\# S))
    \prod_{\mathfrak{p} \in S}
    \mathcal{N}_\alpha(f_\mathfrak{p}).
  \end{equation}
\end{proposition}
\begin{proof}
  Write $\Phi(s)$
  for the LHS of
  \eqref{eqn:reduction-bounbd-last-eight-degen-funcs}.
  Recall from \S\ref{sec:holomorphy-global-functionals}
  that $\Phi$ is holomorphic near
  $s = 0$.
  By Cauchy's theorem, it will suffice to estimate
  $\Phi(s)$ for $s$ belonging to a fixed small circle
  $\mathcal{C}$ in a generic complex line containing the origin.
  For concreteness, take
  \begin{equation}
    \mathcal{C} = \{
    (10^{-6} s_0, 10^{-12} s_0, 10^{- 18} s_0) \in \mathbb{C}^3 :
    s_0 \in \mathbb{C}, |s_0| = \alpha 
    \}.
  \end{equation}
  The products of zeta functions $Z(\dotsb)$ appearing in
  \eqref{eq:factor-degeg-8}--\eqref{eq:factor-degeg-11}, as well
  as the quantities $\beta_i(s)$ of lemma
  \ref{lem:eval-local-funcs-via-D-f}, are $\ll_\alpha 1$ for
  $s \in \mathcal{C}$.  By another application of Cauchy's
  theorem, our task reduces to bounding the normalized Hecke
  integrals $D_{f[s]}^*(\nu)$ in a small neighborhood of the
  relevant points.
  We conclude by the definition of $\mathcal{N}_\alpha(f_\mathfrak{p})$.
\end{proof}

\subsection{Local estimates}
We next obtain local estimates
for the weights relevant
to our applications.
Let $\alpha > 0$ be small.

\begin{lemma}\label{lem:estimate-N-alpha-case-of-interest}
  Let $F$ be a non-archimedean local field equipped with an
  unramified additive character $\psi$.  Let $\chi$ be a
  ramified character of $F^\times$, with conductor
  $Q := C(\chi)$.  Let $h$ be the pre-Kuznetsov weight attached
  to $\Sigma_F(\chi)$ in \S\ref{sec:constr-suit-weight}.  Let
  $f[s] \in \mathcal{I}(s_1) \otimes \mathcal{I}(s_2) \otimes
  \mathbb{C}(s_3)$ be the holomorphic family, defined for $s$
  near zero, that is attached to $h$ by Theorem
  \ref{thm:refined-constr-adm-weight}.  Then
  \begin{equation}
    \mathcal{N}_\alpha(f) \ll Q^{\O(\alpha)}.
  \end{equation}
\end{lemma}
\begin{proof}
  Expanding the construction \eqref{eqn:f-s-via-phi1-iota-N-psi-f-0}
  of $f[s]$
  and
  calculating as in the proof of Theorem
  \ref{thm:constr-admiss-weight}
  (specifically, 
  \eqref{eqn:h-tilde-of-omega-via-h-sharp}),
  we see that
  \begin{equation}
    D_{f[s]}(\nu)
    =
    \frac{\int_A \phi_1 |.|^{\nu_1 + \nu_2}}{
      \int_A \phi_1 |.|^{s_3}
    }
    D_0,
  \end{equation}
  where
  \begin{equation}
    D_0 :=
    \int_{t \in F^\times}
    W_{f_0[s]}(1-t, t) |1-t|^{\nu_1} |t|^{\nu_2}
    \, \frac{d t}{|t(1-t)|}.
  \end{equation}
  By \eqref{eqn:W_f-via-V_f-general-s},
  we have
  \begin{align}
    &W_{f_0[s]}(1-t, t)
    \\ \nonumber
    &=
      |1-t|^{1/2-s_1} |t|^{1/2-s_2}
      \int_{x \in F}
      |1-x|^{2 s_1} |x|^{2 s_2}
      V^\sharp[s](
      x,
      \frac{ x - t }{x(1-x)}
      )
      \,
      \frac{d x}{|x(1-x)|}.
  \end{align}
  
  The calculation thus far has been general.
  Recall now that $F$ is non-archimedean,
  $\chi$ is ramified and $\psi$ is unramified.
  Arguing
  as in the calculation of
  \eqref{eqn:V-sharp-non-arch-evald-0-0},
  we have $V^\wedge[s](\xi,z) = V^\wedge[0](\xi,z)$
  (because $V(a(y) n'(z))$ is supported on $|y| = 1$)
  and $V^\sharp[x](x,y)
  = Q^{-2 s_2}
  V^\sharp[0](x,y)$ (because
  $V^\wedge[s](\xi,-x/\xi)$
  is supported on $|\xi| = Q$),
  hence by Fourier inversion,
  \begin{equation}\label{eqn:V-sharp-of-s-explicit-non-arch}
    V^\sharp[s](x,y)
    = Q^{- 2 s_2}
    1_{|x| \leq 1}
    1_{|y| = 1}
    \chi(y).
  \end{equation}
  Substituting \eqref{eqn:V-sharp-of-s-explicit-non-arch}
  above yields an explicit
  formula for $D_0$
  as a double integral over $x$ and $t$.
  The substitution $(x,t) \mapsto (1-x,1-t)$
  swaps the roles of $(\nu_1,s_1)$
  and $(\nu_2,s_2)$,
  so as in the proof of lemma \ref{prop:non-arch-estimates-key}, we may decompose
  \[
    D_0 = D_1 + D_2 -
    D_3,
  \]
  where
  \begin{itemize}
  \item $D_2$ denotes the contribution from when
    $|1-x| = |1-t| = 1$,
  \item $D_3$ that from when $|1-x| = |1-t| = |x| = |t| = 1$, and
  \item $D_1$ denotes the contribution from $|x| = |t| = 1$, or
    equivalently, that from $|1 - x| = |1-t| = 1$ but with
    $(s_1,\nu_1)$ and $(s_2,\nu_2)$ swapped,
    and with everything multiplied by $\chi(-1)$.
  \end{itemize}
  We change coordinates to
  $(u,v)$ as before,
  with $x = 1/v, t = 1/u v$,
  and dyadically
  decompose
  according to the values $U,V \in \{1,q, q^2, \dotsc \}$
  for $|u|,|v|$.
  We obtain in this way
  for $i=1,2$ that
  \begin{equation}\label{eqn:D-i-formula-yay}
    D_i
    =
    \chi(\pm 1)
    Q^{-2  s_2}
    \sum_{U,V \in \{1, q, q^2, \dotsc \}}
    U^{- 1/2 + s_i - \nu_i}
    V^{-1/2 -   s_i - \nu_i}
    D_{U,V}
  \end{equation}
  and that
  \begin{equation}
    D_3
    = Q^{-2 s_2} D_{1,1},
  \end{equation}
  where
  \begin{equation}\label{eqn:D-1-defn}
    D_{U,V}
    =
    \int_{\substack{
        u,v \in F: \\
        |u  - 1| = |u| = U, \\
        |v - 1| = |v| = V, \\
        |u v - 1| = U V
      }
    }
    \chi \left( \frac{1 - 1/u}{1 - 1/v} \right)
    \, \frac{d u \, d v}{  |u v| }.
  \end{equation}
  The sum \eqref{eqn:D-i-formula-yay} converges absolutely
  for small $\nu_1, \nu_2$,
  and is understood in general by meromorphic continuation.

  The integrals $D_{U,V}$ are evaluated below in lemma
  \ref{lem:exhaustive-char-computation}.
  Substituting that evaluation
  and summing some geometric series
  gives an evaluation of the $D_i$.
  The essential feature of this evaluation
  is that
  $D_{U,V}$ vanishes
  unless
  \begin{itemize}
  \item $(U,V) = (Q/q,Q/q)$,
  \item $U = Q/q$ and $V \geq Q$
    or $V = Q/q$ and $U \geq Q$, or
  \item $U, V \geq Q$.
  \end{itemize}
  Moreover, the 
  value of $D_{U,V}$
  does not vary
  within each of these three cases,
  and has size $\ll U V/Q^2$.
  If we write
  $D_{Q/q,Q/q} = c_0/q^2$
  and, for $U, V \geq q$,
  $D_{Q/q,V} = D_{U, Q/q}
  = c_1/q$
  and $D_{U,V} = c_2$,
  then
  $c_0,c_1,c_2 \ll 1$
  and
  we have
  for $i=1,2$ that
  \begin{equation}
    D_i
    =
    \chi(\pm 1)
  Q^{-1 - 2 \nu_i - 2 s_2}
  \left(
    \begin{gathered}
      \frac{c_0}{q^2}
      q^{1 + 2 \nu_i}
      + \frac{c_1}{q}
      \left(
        \frac{q^{1/2 + s_i + \nu_i}}{
          1 - q^{-1/2 + s_i - \nu_i}
        }
        +     \frac{q^{1/2 - s_i + \nu_i}}{
          1 - q^{-1/2 - s_i - \nu_i}
        }
      \right)
      \\
      +
      c_2
      \frac{    1 }{
        (1 - q^{-1/2+s_i-\nu_i})
        (1 - q^{-1/2-s_i-\nu_i})
      }
    \end{gathered}
  \right)
\end{equation}
and that
\begin{equation}
  D_3
  =
  Q^{-2 s_2}
  \begin{cases}
    c_2 & \text{ if } Q = q, \\
    0 & \text{ if } Q> q.
  \end{cases}
\end{equation}
Dividing by
$L(\mathcal{I}(s_1), 1/2 + \nu_1) L(\mathcal{I}(s_2), 1/2 +
\nu_2)$ has the effect of clearing all denominators.
The required estimate follows readily;
note that $Q^{-1 - 2 \nu_i}, Q^{- 2 s_2} \ll Q^{\O(\alpha)}$ for the
indicated ranges.
\end{proof}

\begin{lemma}\label{lem:local-estimates-final-degen-terms-long-families-crude}
  Let $F, \psi, Q, \phi$ and $h$ be as in
  \S\ref{sec:crude-local-estim}.
  Let
  $f[s] \in \mathcal{I}(s_1) \otimes \mathcal{I}(s_2) \otimes
  \mathbb{C}(s_3)$ be the holomorphic family, defined for $s$
  near zero, that is attached to $h$ by Theorem
  \ref{thm:refined-constr-adm-weight}.  Then
  \begin{equation}
    \mathcal{N}_\alpha(f) \ll Q^{\O(1)}.
  \end{equation}
\end{lemma}
\begin{proof}
  The required estimate follows
  from \eqref{eqn:upper-bound-Mellin-transform-W},
\S\ref{sec:reduct-pure-tens} and Cauchy's theorem.
\end{proof}

\subsection{Proof of Theorem \ref{thm:last-eight-degen-func}}\label{sec:proof-theor-degen-2}
By Proposition \ref{lem:degen-reduce-to-local} and the
consequence $\exp(\# S) \ll Q^{\eps}$ of the divisor bound, we
reduce to estimating $\mathcal{N}_\alpha(f_\mathfrak{p})$ for
$\mathfrak{p} \in S$.
For
$\mathfrak{p} \in S_1$, we apply lemma
\ref{lem:estimate-N-alpha-case-of-interest} with $\alpha$
sufficiently small.
For $\mathfrak{p} \in S_0$:
\begin{itemize}
\item in the setting of \S\ref{sec:CI},
  we apply
  lemma
  \ref{lem:local-estimates-final-degen-terms-long-families-crude},
  taking there for ``$Q$''
  the analytic conductor of the local component
  of $\sigma_0$ at $\mathfrak{p}$;
\item
  in the setting of \S\ref{sec:PY1},
  it suffices to note
  that $\mathcal{N}_\alpha(f_\mathfrak{p})$ is finite.  
\end{itemize}


  


\appendix

\section{Oscillatory integral estimates\label{lem:char-sum-CI}}
\label{sec-7-1}
Let $F$ be a non-archimedean local field,
with ring of integers $\mathfrak{o}$, maximal ideal
$\mathfrak{p}$,
and $q := \# \mathfrak{o}/\mathfrak{p}$.
Let $\chi,\omega$ be unitary characters of $F^\times$ with
$\chi$ ramified.
Set $Q := C(\chi) \in \{q, q^2, q^3, \dotsc \}$.
The following evaluations and estimates were
postponed from \S\ref{sec:upper-bounds-dual-wts}.  The estimates
may be understood as generalizations of
the character sum bounds
established in \cite[\S3]{2019arXiv190810346P}.

\begin{lemma}\label{lem:single-variable-chi-integral-boring}
  For $U \in \{1,q, q^2, \dotsc \}$,
  we have
  \begin{equation}
    \int_{
      \substack{
        u \in F^\times : \\
        |u| = |1-u| = U
      }
    }
    \chi(1 - 1/u) \, \frac{d u}{ |u|}
    =
    \begin{cases}
      0 & \text{ if } U \leq Q/q^2, \\
      -1/q & \text{ if }  U = Q/q, \\
      1 - 1/q & \text{ if }  U \geq Q.
    \end{cases}
  \end{equation}
\end{lemma}
\begin{proof}
  For $A \in \{1, q, q^2, \dotsc\}$,
  let $\mathcal{U}(1/A)$
  denote the subgroup
  of
  $\mathfrak{o}^\times$
  consisting of all
  elements
  $x$
  satisfying $|x - 1| \leq 1/A$.
  By the change of variables
  $x = 1-1/u$,
  \begin{align*}
    \frac{1}{U}
    \int_{
    \substack{
    u \in F^\times : \\
    |u| = |1-u| = U
    }
    }
    \chi(1 - 1/u) \, \frac{d u}{ |u|}
    &=
      \int_{x \in \mathcal{U}(1/U)}
      \chi(x) \, d x
      -
      \int_{x \in \mathcal{U}(1/q U)}
      \chi(x) \, d x
    \\
    &=
      1_{U \geq Q}
      \frac{1}{U}
      - 
      1_{q U \geq Q}
      \frac{1}{q U}.
  \end{align*}
  The required formula follows case-by-case.
\end{proof}

\begin{lemma}\label{lem:exhaustive-char-computation}
  Assume that $F$ is of characteristic zero.
  For $U, V \in \{1,q,q^2, \dotsc \}$,
  the integral
  \[
    \rho :=
    \int_{\substack{
        u,v \in F^\times :
        \\
        |u|  = |1 - u | = U, \\
        |v| = |1-v| = V, \\
        |u v - 1| = U V
      }
    }
    \omega (u v - 1)
    \chi ( \frac{1 - 1/u}{1 - 1/v})
    \, \frac{d u \, d v}{|u v|}
  \]
  satisfies the following estimates:
  \begin{itemize}
  \item If $C(\omega) = 1$,
    then
    \begin{equation}
      \rho =
      \begin{cases}
        0 & \text{ if } \min(U,V) \leq Q/q^2, \\
        2/q^2 & \text{ if } U = V = 1, Q = q, \\
        1/q^2 & \text{ if } 
        U = V = Q / q > 1, \\
        (-1/q)(1-1/q) & \text{ if }
        U = Q/q, V \geq Q, \\
        (-1/q)(1-1/q) & \text{ if }
        V = Q/q, U \geq Q, \\
        (1-1/q)^2 & \text{ if } U, V \geq Q.
      \end{cases}
    \end{equation}
    In particular,
    $\rho = 0$
    unless
    $U, V \geq Q/q$,
    in which case
    \begin{equation}
      \rho \ll \min(1,U/Q) \min(1,V/Q).
    \end{equation}
  \item
    Suppose that $C(\omega) > 1$.
    Then
    $\rho = 0$ unless
    $U = V = Q/C(\omega)$.
  \item Suppose that $C(\omega) > 1$ and $U = V = Q/C(\omega)$.
    Then
    \[
      \rho \ll \frac{\sqrt{U V}}{Q} = \frac{1}{C(\omega)}
    \]
    unless
    $U = V = 1$
    and
    $(\chi,\omega)$ is atypical (see \S\ref{sec:upper-bounds-dual-wts}).
  \item
    Suppose that $U = V = 1$ and $(\chi,\omega)$
    is atypical.
    Let $\alpha$ and $\xi$ be as above.
    Then
    \begin{equation}\label{eqn:main-osc-int-est-p-adic}
      \rho \ll
      \frac{N_\alpha(\xi)}{Q}
      \cdot
      \begin{cases}
        q^{1/2} & \text{ if } Q = q^{2 \alpha + 1}, \\
        1& \text{ if } Q = q^{2 \alpha},
      \end{cases}
    \end{equation}
    with
    $N_\alpha(\xi)$
    as in Proposition \ref{prop:non-arch-estimates-key}.
  \end{itemize}
\end{lemma}

\begin{proof}
  Consider first the case that $C(\omega) = 1$.
  Then $\omega$ is
  constant on the domain of integration.
  If $(U,V) \neq (1,1)$,
  then the integration constraint
  $|u v - 1| = U V$ may be omitted,
  so $\rho$ factors as a
  product of integrals
  to which lemma \ref{lem:single-variable-chi-integral-boring}
  applies,
  giving the required assertions.
  The case $U = V = 1$
  may be treated similarly
  using the formula
  \begin{equation}
    \int_{\substack{
        u,v \in F^\times :
        \\
        |u|  = |1 - u | = 1, \\
        |v| = |1-v| = 1, \\
        |u v - 1| < 1
      }
    }
    \chi ( \frac{1 - 1/u}{1 - 1/v})
    \, \frac{d u \, d v}{|u v|}
    =
    \begin{cases}
      0 & \text{ if } Q \geq q^2 \\
      -1/q^2 & \text{ if } Q = q
    \end{cases},
  \end{equation}
  whose proof is similar to that of
  lemma \ref{lem:single-variable-chi-integral-boring}.

  We suppose henceforth that $C(\omega) > 1$.

  If $U \geq Q$, then $\chi(1/u - 1)$ is constant on the domain
  of integration.  Since $U \geq q$, the conditions
  $|1 - u| = U$ and $|u v - 1| = U V$ are redundant, so we may
  factor off the integral
  $\int_{|u| = U} \omega(u v - 1) \, \frac{d u}{|u|} = \int_{|u|
    = U V} \omega(u) \, \frac{d u }{|u|} = 0$.  Thus $\rho = 0$
  when $U \geq Q$, and likewise when $V \geq Q$.  We thus
  suppose henceforth that $U,V < Q$.  By replacing $\chi$ with
  its inverse if necessary, we may reduce further to the case
  that $U \leq V$.

  Suppose now that $C(\omega) > Q/V$, hence in particular that
  $C(\omega) \geq q^2$.  We then apply the change of variables
  $v \mapsto v' := (1 + b) v$, where
  $|b| = C(\omega)/q \leq V/Q$.  Then
  $\chi(1-1/v') = \chi(1-1/v)$, while a simple calculation gives
  $\omega(u v' - 1) = \omega(u v - 1) \omega (1 + t b)$ with
  $t := u v / (u v - 1)$.  On the domain of integration, we have
  $|t| = 1$, hence
  $\int_{b : |b| = C(\omega)/q} \omega(1 + t b) \, d b = 0$.
  It follows that $\rho = 0$.

  We have reduced to the case that $q \leq C(\omega) \leq Q/V$
  and $U \leq V < Q$.  We suppose next that $Q/U = q$, hence
  also that $Q/V = q = C(\omega)$.  If $Q = q$, then
  $U = V = 1$, and the estimate $\rho \ll 1/q$
  follows from
  \cite[\S9]{2018arXiv181102452P}.
  (Their estimates are stated over the finite fields
  of prime cardinality, but the same arguments apply
  over any finite field.)
  If $Q > q$, then $U V \geq q^2 > C(\omega)$, so
  $\omega(u v - 1) = \omega(u v)$ while
  $\chi(\frac{1- 1/u }{1 - 1/v}) = \chi(1 - 1/u + 1/v)$
  factors as a product of additive characters of frequency $Q$
  evaluated at $1/u$ and $1/v$.  Thus $\rho$ factors as a
  product of Gauss sums of size $\rho \ll 1/q$.

  It remains to treat the case that
  \begin{equation}\label{eqn:assumptions-before-non-dyadic-assumption}
    \text{$q \leq C(\omega) \leq Q/V$
      and $U \leq V < Q$ and $Q/U \geq q^2$.}
  \end{equation}
  For this we assume first that $F$ is non-dyadic;
  the dyadic case will be treated at the end.
  We may uniquely decompose
  $Q/U = A B$, where $A, B \in \{1, q, q^2, \dotsc \}$
  with $A \leq B \leq q B$.
  Then $q \leq A$.
  For $X \in \{1, q^{-1}, q^{-2}, \dotsc \}$,
  we introduce the notation $\mathcal{O}(X) := \{x \in
  \mathfrak{o} :
  |x| \leq X\}$
  and $\mathcal{U}(X) := \{x \in \mathfrak{o}^\times : |x - 1|
  \leq X\}$; these define subgroups of $F$ and of $F^\times$, respectively.
  We may assume that
  $\exp : \mathcal{O}(1/A) \rightarrow \mathcal{U}(1/A)$
  is an isomorphism with inverse
  $\log : \mathcal{U}(1/A) \rightarrow \mathcal{O}(1/A)$;
  otherwise, $A \ll 1$, and so the required estimate
  follows from the trivial bound.
  We may assume similarly that $A$
  is large enough that for
  each
  $a \in \mathcal{O}(1/A)$, we have
  the following congruences:
  \begin{equation}
    \exp(a) \equiv 1 + a,
    \quad
    -\log(1 - a)
    \equiv
    a
    \pmod{\mathcal{O}(1/A)},
  \end{equation}
  \begin{equation}
    \exp(a) \equiv 1 + a + a^2/2,
    \quad 
    -\log(1-a) \equiv a + a^2/2
    \pmod{\mathcal{O}(1/B)}.
  \end{equation}

  Let
  $\psi_\chi : \mathcal{O}(1/A) \rightarrow \U(1)$
  denote the character defined by
  $\psi_\chi(a) := \chi(\exp(a))$; it is trivial on
  $\mathcal{O}(1/Q)$ but not on $\mathcal{O}(q/Q)$.  Since
  $C(\omega) \leq Q/V \leq Q$, we then have
  $\omega(\exp(a)) = \chi_\psi(\xi a)$ for some
  $\xi \in \mathcal{O}(1/V)$.

  For $u,v$ in the domain of integration
  and $a,b \in \mathcal{O}(1/A)$,
  set
  $u' := u \exp(a)$,
  $v' := v \exp(b)$.
  The change of variables $(u,v) \mapsto (u', v')$
  preserves the domain of integration and the measure.
  We compute that
  \[
    \log \frac{1 - 1/u'}{1 - 1/u}
    \equiv 
    \frac{ a - a^2/2}{u-1} 
    -
    \frac{a^2/2}{(u-1)^2} 
    \pmod{\mathcal{O}(1/Q)},
  \]
  and also that the terms involving $a^2$
  may be omitted when $a \in \mathcal{O}(1/B)$.
  By this and the analogous identities
  obtained by replacing $u$ with $v$ or with $1/uv$,
  and recalling that $U \leq V$ and $c(\omega) \leq \chi(\chi) = n$,
  we obtain
  \begin{equation}
    \omega (u' v' - 1)
    \chi ( \frac{1 - 1/u'}{1 - 1/v'})
    = 
    \omega (u v - 1)
    \chi ( \frac{1 - 1/u}{1 - 1/v})
    \psi_\chi
    (\phi_1(a,b) - \phi_2(a,b)/2),
  \end{equation}
  where
  \begin{equation}\label{eqn:defn-phi-1}
    \phi_1(a,b)
    :=
    \frac{1}{u-1} a
    - \frac{1}{v-1} b
    + \xi \frac{u v }{1 - u v} ( a + b) 
  \end{equation}
  and
  \begin{equation}\label{eqn:defn-phi-2}
    \phi_2(a,b) := \frac{u}{(1-u)^2} a^2 - \frac{v}{(1-v)^2} b^2 +
    \xi \frac{uv}{(uv-1)^2} (a + b)^2.
  \end{equation}
  We now consider the average
  \begin{equation}\label{eqn:rho-u-v-defn}
    \rho(u,v)
    :=
    \mathbb{E}_{a,b \in \mathcal{O}(1/A)}
    \psi_\chi
    (\phi_1(a,b) - \phi_2(a,b)/2).
  \end{equation}
  where $\mathbb{E}$ denotes the integral with respect
  to a probability Haar measure.
  Then
  \begin{equation}\label{eqn:rho-via-rho-u-v}
    \rho = \int_{\substack{
        u,v \in F^\times :
        \\
        |u|  = |1 - u | = U, \\
        |v| = |1-v| = V, \\
        |u v - 1| = U V
      }
    }
    \omega (u v - 1)
    \chi ( \frac{1 - 1/u}{1 - 1/v})
    \rho(u,v)
    \, \frac{d u \, d v}{|u v|},
  \end{equation}
  where the integrand
  now factors through
  $F^\times / \mathcal{U}(1/A)$.

  We compute first the
  analogous averages
  over cosets of $\mathcal{O}(1/B)$,
  i.e., we write
  \[
    \rho(u,v)
    =
    \mathbb{E}_{a,b \in \mathcal{O}(1/A) / \mathcal{O}(1/B)}
    \psi_\chi(- \phi_2(a,b)/2)
    \mathbb{E}_{r,s \in \mathcal{O}(1/B)}
    \psi_\chi(\phi_1(a+r,b+s)).
  \]
  Set $\tau := \frac{u v}{1 - u v}$.
  The inner integral vanishes identically unless the condition
  \[
    -\frac{1}{u-1}
    \equiv
    \xi \tau 
    \equiv \frac{1}{v-1}
    \pmod{\mathcal{O}(B/Q)}
  \]
  is satisfied, in which case it evaluates to $\phi_1(a,b)$.
  Since $1/U > B/Q$ and $|\tau| = 1$,
  the above condition implies that $U = V$ and
  $|\xi| = 1/V$, hence that $C(\omega) = Q/V$.
  In view of the support conditions on $u,v$,
  we may rewrite the above as
  \begin{equation}\label{eqn:congruences-for-u-v-via-xi-tau}
    u \equiv 1 - 1/\xi \tau,
    \quad 
    v \equiv 1 + 1/\xi \tau
    \pmod{\mathcal{U}(1/A)}.
  \end{equation}
  From this we deduce that
  \begin{equation}\label{eqn:tau-quadraitc-congruence}
    \tau \equiv \xi^2 \tau^2 - 1 \pmod{\mathcal{O}(1/A)}.
  \end{equation}
  Note that the number of
  solutions modulo $\mathcal{O}(1/A)$ to this congruence
  is $\O(1)$ unless $1 + 4 \xi^2 \in \mathfrak{p}$,
  in which case $|\xi| = 1$.

  If $B = A$, then we may conclude already that
  \eqref{eqn:main-osc-int-est-p-adic} holds
  if $U = V = 1$ and $(\chi,\omega)$
  is atypical, and otherwise that
  \[
    \rho \ll 1/A^2
    = \sqrt{U V}/{Q}.
  \]
  In fact, we've seen that $\rho = 0$ unless $U = V =
  Q/C(\omega)$.
  
  We turn to the case $B = q A$.
  For this we have reduced
  to bounding the integral
  \begin{equation}\label{eqn:key-integral-O-1-A-O-1-B}
    \mathbb{E}_{a,b \in \mathcal{O}(1/A) / \mathcal{O}(1/B)}
    \psi_\chi(\phi_1(a,b) - \phi_2(a,b)/2)
  \end{equation}
  for $u,v$ satisfying \eqref{eqn:congruences-for-u-v-via-xi-tau}.
  This integral is a normalized two-variable quadratic
  Gauss sum on $(\mathfrak{o}/\mathfrak{p})^2$,
  and is thus $\O(1/q)$
  unless the quadratic term degenerates.
  To investigate the latter possibility,
  we set $\eta := \tau \xi$
  (thus $|\eta| = 1/V = 1/U = A B/Q = q A^2/Q$)
  and invoke the assumed congruences on $u,v$
  to obtain,
  for $a,b \in \mathcal{O}(1/A)$,
  \[
    \phi_2(a,b)
    \equiv
    \eta ( \eta- 1) a^2 - \eta (\eta + 1) b^2
    + \eta^3 ( a + b)^2 \pmod{\mathcal{O}(q/Q)}.
  \]
  If $U > 1$,
  so that $|\eta| < 1$,
  then we obtain the further simplification
  \[
    \phi_2(a,b)
    \equiv
    -\eta a^2 - \eta b^2 \pmod{\mathcal{O}(q/Q)},
  \]
  and so $\phi_2$ is manifestly nondegenerate.  It remains to
  consider the case $U = 1$, so that $Q = q A^2$ and
  $|\eta| = 1$.  In general, the discriminant of a polynomial
  $(a,b) \mapsto c_1 a^2 + c_2 b^2 + c_3 (a +b)^2$ is given by
  $- 4 ( c_1 c_2 + c_2 c_3 + c_3 c_1)$.  In the case of
  $\phi_2$, this specializes to $-4 \eta^2 (\eta^2 + 1)$.  The
  degeneracy condition is thus that
  $\eta^2 + 1 \equiv 0 \pmod{\mathfrak{p}}$, which is possible
  only if $-1$ is a square modulo $\mathfrak{p}$.  Let us fix
  once such $\eta \in \mathfrak{o}^\times$, thus
  $\eta^2 \equiv -1 \pmod{\mathfrak{p}}$.
  Then $|\xi| = 1$.
  We compute that
  \[
    \phi_2(a,b)
    \equiv
    - (a + \eta b)^2 \pmod{\mathcal{O}(q/Q)}.
  \]
  Thus in all cases the quadratic term
  has rank $\geq 1$,
  and so \eqref{eqn:key-integral-O-1-A-O-1-B}
  is $\O(q^{-1/2})$.
  

  We turn finally to the dyadic case, retaining
  the assumptions \eqref{eqn:assumptions-before-non-dyadic-assumption}.
  We may assume that $Q/U$ is sufficiently large,
  since otherwise the required bound is trivial.
  We may then choose
  $A \in \{q, q^2, q^3, \dotsc  \}$ so that
  $Q/U \ll A^2 \leq |2|_F Q/U$.
  Then \eqref{eqn:rho-via-rho-u-v} holds
  with $\rho(u,v)$ defined
  by \eqref{eqn:rho-u-v-defn},
  with $\phi_1$ defined by
  \eqref{eqn:defn-phi-1},
  and with $\phi_2 := 0$.
  The same analysis as before
  then gives $\rho \ll \sqrt{U V}/Q$.
\end{proof}

\begin{remark}
  With slightly more care it should be possible to improve the
  estimate \eqref{eqn:main-osc-int-est-p-adic} in the case
  $Q = q^{2 \alpha + 1}$; compare with
  \cite[Thm 3.4]{2019arXiv190810346P}.
\end{remark}

\section{Analytic newvectors for $\operatorname{GL}_2$}\label{sec:analyt-newv-oper}
Over a non-archimedean local field, classical newvector theory
\cite{MR0337789}
provides a convenient way
to
(among other things)
construct pre-Kuznetsov weights that localize on a fairly small subset of
generic dual (see
\S\ref{sec:constr-wt-non-archimedean-case},
\S\ref{sec:lower-bounds-for-wts}).
Here we record
an analytic variant of that theory for $\GL_2$
that is valid also over an
archimedean local field.

We note that a further extension to $\GL_n(\mathbb{R})$ of the
analytic newvector theory recorded here has been developed with
S. Jana \cite{JN19a}, while an algebraic newvector theory for
$\GL_n$ over an archimedean local field has been given by Popa
\cite{MR2419183} when $n= 2$ and is being developed by
P. Humphries for general $n$
(see \cite[Rmk 4]{JN19a} for further discussion).

\subsection{Notation}
Throughout this section $F$ denotes a local field,
$\psi$ a nontrivial unitary character of $F$.
We use the
the following notation for elements of $\GL_2(F)$:
\[
  n(x) := \begin{pmatrix}
    1 & x \\
    0 & 1
  \end{pmatrix},
  \quad
  a(y) :=
  \begin{pmatrix}
    y & 0 \\
    0 & 1
  \end{pmatrix},
  \quad
  w := \begin{pmatrix}
    & -1 \\
    1 & 
  \end{pmatrix}.
\]
We identify each generic irreducible representation $\pi$ of
$\GL_2(F)$ with its Whittaker model $\mathcal{W}(\pi,\psi)$,
consisting of $W : \GL_2(F) \rightarrow \mathbb{C}$ satisfying
$W(n(x) g) = \psi(x) W(g)$.  We will often abbreviate
$W(y) := W(a(y))$ for $y \in F^\times$.  We equip $F^\times$
with any Haar measure and write simply
$\int_{y \in F^\times} f(y)$ for the corresponding integral of a
function.  When $\pi$ is unitary, we normalize the inner product
to be given in the Kirillov model by
$\|W\|^2 = \int_{y \in F^\times} |W(y)|^2$.

\subsection{Local $\gamma$-factors and analytic
  conductors}\label{sec:local-gamma-factors}
Let $\pi$ be a generic irreducible
representation of $\GL_2(F)$,
and let $\chi$ be a character of $F^\times = \GL_1(F)$.
Recall from \S\ref{sec:intro-l-factors}
the corresponding local $L$-, $\eps$- and $\gamma$-factors,
which are related by the identity
\[
  \gamma(\psi,\pi \otimes \chi,s)
  =
  \eps(\psi,\pi \otimes \chi,s)
  \frac{L(\widetilde{\pi} \otimes  \chi^{-1}, 1-s)}{L(\pi \otimes \chi, s)}.
\]
For each $W \in \pi$ the Mellin transform
$\int_{y \in F^\times} W(a(y)) \chi(y)$
converges absolutely at least for $\Re(\chi)$ sufficiently large,
and the ratio
\[
  \frac{\int_{y \in F^\times} W(a(y)) \chi(y) }{ L(\pi \otimes \chi,
    1/2)
  }
\]
extends to a holomorphic function on the character group of $F^\times$.
We have the local functional equation
\begin{equation}\label{eq:31}
  \int_{y \in F^\times} W(a(y)) \chi(y)
  =
  \frac{
    \int_{y \in F^\times} W(a(y) w) \omega_\pi^{-1} \chi^{-1} |.|(y)
  }
  {\gamma(\psi,\pi \otimes \chi,1/2)},
\end{equation}
with $\omega_\pi$ the central character and each integral
interpreted via meromorphic continuation in $\chi$.
Recall also the Stirling-type estimate \eqref{eq:stirling-for-general-RS-gamma}.

\subsection{Analytic congruence subsets}\label{sec:analyt-congr-subgr}
For $C, D \in \mathbb{R}_{\geq 1}$ and $\eps \in (0,1]$
set
\[
  U_1(C,\eps)
  :=
  \left\{
    y \in F^\times :
    C |y - 1|, |y^{-1} - 1|  \leq \eps
  \right\}
\]
and
\[
  K_1(C,D,\eps)
  :=
  \left\{g = \left(
    \begin{smallmatrix}
      a&b\\
      c&d
    \end{smallmatrix}
  \right) \in \GL_2(F):
  \begin{gathered}
    |a-1|, |b|, C |c|, \\
    D |d-1|, |\det(g)^{-1} - 1| \leq \eps
  \end{gathered}
  \right\}.
\]
Set also
\[
  K_0(C,\eps)
  := K_1(C,1,\eps).
\]
Note that
$K_1(C',D',\eps')
\subset 
K_1(C,D,\eps)$
whenever
$C' \geq C, D' \geq D$ and $\eps' \leq  \eps$,
and in particular that
$K_1(C,D,\eps)
\subset K_0(C,\eps)$.


For example, suppose $F$ is non-archimedean
and $m, n$ are nonnegative integers.
Then
$U_1(1,1)
=
\mathfrak{o}^\times$
is the unit group
and
$U_1(q^n,1)
=
F^\times \cap (1 + \mathfrak{p}^{n})$
belongs to its standard filtration.
Similarly
$K_0(1,1) = \GL_2(\mathfrak{o})$
is a maximal compact subgroup
with congruence subgroups
$K_0(q^n,1)
=
\GL_2(F) \cap
\left(
  \begin{smallmatrix}
    \mathfrak{o}  & \mathfrak{o}  \\
    \mathfrak{p}^n & \mathfrak{o} 
  \end{smallmatrix}
\right)$.
The subset
$K_1(q^n,q^m,1)
=
\GL_2(F)
\cap
\left(
  \begin{smallmatrix}
    \mathfrak{o}  & \mathfrak{o}  \\
    \mathfrak{p}^n & 1 + \mathfrak{p}^m
  \end{smallmatrix}
\right)$
is a subgroup
if and only if $m \leq n$.

\subsection{}
\label{sec:define-eta-central-character-congruence-subgp}
For each
unitary character $\omega$ of $F^\times$,
define a map of sets
$\eta_{\omega} : K_0(C,\eps) \rightarrow \U(1)$
by the formula
\begin{equation}\label{eq:48}
  \eta_{\omega}
  \left( \begin{pmatrix}
      a & b \\
      c & d
    \end{pmatrix} \right)
  :=
  \begin{cases}
    \omega(d) &
    \text{ if }
    d/a \in U_1(1,1), \\
    1 &
    \text{ otherwise.}
  \end{cases}
\end{equation}
We note that
for any character $\chi$,
the map $g \mapsto \chi(\det g) \eta_{\chi^{-2}}$
defined initially on the set $K_0(C,\eps)$
descends to a well-defined map on the image of that set in $\PGL_2(F)$.

\subsection{}\label{sec:anal-newv-stmt}
The main result of this section is the following
analytic variant of some of the main results
of local newvector theory:

\begin{theorem}~\label{lem:produce-whittaker-against-random-central-character-wannabe}
  \begin{enumerate}[(i)]
  \item For each $\delta > 0$ there exists
    $\eps > 0$
    so that for each
    generic irreducible unitary
    representation $\pi$ of $\GL_2(F)$
    there exists a unit vector $W \in \pi$
    so that
    for each unitary character $\omega$ of $F^\times$,
    \begin{equation}\label{eq:25}
      |W(g) - \eta_\omega(g)| \leq  \delta
      \text{ for all }
      g \in K_1(C(\pi), C(\omega_\pi \omega^{-1}), \eps).
    \end{equation}
  \item Suppose $(F,\psi)$ is unramified.
    For each generic irreducible representation $\pi$ of $\GL_2(F)$
    there exists a nonzero vector $W \in \pi$
    so that
    for each
    unitary character $\omega$ of $F^\times$,
    \eqref{eq:25}
    holds with $\delta = 0$
    and $\eps = 1$.
  \end{enumerate}      
\end{theorem}

Thanks to the identity
$W(g) - \eta_\omega(g)
=
(W(g) - \eta_{\omega_\pi}(g))
+ \eta_\omega(g) (\eta_{\omega_\pi - \omega}(g) - 1)$
and the triangle inequality,
the proof of Theorem
\ref{lem:produce-whittaker-against-random-central-character-wannabe}
reduces to that of Lemmas
\ref{lem:gl1-character-analytic-conductor-invariance}
and \ref{lem:produce-vector-local-whittaker} below:

\begin{lemma}\label{lem:gl1-character-analytic-conductor-invariance}~
  \begin{enumerate}[(i)]
  \item For each $\delta > 0$ there exists $\eps > 0$
    so that
    for each
    unitary character $\omega$ of $F^\times$,
    one has
    \begin{equation}\label{eq:24}
      |\omega(y) - 1| \leq  \delta
      \text{ for all }
      y \in U_1(C(\omega),\eps).
    \end{equation}
  \item Suppose $(F,\psi)$ is unramified.
    For each
    unitary character $\omega$ of $F^\times$,
    one has \eqref{eq:24} with $\delta = 0$ and $\eps = 1$.
  \end{enumerate}
\end{lemma}
\begin{proof}
  Assertion (ii) amounts to the definition
  of the conductor.
  Assertion (i) is readily verified
  by a case-by-case analysis;
  for instance, in the case $F = \mathbb{R}$ it follows from
  the estimate $|y^{i t} - 1| = O(t |y - 1|)$
  for $y \in [1/2,2]$.
\end{proof}

\begin{lemma}~\label{lem:produce-vector-local-whittaker}
  \begin{enumerate}[(i)]
  \item
    For each
    $\delta > 0$
    there exists $\eps > 0$
    so that
    for each generic irreducible unitary
    representation $\pi$ of $\GL_2(F)$
    there exists a unit vector $W \in \pi$
    so that
    \begin{equation}\label{eq:10}
      |W(g)  - \eta_{\omega_\pi}(g)| \leq  \delta
      \text{ for all }
      g \in K_0(C(\pi),\eps).
    \end{equation}
  \item
    Suppose $(F,\psi)$ is unramified.
    For each generic irreducible representation $\pi$
    of $\GL_2(F)$
    there exists a nonzero vector $W \in \pi$
    so that
    \eqref{eq:10}
    holds with $\delta = 0$ and $\eps = 1$.
  \end{enumerate}
\end{lemma}
\begin{proof}
  Assertion (ii) is a well-known consequence of the theory of
  local newvectors \cite{MR0337789}.
  By (ii), we may assume in the proof of (i) that $q = \O(1)$.
  
  As we shall shortly explain
  in more detail, assertion (i)
  is a consequence of the local functional equation
  after choosing $y \mapsto W(y)$ to be a fixed bump function
  on $F^\times$ taking the value $1$ at $y=1$;
  the basic idea is that then $w W(y)$
  is mostly supported on $|y| \ll  C(\pi)$,
  hence
  $W$
  is roughly invariant by
  $w n(x)w $ for $|x|$ a bit smaller than $1/C(\pi)$.
  The method
  of proof employed here
  may be understood as a soft
  analytic variant of the that used to establish
  the theory of local newvectors itself, as in \cite{MR0337789}.
  
  To implement this idea,
  suppose for the sake contradiction that
  assertion (i) fails.
  Then there exists a fixed $\delta > 0$
  and a sequence of tuples $(\eps,\pi)$ as above
  with $\eps \rightarrow 0$
  so that there does not exist a unit vector
  $W \in \pi$
  satisfying \eqref{eq:10}.
  Here and in what
  follows asymptotic notation refers to the $\eps
  \rightarrow 0$
  limit, so that for instance
  $o(1) := o_{\eps \rightarrow 0}(1)$,
  and ``fixed'' means ``independent of $\eps$.''
  We aim to produce a contradiction by showing
  that such a vector indeed exists.


  Abbreviate $C := C(\pi)$.
  Let $c$ be an element of $F^\times$ with
  $|c| = C$.
  Introduce the shorthand
  $z_1 \simeq z_2$
  for $z_1 = z_2 + o(1)$.
  Each $g \in K_0(C,\eps)$
  can be expressed as
  \begin{equation}\label{eq:produce-vector:42}
    g =
    z(u)
    n(x_1) a(y_1) 
    w
    n(x_2/c)
    w
    \text{ with }
    x_1, x_2 \simeq 0
    \text{ and }
    u,y_1 \simeq 1.
  \end{equation}
  Choose a bump function $f \in C_c^\infty(F^\times)$
  with $f(1) = 1$ and $\int_{F^\times} |f|^2 = 1$.
  Then the Mellin transform
  $F_4(\chi)
  :=
  \int_{y \in F^\times}
  f(y) \chi^{-1}(y)$
  satisfies
  \begin{equation}\label{eq:8}
    F_4(\chi) \ll C(\chi)^{-4 A}
    \text{ for each fixed } A
  \end{equation}
  and for all $\chi$ with $\Re(\chi)$ in a fixed compact set.
  Using the theory of the Kirillov model,
  we may choose $W \in \pi$ so that $W(y) = f(y)$;
  then
  $W(1) = 1$ and $\|W\| = 1$.
  For $g$ as in (\ref{eq:produce-vector:42})
  the formula $W(g)
  =
  \omega_\pi(u)
  \psi(x_1) w n(x_2/c) w W(y_1)$
  and the estimates
  $\psi(x_1) \simeq 1$,
  $W(y_1) =f(y_1) \simeq f(1) = W(1)$
  and
  $\omega_\pi(u) \simeq \eta_{\omega_\pi}(g)$
  reduce our task to showing that
  \[
    w n(x/c) w W(y)
    \simeq W(y) \text{ for all }
    y \simeq 1,
    x \simeq 0.
  \]
  Let $\tau \in (0,1/4)$ be fixed.
  By Mellin inversion, the local functional equation,
  and the identity
  $n(x/c) w W(t)
  = \psi(x t/c) w W(t)$,
  we have
  \[
    w n(x/C) w W(y)
    -
    W(y) 
    =
    \int_{\chi : \Re(\chi) = \tau}
    \chi^{-1}(y)
    F_1(\chi)
    F_2(\chi)
    \,
    d \chi
  \]
  with
  \[
    F_1(\chi) :=
    \chi^{-1}(c) /\gamma(\psi,\pi \otimes \chi,1/2),
  \]
  \[
    F_2(\chi)
  :=
  \int_{t \in F^\times}
  F_3(t) \chi(t),
  \]
  \[
    F_3(t) := (\psi(t x) - 1) W_{w \varphi}(c t).
  \]
  (Here we equip the group
  of unitary characters $\chi$ of $F^\times$,
  hence
  also
  its cosets consisting of characters of given real part,
  with the measure dual to the chosen Haar on $F^\times$.)
  By the Stirling-type estimate \eqref{eq:stirling-for-general-RS-gamma}
  for local $\gamma$-factors,
  we have
  \begin{equation}\label{eq:produce-vector:45}
    F_1(\chi) \ll C(\chi)^{2 \Re(\chi)}
    \text{ for $\tau \leq \Re(\chi) \ll 1$}.
  \end{equation}
  Our task thereby reduces to showing
  that
  (for instance)
  the estimate
  $F_2(\chi) \ll |x| C(\chi)^{-2}$
  holds
  for each character
  $\chi$ of $F^\times$
  of real part $\tau$.
  By partial integration
  we reduce to showing that
  \begin{equation}\label{eq:13}
    \int_{t \in F^\times} |\Theta F_3(t)| \ll |x|
  \end{equation}
  when $\Theta := \Delta^d$
  with $d \geq 0$ fixed
  and $\Delta = \Delta_{\GL_1(F)}$
  as defined in \S\ref{sec:comp-betw-lapl}.
  We claim that
  \begin{equation}\label{eq:produce-vector:43}
    \Theta w W(c t)
    \ll (1 + |t|)^{-A}
  \end{equation}
  for each fixed $\Theta$ and each fixed $A \geq 10$.
  From (\ref{eq:produce-vector:43}),
  the product rule,
  and the easy estimate
  $\Theta[t \mapsto  (\psi(t x) - 1)](t) \ll
  |x t|
  (1 + |x t|)^{O(1)}$
  it then follows that
  \[
    \int_{t \in F^\times}
    |\Theta F_3(t)|
    \ll
    |x|
    \int_{t \in F^\times}
    \frac{|t|(1+|x t|)^{O(1)}}{
      (1 + |t|)^A
    }
    \ll |x|,
  \]
  giving the desired estimate (\ref{eq:13}).
  We turn now to  the remaining task
  of establishing the claim
  (\ref{eq:produce-vector:43}).
  We appeal once again to Mellin inversion and the local
  functional
  equation,
  giving for $\sigma \geq 0$ that
  \[
    \Theta w W(c t)
    =
    \int_{\chi : \Re(\chi) = \sigma}
    \chi^{-1}(t)
    \Theta^\wedge(-\chi)
    F_1(\chi) 
    F_4(\chi) \,
    d \chi 
  \]
  with the complex scalar
  $\Theta^\wedge(-\chi) \ll C(\chi)^{O(1)}$ defined via
  the relation
  $\Theta^\wedge(-\chi)(\chi^{-1})
  = 
  \Theta(\chi^{-1})$
  and
  with $F_1(\chi), F_4(\chi)$ as defined above.
  By (\ref{eq:8}) and (\ref{eq:produce-vector:45}),
  we deduce that $\Theta W_{w \varphi}(c t) \ll |t|^{-\sigma}$
  for each
  $\sigma \geq 0$ belonging to a fixed compact set;
  taking $\sigma = 0$ and $\sigma = A$,
  we finally obtain (\ref{eq:produce-vector:43}).
\end{proof}

\subsection*{Acknowledgements}
We thank Giacomo Cherubini, Subhajit Jana, Emmanuel Kowalski,
Philippe Michel, Ian Petrow, Zhi Qi, K. Soundararajan, Akshay
Venkatesh, Han Wu and Matthew Young for helpful discussions and
feedback.

\bibliography{refs}{}

\def\cprime{$'$} \def\cprime{$'$} \def\cprime{$'$} \def\cprime{$'$}
\begin{thebibliography}{10}

\bibitem{2019arXiv191101800B}
Antal {Balog}, Andr{\'a}s {Bir{\'o}}, Giacomo {Cherubini}, and Niko
  {Laaksonen}.
\newblock {Bykovskii-type theorem for the Picard manifold}.
\newblock {\em arXiv e-prints}, page arXiv:1911.01800, Nov 2019.

\bibitem{MR1999922}
Ehud~Moshe Baruch.
\newblock A proof of {K}irillov's conjecture.
\newblock {\em Ann. of Math. (2)}, 158(1):207--252, 2003.

\bibitem{MR2322488}
Ehud~Moshe Baruch and Zhengyu Mao.
\newblock Central value of automorphic {$L$}-functions.
\newblock {\em Geom. Funct. Anal.}, 17(2):333--384, 2007.

\bibitem{MR3219530}
Joseph Bernstein and Bernhard Kr\"{o}tz.
\newblock Smooth {F}r\'{e}chet globalizations of {H}arish-{C}handra modules.
\newblock {\em Israel J. Math.}, 199(1):45--111, 2014.

\bibitem{MR748505}
Joseph~N. Bernstein.
\newblock {$P$}-invariant distributions on {${\rm GL}(N)$} and the
  classification of unitary representations of {${\rm GL}(N)$}
  (non-{A}rchimedean case).
\newblock In {\em Lie group representations, {II} ({C}ollege {P}ark, {M}d.,
  1982/1983)}, volume 1041 of {\em Lecture Notes in Math.}, pages 50--102.
  Springer, Berlin, 1984.

\bibitem{MR2811610}
Valentin Blomer and Farrell Brumley.
\newblock On the {R}amanujan conjecture over number fields.
\newblock {\em Ann. of Math. (2)}, 174(1):581--605, 2011.

\bibitem{2009arXiv0904.2429B}
Valentin Blomer and Gergely Harcos.
\newblock Twisted {$L$}-functions over number fields and {H}ilbert's eleventh
  problem.
\newblock {\em Geom. Funct. Anal.}, 20(1):1--52, 2010.

\bibitem{2019arXiv190207042B}
Valentin {Blomer}, Peter {Humphries}, Rizwanur {Khan}, and Micah {Milinovich}.
\newblock {Motohashi's fourth moment identity for non-archimedean test
  functions and applications}.
\newblock {\em arXiv e-prints}, page arXiv:1902.07042, Feb 2019.

\bibitem{MR2124019}
Roelof~W. Bruggeman and Yoichi Motohashi.
\newblock A new approach to the spectral theory of the fourth moment of the
  {R}iemann zeta-function.
\newblock {\em J. Reine Angew. Math.}, 579:75--114, 2005.

\bibitem{MR1431508}
Daniel Bump.
\newblock {\em Automorphic Forms and Representations}, volume~55 of {\em
  Cambridge Studies in Advanced Mathematics}.
\newblock Cambridge University Press, Cambridge, 1997.

\bibitem{MR1462836}
C.~J. Bushnell and G.~Henniart.
\newblock An upper bound on conductors for pairs.
\newblock {\em J. Number Theory}, 65(2):183--196, 1997.

\bibitem{MR1606410}
Colin~J. Bushnell, Guy~M. Henniart, and Philip~C. Kutzko.
\newblock Local {R}ankin-{S}elberg convolutions for {${\rm GL}_n$}: explicit
  conductor formula.
\newblock {\em J. Amer. Math. Soc.}, 11(3):703--730, 1998.

\bibitem{MR1013462}
W.~Casselman.
\newblock Canonical extensions of {H}arish-{C}handra modules to representations
  of {$G$}.
\newblock {\em Canad. J. Math.}, 41(3):385--438, 1989.

\bibitem{MR0337789}
William Casselman.
\newblock On some results of {A}tkin and {L}ehner.
\newblock {\em Math. Ann.}, 201:301--314, 1973.

\bibitem{MR4001088}
Jingsong Chai and Zhi Qi.
\newblock On the {W}aldspurger formula and the metaplectic {R}amanujan
  conjecture over number fields.
\newblock {\em J. Funct. Anal.}, 277(10):3757--3782, 2019.

\bibitem{MR2508768}
J.~W. Cogdell.
\newblock Notes on {$L$}-functions for {${\rm GL}_n$}.
\newblock In {\em School on {A}utomorphic {F}orms on {${\rm GL}(n)$}},
  volume~21 of {\em ICTP Lect. Notes}, pages 75--158. Abdus Salam Int. Cent.
  Theoret. Phys., Trieste, 2008.

\bibitem{MR3753910}
J.~W. Cogdell and I.~I. Piatetski-Shapiro.
\newblock Derivatives and {L}-functions for {$GL_n$}.
\newblock In {\em Representation theory, number theory, and invariant theory},
  volume 323 of {\em Progr. Math.}, pages 115--173. Birkh\"{a}user/Springer,
  Cham, 2017.

\bibitem{MR1779567}
J.~B. Conrey and H.~Iwaniec.
\newblock The cubic moment of central values of automorphic {$L$}-functions.
\newblock {\em Ann. of Math. (2)}, 151(3):1175--1216, 2000.

\bibitem{MR0379375}
Stephen Gelbart.
\newblock {\em Automorphic Forms on Ad\`ele Groups}.
\newblock Princeton University Press, Princeton, N.J., 1975.
\newblock Annals of Mathematics Studies, No. 83.

\bibitem{MR546600}
Stephen Gelbart and Herv{\'e} Jacquet.
\newblock Forms of {${\rm GL}(2)$} from the analytic point of view.
\newblock In {\em Automorphic Forms, Representations and {$L$}-functions
  ({P}roc. {S}ympos. {P}ure {M}ath., {O}regon {S}tate {U}niv., {C}orvallis,
  {O}re., 1977), {P}art 1}, Proc. Sympos. Pure Math., XXXIII, pages 213--251.
  Amer. Math. Soc., Providence, R.I., 1979.

\bibitem{MR0191899}
Roger Godement.
\newblock Domaines fondamentaux des groupes arithm\'etiques.
\newblock In {\em S\'eminaire {B}ourbaki, 1962/63. {F}asc. 3, {N}o. 257},
  page~25. Secr\'etariat math\'ematique, Paris, 1964.

\bibitem{MR3889963}
Roger Godement.
\newblock {\em Notes on {J}acquet-{L}anglands' theory}, volume~8 of {\em CTM.
  Classical Topics in Mathematics}.
\newblock Higher Education Press, Beijing, 2018.
\newblock With commentaries by Robert Langlands and Herve Jacquet.

\bibitem{MR1879668}
Aleksandar Ivi\'{c}.
\newblock On sums of {H}ecke series in short intervals.
\newblock {\em J. Th\'{e}or. Nombres Bordeaux}, 13(2):453--468, 2001.

\bibitem{MR701565}
H.~Jacquet, I.~I. Piatetskii-Shapiro, and J.~A. Shalika.
\newblock Rankin-{S}elberg convolutions.
\newblock {\em Amer. J. Math.}, 105(2):367--464, 1983.

\bibitem{MR871663}
H.~Jacquet and D.~Zagier.
\newblock Eisenstein series and the {S}elberg trace formula. {II}.
\newblock {\em Trans. Amer. Math. Soc.}, 300(1):1--48, 1987.

\bibitem{Ja72}
Herv{\'e} Jacquet.
\newblock {\em Automorphic forms on {${\rm GL}(2)$}. {P}art {II}}.
\newblock Lecture Notes in Mathematics, Vol. 278. Springer-Verlag, Berlin,
  1972.

\bibitem{MR2533003}
Herv\'{e} Jacquet.
\newblock Archimedean {R}ankin-{S}elberg integrals.
\newblock In {\em Automorphic forms and {$L$}-functions {II}. {L}ocal aspects},
  volume 489 of {\em Contemp. Math.}, pages 57--172. Amer. Math. Soc.,
  Providence, RI, 2009.

\bibitem{MR0401654}
Herv{\'e} Jacquet and R.~P. Langlands.
\newblock {\em Automorphic forms on {${\rm GL}(2)$}}.
\newblock Lecture Notes in Mathematics, Vol. 114. Springer-Verlag, Berlin,
  1970.

\bibitem{JN19a}
Subhajit {Jana} and Paul~D. {Nelson}.
\newblock {Analytic newvectors for $\mathrm{GL}_n(\mathbb{R})$}.
\newblock {\em arXiv e-prints}, page arXiv:1911.01880, Nov 2019.

\bibitem{MR1404335}
Kamal Khuri-Makdisi.
\newblock On the {F}ourier coefficients of nonholomorphic {H}ilbert modular
  forms of half-integral weight.
\newblock {\em Duke Math. J.}, 84(2):399--452, 1996.

\bibitem{MR783554}
Winfried Kohnen.
\newblock Fourier coefficients of modular forms of half-integral weight.
\newblock {\em Math. Ann.}, 271(2):237--268, 1985.

\bibitem{MR3112415}
Hisashi Kojima.
\newblock On the {F}ourier coefficients of {H}ilbert modular forms of
  half-integral weight over arbitrary algebraic number fields.
\newblock {\em Tsukuba J. Math.}, 37(1):1--11, 2013.

\bibitem{MR507800}
Bertram Kostant.
\newblock On {W}hittaker vectors and representation theory.
\newblock {\em Invent. Math.}, 48(2):101--184, 1978.

\bibitem{MR3649366}
Erez Lapid and Zhengyu Mao.
\newblock On an analogue of the {I}chino-{I}keda conjecture for {W}hittaker
  coefficients on the metaplectic group.
\newblock {\em Algebra Number Theory}, 11(3):713--765, 2017.

\bibitem{MR2046512}
Erez Lapid, Goran Mui\'{c}, and Marko Tadi\'{c}.
\newblock On the generic unitary dual of quasisplit classical groups.
\newblock {\em Int. Math. Res. Not.}, (26):1335--1354, 2004.

\bibitem{MR1810211}
Hung~Yean Loke.
\newblock Trilinear forms of {$\mathfrak{gl}_2$}.
\newblock {\em Pacific J. Math.}, 197(1):119--144, 2001.

\bibitem{MR3594414}
P.~Maga.
\newblock Subconvexity for twisted {$L$}-functions over number fields via
  shifted convolution sums.
\newblock {\em Acta Math. Hungar.}, 151(1):232--257, 2017.

\bibitem{MichelVenkateshICM}
Philippe Michel and Akshay Venkatesh.
\newblock Equidistribution, {$L$}-functions and ergodic theory: on some
  problems of {Y}u.\ {L}innik.
\newblock In {\em International {C}ongress of {M}athematicians. {V}ol. {II}},
  pages 421--457. Eur. Math. Soc., Z\"urich, 2006.

\bibitem{michel-2009}
Philippe Michel and Akshay Venkatesh.
\newblock The subconvexity problem for {${\rm GL}_2$}.
\newblock {\em Publ. Math. Inst. Hautes \'Etudes Sci.}, (111):171--271, 2010.

\bibitem{MR1226527}
Yoichi Motohashi.
\newblock An explicit formula for the fourth power mean of the {R}iemann
  zeta-function.
\newblock {\em Acta Math.}, 170(2):181--220, 1993.

\bibitem{MR2279942}
Yoichi Motohashi.
\newblock Mean values of zeta-functions via representation theory.
\newblock In {\em Multiple {D}irichlet series, automorphic forms, and analytic
  number theory}, volume~75 of {\em Proc. Sympos. Pure Math.}, pages 257--279.
  Amer. Math. Soc., Providence, RI, 2006.

\bibitem{nelson-venkatesh-1}
P.~D. {Nelson} and A.~{Venkatesh}.
\newblock {The orbit method and analysis of automorphic forms}.
\newblock {\em ArXiv e-prints}, May 2018.

\bibitem{nelson-variance-II}
Paul~{D}. Nelson.
\newblock Quantum variance on quaternion algebras, {II}.
\newblock preprint, 2017.

\bibitem{nelson-padic-que}
Paul~D. Nelson.
\newblock Microlocal lifts and quantum unique ergodicity on {${\rm GL}\sb
  2(\Bbb{Q}_p)$}.
\newblock {\em Algebra Number Theory}, 12(9):2033--2064, 2018.

\bibitem{2018arXiv181102452P}
Ian {Petrow} and Matthew~P. {Young}.
\newblock {The Weyl bound for Dirichlet $L$-functions of cube-free conductor}.
\newblock {\em arXiv e-prints}, page arXiv:1811.02452, Nov 2018.

\bibitem{MR3968874}
Ian Petrow and Matthew~P. Young.
\newblock A generalized cubic moment and the {P}etersson formula for newforms.
\newblock {\em Math. Ann.}, 373(1-2):287--353, 2019.

\bibitem{2019arXiv190810346P}
Ian {Petrow} and Matthew~P. {Young}.
\newblock {The fourth moment of Dirichlet $L$-functions along a coset and the
  Weyl bound}.
\newblock {\em arXiv e-prints}, page arXiv:1908.10346, Aug 2019.

\bibitem{MR3394377}
Ian~N. Petrow.
\newblock A twisted {M}otohashi formula and {W}eyl-subconvexity for
  {$L$}-functions of weight two cusp forms.
\newblock {\em Math. Ann.}, 363(1-2):175--216, 2015.

\bibitem{MR2419183}
Alexandru~A. Popa.
\newblock Whittaker newforms for {A}rchimedean representations.
\newblock {\em J. Number Theory}, 128(6):1637--1645, 2008.

\bibitem{MR1198305}
Dipendra Prasad.
\newblock Invariant forms for representations of {${\rm GL}_2$} over a local
  field.
\newblock {\em Amer. J. Math.}, 114(6):1317--1363, 1992.

\bibitem{MR3885172}
Yannan Qiu.
\newblock The {W}hittaker period formula on metaplectic {$\rm SL_2$}.
\newblock {\em Trans. Amer. Math. Soc.}, 371(2):1083--1117, 2019.

\bibitem{MR2373356}
Andre Reznikov.
\newblock Rankin-{S}elberg without unfolding and bounds for spherical {F}ourier
  coefficients of {M}aass forms.
\newblock {\em J. Amer. Math. Soc.}, 21(2):439--477, 2008.

\bibitem{MR723012}
P.~Sarnak.
\newblock The arithmetic and geometry of some hyperbolic three-manifolds.
\newblock {\em Acta Math.}, 151(3-4):253--295, 1983.

\bibitem{MR780071}
Peter Sarnak.
\newblock Fourth moments of {G}r\"ossencharakteren zeta functions.
\newblock {\em Comm. Pure Appl. Math.}, 38(2):167--178, 1985.

\bibitem{Sch02}
Ralf Schmidt.
\newblock Some remarks on local newforms for {$\rm GL(2)$}.
\newblock {\em J. Ramanujan Math. Soc.}, 17(2):115--147, 2002.

\bibitem{MR1233447}
Goro Shimura.
\newblock On the {F}ourier coefficients of {H}ilbert modular forms of
  half-integral weight.
\newblock {\em Duke Math. J.}, 71(2):501--557, 1993.

\bibitem{xoMR0476703}
Jerrold~B. Tunnell.
\newblock On the local {L}anglands conjecture for {$GL(2)$}.
\newblock {\em Invent. Math.}, 46(2):179--200, 1978.

\bibitem{GKdim}
David~A. Vogan, Jr.
\newblock Gel\cprime fand-{K}irillov dimension for {H}arish-{C}handra modules.
\newblock {\em Invent. Math.}, 48(1):75--98, 1978.

\bibitem{MR646366}
J.-L. Waldspurger.
\newblock Sur les coefficients de {F}ourier des formes modulaires de poids
  demi-entier.
\newblock {\em J. Math. Pures Appl. (9)}, 60(4):375--484, 1981.

\bibitem{MR1170566}
Nolan~R. Wallach.
\newblock {\em Real reductive groups. {II}}, volume 132 of {\em Pure and
  Applied Mathematics}.
\newblock Academic Press, Inc., Boston, MA, 1992.

\bibitem{MR3213837}
Han Wu.
\newblock Burgess-like subconvex bounds for {$\text{GL}_2\times\text{GL}_1$}.
\newblock {\em Geom. Funct. Anal.}, 24(3):968--1036, 2014.

\bibitem{MR3977317}
Han Wu.
\newblock Burgess-like subconvexity for {${\rm GL}_1$}.
\newblock {\em Compos. Math.}, 155(8):1457--1499, 2019.

\bibitem{MR3635360}
Matthew~P. Young.
\newblock Weyl-type hybrid subconvexity bounds for twisted {$L$}-functions and
  {H}eegner points on shrinking sets.
\newblock {\em J. Eur. Math. Soc. (JEMS)}, 19(5):1545--1576, 2017.

\bibitem{MR656029}
Don Zagier.
\newblock The {R}ankin-{S}elberg method for automorphic functions which are not
  of rapid decay.
\newblock {\em J. Fac. Sci. Univ. Tokyo Sect. IA Math.}, 28(3):415--437 (1982),
  1981.

\bibitem{zagier-mellin}
Don Zagier.
\newblock The mellin transform and other useful analytic techniques.
\newblock
  \url{http://people.mpim-bonn.mpg.de/zagier/files/tex/MellinTransform/fulltext.pdf},
  2006.

\end{thebibliography}
\bibliographystyle{plain}
\end{document}